\documentclass[onecolumn,draft]{IEEEtran}

\usepackage{graphicx}
\usepackage{amsmath}
\usepackage{amssymb}
\usepackage{cite}

\newcommand{\E}{\mathsf{E}}

\newcommand{\NN}{{\ensuremath{\mathbb N}}}
\newcommand{\RR}{{\ensuremath{\mathbb R}}}
\newcommand{\ZZ}{{\ensuremath{\mathbb Z}}}
\newcommand{\LL}{{\ensuremath{\mathrm L}}}
\newcommand{\HH}{{\ensuremath{\mathrm H}}}
\newcommand{\kk}{{\ensuremath{\mathbf k}}}
\newcommand{\mm}{{\ensuremath{\mathbf m}}}
\newcommand{\xx}{{\ensuremath{\mathbf x}}}
\newcommand{\ttau}{{\ensuremath{\boldsymbol{\tau}}}}
\newcommand{\eell}{{\ensuremath{\boldsymbol{\ell}}}}
\newcommand{\NM}[1][M]{\ensuremath{{\mathbb N}_{#1}}}
\newcommand{\NMSQ}[1][M]{\ensuremath{{\mathbb N}^{2}_{#1}}}
\newcommand{\NMs}[1][M]{\ensuremath{{\mathbb N}^{\star}_{#1}}}
\newcommand{\NMsSQ}[1][M]{\ensuremath{{\mathbb N}^{\star 2}_{#1}}}

\newtheorem{prop}{Proposition}

\begin{document}

\title{Noise Covariance Properties in Dual-Tree Wavelet Decompositions}

\author{Caroline Chaux, {\em Member, IEEE}, Jean-Christophe Pesquet, {\em Senior Member, IEEE} and\\
 Laurent Duval, {\em Member, IEEE}
\thanks{C. Chaux and J.-C. Pesquet are with the Institut Gaspard Monge and CNRS-UMR 8049,
Universit\'e de Paris-Est Marne-la-Vall\'ee, 77454 Marne-la-Vall\'ee Cedex 2, France.
E-mail: \texttt{\{chaux,pesquet\}@univ-mlv.fr}.}
\thanks{L. Duval is with the Institut fran\c{c}ais du p\'{e}trole, IFP, Technology, Computer Science and Applied Mathematics Division, 1 et 4, avenue de Bois-Pr{\'e}au  F-92852 Rueil-Malmaison, France.
E-mail: \texttt{laurent.duval@ifp.fr}.}}

\maketitle

\begin{abstract}
Dual-tree wavelet decompositions have recently gained much popularity, mainly due to their ability to provide an accurate directional 
analysis of images combined with a reduced redundancy. When the decomposition of a random process is performed -- which occurs in particular when  an additive noise is
corrupting the signal to be analyzed -- it is useful to characterize the statistical properties of the dual-tree wavelet coefficients of this process. As dual-tree decompositions constitute overcomplete frame expansions, correlation structures are introduced among the coefficients, even when a white noise is analyzed. In this paper, we show that it is possible to provide an accurate description of the covariance properties of the dual-tree coefficients of a wide-sense stationary process. The expressions of the (cross-)covariance
sequences of the coefficients are derived in the one and
two-dimensional cases.
Asymptotic results are also provided, allowing to predict the behaviour of the second-order moments for large lag values or at coarse resolution. In addition, the cross-correlations between the primal and dual wavelets, which play a primary role in our theoretical analysis, are calculated for a number of classical wavelet families. Simulation results are finally provided to validate these results.
 \end{abstract}

\begin{keywords}
Dual-tree, wavelets, frames, Hilbert transform, filter banks, cross-correlation, covariance, random processes, stationarity, noise, dependence,  statistics.
\end{keywords}

\section{Introduction}
% Introduction on DWT and openijng to frames
The discrete wavelet transform (DWT) \cite{Mallat_S_1998_book_wav_tsp} is a powerful tool in signal processing, since it provides ``efficient''  basis representations of regular signals \cite{Cambanis_S_1994_tit_wav_adrscpr}.
%While the DWT decomposition  -- by essence --  perfectly recovers signals  in the absence of  processing,
It nevertheless suffers from a few limitations such as aliasing effects in the transform domain,
%(due to  iterated subsampling operators or frequency domain overlapping filters),
coefficient oscillations around singularities and a lack of shift invariance. Frames (see \cite{Daubechies_I_1990_tit_wttflsa,Gribonval_R_2003_tit_spa_rub} or \cite{Casazza_P_2000_tjm_art_ft} for a tutorial), reckoned as  more general signal representations, represent an outlet for these inherent constraints laid on basis functions.

% Some types of redundant transforms generally used with simple thresholding
Redundant DWTs (RDWTs) are shift-invariant non-subsampled frame extensions to the DWT.  They have proved more error or quantization resilient \cite{Coifman_R_1995_was_tra_id,Pesquet_J_1996_tsp_tim_iowr,Bolcskei_H_1998_tsp_fra_taofb}, at the price of  an increased computational cost, especially in higher dimensions.
They do not however take on the lack of rotation invariance or poor directionality of classical separable schemes. These features are particularly sensitive to image and video processing. Recently, several other types of frames have been proposed to incorporate more geometric features, aiming at sparser representations and improved robustness. Early examples of such frames are shiftable multiscale transforms  or steerable pyramids \cite{Simoncelli_E_1996_tip_ste_wfloa}.
%\cite{Simoncelli_E_1992_tit_shi_mst,Simoncelli_E_1995_icip_ste_pfamdc,Simoncelli_E_1996_tip_ste_wfloa}.
To name a few others, there also exist contourlets
\cite{Do_M_2005_tip_con_tedmir},
%(with limited  \cite{Do_M_2005_tip_con_tedmir} or full redundancy \cite{Cunha_A_xxxx_tip_non_cttda}),
bandelets \cite{LePennec_E_2005_tip_spa_girb}, curvelets \cite{Candes_E_2006_siam-mms_fas_dct}, phaselets \cite{Gopinath_R_2003_tsp_pha_tirnsiwt}, directionlets \cite{Velisavljevic_V_2006_tip_dir_amrsf} or other representations involving multiple dictionaries \cite{FiguerasIVentura_R_2006_tip_low_rficrr}.
%All those  approaches have demonstrated great success for denoising, deblurring, texture analysis and synthesis or even compression.

% The TURN
Two  important facets need to be addressed, when resorting to the inherent frame redundancy:
\begin{enumerate}
\item  \emph{multiplicity}:  frame decompositions or reconstructions are not unique in general,
\item \emph{correlation}: transformed  coefficients (and especially those related to noise) are usually correlated, in contrast with the classical uncorrelatedness property of the components of a white noise after an orthogonal transform.
\end{enumerate}

% Early work of statistical properties, and
If the \emph{multiplicity} aspect  is usually recognized (and  often addressed via  averaging techniques\cite{Coifman_R_1995_was_tra_id}), the \emph{correlation} of the transformed coefficients have not received much consideration until recently. Most of the efforts have been
devoted to the analysis of random processes by the DWT
%Certainly, several authors have  addressed the property inheritance of certain features random processes for discrete wavelet bases
\cite{Pastor_D_1995_ts_dec_pssopso2colfpo,Vanucci_M_19978_som_fcswctmb,Leporini_D_1999_tit_hig_owpcfa,Averkamp_R_tit_2000_not_dwtsop}.
%However exact energetic expressions given by the Parseval's theorem for orthonormal bases often  turn into distortion bounds  for frames, rendering noise characterics extraction more complex after a redundant transform.
It should be noted that early works by C.~Houdr{\'e} \emph{et al.} \cite{Cambanis_S_1995_tit_con_wtsorp,Averkamp_R_tit_1998_som_dpcwtrp} consider the continuous wavelet transform of random processes, but only in a recent work by J.~Fowler  exact energetic relationships for an additive noise in the case of the non-tight RDWT have been provided
%, and acknowlegdes the practical usefullness of these results for image compression
\cite{Fowler_J_2005_spl_red_dwtan}.
It must be pointed out that the difficulty to characterize noise properties after a frame decomposition
may limit the design of sophisticated estimation methods in denoising applications.
%relatively simple  thresholding rules are generally applied when processing signals in frames.

Fortunately, there exist redundant signal representations allowing
finer noise behaviour assessment: in particular the dual-tree
wavelet transform, based on the Hilbert transform, whose
advantages in wavelet analysis have been recognized by several
authors
\cite{Abry_P_1994_stfts_mul_td,Olhede_S_2004_biomtrka_ana_wt}. It
consists of two classical wavelet trees developed in parallel. The
second decomposition is refered to as the ``dual'' of the first
one, which is sometimes called the ``primal'' decomposition. The
corresponding analyzing wavelets form Hilbert pairs
\cite[p.198 sq]{Abry_P_1997_book_ond_t}. The dual-tree wavelet
transform was initially proposed by N. Kingsbury
\cite{Kingsbury_N_1998_dspw_dua_tcwtntsidf} and further
investigated by I. Selesnick \cite{Selesnick_I_2001_spl_hil_tpwb}
in the dyadic case. An excellent overview of the topic by I.
Selesnick, R. Baraniuk and N. Kingsbury is provided in
\cite{Selesnick_I_2005_spm_dua_tcwt} and an example of application
is provided in \cite{Jalobeanu_A_2003_ijcv_sat_idcwp}. We recently
have generalized this frame decomposition to the $M$-band case ($M
\geq 2$) (see
\cite{Chaux_C_2004_eusipco_hil_pmbowb,Chaux_C_2005_icassp_2D_dtmbwd,Chaux_C_2006_tip_ima_adtmbwt}).
In the later works, we revamped  the construction of the dual
basis and the pre-processing stage, necessary in the case of digital signal analysis \cite{Rioul_O_1992_tit_fas_adcwt,Abry_P_1994_spl_ini_dwta} and mandatory to accurate directional analysis of images, and we proposed an
optimized reconstruction, thus addressing the first important
facet of the resulting frame \emph{multiplicity}. The $M$-band
($M>2$) dual-tree wavelets prove more selective  in the  frequency
domain than their dyadic counterparts, with improved directional
selectivity as well. Furthermore, a larger choice of filters
satisfying symmetry and orthogonality properties is available.

In this paper, we focus on the second  facet, \emph{correlation}, by studying the second-order statistical properties, in the transform domain, of a stationary random process undergoing   a dual-tree $M$-band  wavelet decomposition. In practice, such a random process typically models an additive noise. Preliminary comments on dual-tree coefficient correlation may be found in \cite{Wang_B_2004_icip_inv_3ddtwtvc}. Dependencies between the coefficients already have been exploited for dual-tree wavelet denoising in \cite{Sendur_L_2002_tsp_biv_sfwbdeid,Rabbani_H_2006_icip_ima_dmldlpcwd}. A parametric nonlinear estimator based on Stein's principle, making explicit use of the correlation properties derived here, is proposed in \cite{Chaux_C_2007_tsp_non_sbemid}.
At first, we briefly recall some properties of the dual-tree wavelet decomposition in Section \ref{sec:DTT}, refering to \cite{Chaux_C_2006_tip_ima_adtmbwt} for more detail. In Section \ref{se:2ndexp}, we express in a general form the second-order moments of the noise coefficients in each tree, both in the one and two-dimensional cases. We also discuss the role of the
post-transform --- often performed on the dual-tree wavelet coefficients --- with respect to (w.r.t.) decorrelation. In Section \ref{se:asympt}, we provide   upper bounds  for the decay of the correlations existing between pairs of primal/dual coefficients as well as an asymptotic result concerning coefficient whitening. The cross-correlations between primal and dual wavelets playing a key role in our analysis, their expressions are derived for several wavelet families in Section \ref{se:examplewav}. Simulation results are provided in Section \ref{sec:applinum} in order to validate our theoretical results and better evaluate the importance of the correlations introduced by the dual-tree decomposition. Some final remarks are drawn in Section \ref{se:conclu}.

% we strengthen the aforesaid literal expressions by numerical experiments in Section VI and draw some conclusions in the importance of addressing the coefficients correlation in algorithm design for frames.

Throughout the paper, the following notations will be used:
$\ZZ$, $\ZZ^*$, $\NN$, $\NN^*$, $\RR$, $\RR^*$, $\RR_+$ and $\RR_+^*$ are the set of integers, nonzero integers, nonnegative integers, positive integers, reals, nonzero reals, nonnegative reals and
positive reals, respectively. Let
$M$ be an integer greater than or equal to 2, $\NM =
\{0,\ldots,M-1\}$ and $\NMs = \{1,\ldots,M-1\}$.

\section{$M$-band dual-tree wavelet analysis}
\label{sec:DTT}

In this section, we recall the basic principles of an $M$-band \cite{Steffen_P_1993_tsp_the_rmbwb}
dual-tree decomposition. Here, we will focus on 1D real signals
belonging to the space $\LL^2(\RR)$ of square integrable
functions. Let $M$ be an integer greater than or equal to 2. An
$M$-band multiresolution analysis of $\LL^2(\RR)$ is defined using
one scaling function (or father wavelet) $\psi_0 \in \LL^2(\RR)$
and $(M-1)$ mother wavelets $\psi_m \in  \LL^2(\RR)$, $m\in \NMs$.
%cite{Steffen_P1993tsp_the_rmbwb}.
In the frequency domain, the so-called
scaling equations are expressed as:
\begin{equation}
\forall m \in \NM,\qquad \sqrt{M}\widehat{\psi}_m(M\omega) =
H_m(\omega) \widehat{\psi}_0(\omega), \label{eq:twoscalef}
\end{equation}
where $\widehat{a}$ denotes the Fourier transform of a function
$a$.

In order to generate an orthonormal $M$-band wavelet basis
$\bigcup_{m\in\NMs,j\in \ZZ}\{M^{-j/2}\psi_{m}(M^{-j}t-k),k\in \ZZ\}$
of $\LL^2(\RR)$,
the following para-unitarity conditions must hold:
\begin{multline}
\forall (m,m^\prime) \in \NMSQ,\\ \sum_{p=0}^{M-1}
H_m(\omega+p \frac{2\pi}{M}) H_{m^\prime}^*(\omega+p
\frac{2\pi}{M}) = M \delta_{m-m'}, \label{eq:paraunitarity}
\end{multline}
where $\delta_m = 1$ if $m=0$ and 0 otherwise. The filter with
frequency response $H_0$ is low-pass whereas the filters with frequency
response $H_m$, $m \in \{1,\ldots,M-2\}$ (resp. $m=M-1$) are band-pass (resp. high-pass).
In this case,
cascading  the $M$-band para-unitary analysis and  synthesis
filter banks, represented by the upper structures in Fig. \ref{fig:Mband}, allows us to decompose
and to  perfectly reconstruct  a given signal.

A ``dual'' $M$-band multiresolution analysis is built by defining
another $M$-band wavelet orthonormal basis associated with a
scaling function $\psi_0^\HH$ and mother wavelets $\psi_m^\HH$,
$m\in \NMs$. More precisely, the mother wavelets are the Hilbert
transforms of the ``original'' ones $\psi_m$, $m\in \NMs$. In the
Fourier domain, the desired property reads:
\begin{equation}
\forall m \in \NMs,\qquad \widehat{\psi}_m^\HH(\omega) = -
\imath\;\mathrm{sign}(\omega) \widehat{\psi}_m(\omega),
\label{eq:Hilbertcond}
\end{equation}
where $\mathrm{sign(\cdot)}$ is the signum function.
 Then, it can be proved \cite{Chaux_C_2005_icassp_2D_dtmbwd} that the dual scaling
function can be chosen such that
\begin{multline}
\forall k \in \ZZ ,\; \forall \omega \in [2k\pi,2(k+1)\pi),\\
\widehat{\psi}_0^\HH(\omega) =
\begin{cases}
(-1)^k e^{-\imath(d+\frac{1}{2})\omega}\;\widehat{\psi}_0(\omega) & \mbox{if $k \geq 0$}\\
(-1)^{k+1} e^{-\imath(d+\frac{1}{2})\omega}\;\widehat{\psi}_0(\omega) & \mbox{otherwise,}
\end{cases}
\label{eq:linkpsi0Hpsi0}
\end{multline}
where $d$ is an arbitrary integer delay. The corresponding
analysis/synthesis para-unitary Hilbert filter banks are
illustrated by the lower structures in Fig.~\ref{fig:Mband}.
Conditions for designing the  involved frequency responses $G_m$,
$m \in \NM$, have been recently provided in
\cite{Chaux_C_2006_tip_ima_adtmbwt}. As the union of two orthonormal basis decomposition, the global dual-tree representation corresponds to a tight frame analysis of $\LL^2(\RR)$.

\section{Second-order moments of the noise wavelet coefficients}
 \label{se:2ndexp}
In this part, we first consider the analysis of a one-dimensional,
real-valued, wide-sense stationary and zero-mean noise $n$, with
autocovariance function
\begin{equation}
\forall (\tau,x) \in \RR^2,\qquad
\Gamma_n(\tau) = \E\{n(x+\tau)n(x)\}.
\label{eq:covb}
\end{equation}
We then extend our results to the two-dimensional case.

\subsection{Expression of the covariances in the $1D$ case}
\label{se:2ndexp1}
We denote by $(n_{j,m}[k])_{k\in \ZZ}$ the coefficients resulting
from a $1D$ $M$-band wavelet decomposition of the noise, in a
given subband $(j,m)$ where $j\in \ZZ$ and $m\in \NM$. In the
$(j,m)$ subband, the wavelet coefficients generated by the dual
decomposition are denoted by $(n^\HH_{j,m}[k])_{k\in \ZZ}$. At resolution level $j$, the statistical second-order
properties of the dual-tree wavelet decomposition of the noise are
characterized as follows.
%The second-order moments of the wavelet coefficients are expressed as follows.
\begin{prop} \label{prop:corrgen}
For all $(m,m')\in \NMSQ$, $([n_{j,m}[k]\quad
n_{j,m}^\HH[k]])_{k\in\ZZ}$ is a wide-sense stationary vector
sequence. More precisely, for all $(\ell,k)\in\ZZ^2$, we have
\begin{align}
\E\{n_{j,m}[k+\ell]n_{j,m'}[k]\}&= \Gamma_{n_{j,m},n_{j,m'}}[\ell] \label{eq:bb1D} \\
&= \int_{-\infty}^\infty \Gamma_n(x)\,
\gamma_{\psi_m,\psi_{m'}}\left(\frac{x}{M^j}-\ell\right)dx \nonumber\\
\E\{n_{j,m}^\HH[k+\ell]n_{j,m'}^\HH[k]\} &=
\Gamma_{n_{j,m}^\HH,n_{j,m'}^\HH}[\ell] \label{eq:bhbh1D} \\ 
&= \int_{-\infty}^\infty
\Gamma_n(x)
\gamma_{\psi_m^\HH,\psi_{m'}^\HH}(\frac{x}{M^j}-\ell)\,dx \nonumber\\
\E\{n_{j,m}[k+\ell]n^{\HH}_{j,m'}[k]\}&=
\Gamma_{n_{j,m},n_{j,m'}^\HH}[\ell] \label{eq:bbh1D} \\
&= \int_{-\infty}^\infty
\Gamma_n(x)
\gamma_{\psi_m,\psi_{m'}^\HH}\left(\frac{x}{M^j}-\ell\right)dx, \nonumber
\end{align}
where the deterministic cross-correlation function of
two real-valued functions $f$ and $g$ in
$\mathrm{L}^2(\RR)$ is expressed as
\begin{equation}
  \forall \tau \in \RR,\qquad
\gamma_{f,g}(\tau) = \int_{-\infty}^\infty
f(x)g(x-\tau)\;dx. \label{eq:intnoise}
\end{equation}
\end{prop}
%substituting $\gamma_{\psi_m,\psi_m^\HH}$ by the autocorrelation
%$\gamma_{\psi_m}$ of function $\psi_m$.
\begin{proof}
See Appendix \ref{ap:corrgen}.
\end{proof}
The classical properties of covariance/correlation functions are
satisfied. In particular, since for all $m \in \NM$, $\psi_m$ and
$\psi_m^\HH$ are unit norm functions, for all $(m,m')\in \NMSQ$,
the absolute values of $\gamma_{\psi_m,\psi_{m}}$,
$\gamma_{\psi_m^\HH,\psi_{m'}^\HH}$ and
$\gamma_{\psi_m,\psi_{m'}^\HH}$ are upper bounded by 1. In
addition, the following symmetry properties are satisfied.
\begin{prop} \label{prop:sym}
For all $(m,m')\in \NM$ with $m=m'=0$ or $mm'\neq 0$, we have
$\gamma_{\psi_m^\HH,\psi_{m'}^\HH}= \gamma_{\psi_m,\psi_{m'}}$. As a consequence,
\begin{equation}
%\forall (k,k') \in \ZZ^2,\qquad
%\E\{n_{j,m}[k]n_{j,m'}[k']\}= \E\{n_{j,m}^\HH[k]n_{j,m'}^\HH[k']\}.
\Gamma_{n_{j,m},n_{j,m'}} = \Gamma_{n_{j,m}^\HH,n_{j,m'}^\HH}.
\label{eq:covbbbhbheq}
\end{equation}
When $m m'\neq 0$, we have
\begin{equation}
\forall \tau \in \RR,\qquad
\gamma_{\psi_m,\psi_{m'}^\HH}(\tau) = -
\gamma_{\psi_{m'},\psi_m^\HH}(-\tau)
\label{eq:symmmplm}
\end{equation}
and, consequently,
\begin{equation}
\forall \ell \in \ZZ,\qquad \Gamma_{n_{j,m},n_{j,m'}^\HH}[\ell] =
-\Gamma_{n_{j,m'},n_{j,m}^\HH}[-\ell]. \label{eq:symmmplmc}
\end{equation}
%of $k-k'$.
Besides, the function $\gamma_{\psi_0,\psi_0^\HH}$ is symmetric
w.r.t. $-d-1/2$, which entails that
$\Gamma_{n_{j,0},n_{j,0}^\HH}$
%$\E\{n_{j,0}[k]n_{j,0}^\HH[k']\}$, viewed as a function of $k-k'$,
is symmetric w.r.t. $d+1/2$.
\end{prop}
\begin{proof}
See Appendix \ref{ap:sym}.
\end{proof}

As a particular case of \eqref{eq:covbbbhbheq}
when $m=m'$, it appears that the sequences
$(n_{j,m}[k])_{k\in\ZZ}$ and $(n_{j,m}^\HH[k])_{k\in\ZZ}$
have the same autocovariance sequence.
We also deduce from  Prop. \ref{prop:sym} that, for all
$m\neq 0$,  $\gamma_{\psi_m,\psi_m^\HH}$ is
an odd function, and the cross-covariance
$\Gamma_{n_{j,m},n_{j,m}^\HH}$
%$\E\{n_{j,m}[k]n_{j,m}^\HH[k']\}$
is an odd sequence.
This implies, in particular, that for
all $m\neq 0$,
%and $k\in \ZZ$,
\begin{equation}
%\E\{n_{j,m}[k]n_{j,m}^\HH[k]\} = 0.
\Gamma_{n_{j,m},n_{j,m}^\HH}[0] = 0. \label{eq:internoise0}
\end{equation}
The latter equality means that, for all $m\neq0$ and
$k\in \ZZ$, the random vector
$[n_{j,m}[k]\quad n_{j,m}^\HH[k]]$ has uncorrelated components with equal variance.

The previous results are applicable to an arbitrary stationary noise but the resulting expressions may be intricate depending on the specific form of the autocovariance $\Gamma_n$.
Subsequently, we will be mainly interested in the study
of the dual-tree decomposition of a white noise, for which tractable expressions of the second-order statistics of the coefficients can be obtained.
% Although the previous results are applicable to an arbitrary stationary noise,
% the study of a white noise will be of main interest
% subsequently.
The autocovariance of $n$ is then given by $\Gamma_n(x) =
\sigma^2\,\delta(x)$, where $\delta$ denotes the Dirac
distribution. As the primal (resp. dual) wavelet basis is
orthonormal, it can be deduced from
\eqref{eq:bb1D}-\eqref{eq:bbh1D} (see Appendix \ref{ap:whitenoise}) that, for all
$(m,m')\in \NMSQ$ and $\ell\in \ZZ$,
\begin{align}
\Gamma_{n_{j,m},n_{j,m'}}[\ell] & =
\Gamma_{n_{j,m}^\HH,n_{j,m'}^\HH}[\ell]
= \sigma^2 \delta_{m-m'} \delta_{\ell}\label{eq:internoiseb1D}\\
%\E\{n_{j,m}[k]n_{j,m'}[k']\}&= \E\{n_{j,m}^\HH[k]n_{j,m'}^\HH[k']\} = \sigma^2 \delta_{m-m'} \delta_{k-k'}\\
\Gamma_{n_{j,m},n_{j,m'}^\HH}[\ell],
%\E\{n_{j,m}[k]b^{\HH}_{j,m'}[k']\}
&= \sigma^2
\gamma_{\psi_m,\psi_{m'}^\HH}\left(-\ell\right),
\label{eq:internoise1D}
\end{align}
where $(\delta_k)_{k\in\ZZ}$ is the Kronecker sequence
($\delta_k = 1$ if $k=0$ and 0 otherwise).
Therefore, $(n_{j,m}[k])_{k\in \ZZ}$ and
$(n_{j,m}^\HH[k])_{k\in \ZZ}$ are cross-correlated
zero-mean, white
random sequences with variance $\sigma^2$.
% \begin{equation}
% \forall (k,k')
% \E\{n_{j,m}[k]n_{j,m}^\HH[k']\} =
% \sigma^2\gama_{\psi_m,\psi_m^\HH}(k'-k)\,.
% \end{equation}

The determination of the cross-covariance requires
the computation of $\gamma_{\psi_m,\psi_{m'}^\HH}$.
We distinguish between the mother ($m' \neq 0$)
and father ($m'= 0$) wavelet case.
\begin{itemize}
\item By using \eqref{eq:Hilbertcond},
for $m'\neq 0$, Parseval-Plancherel formula yields
%\begin{equation}
%\widehat{\gamma}_{\psi_m,\psi_{m'}^\HH}(\omega)=\imath
%\;\mathrm{sign}(\omega) \widehat{\psi}_m(\omega)
%\big(\widehat{\psi}_{m'}(\omega))^*
%\end{equation}
%which leads to:
\begin{align}
\gamma_{\psi_m,\psi_{m'}^\HH}(\tau)= \frac{1}{2\pi}
& \int_{-\infty}^{\infty}\widehat{\psi}_m(\omega)
\big(\widehat{\psi}_{m'}(\omega)^\HH)^*\exp(\imath \omega \tau)\,d\omega \nonumber \\
%&\frac{\imath}{2\pi} \int_{-\infty}^{\infty}\mathrm{sign}(\omega)\;\widehat{\psi}_m(\omega)
%\big(\widehat{\psi}_{m'}(\omega))^*\; \exp(\imath \omega \tau )d\omega\nonumber\\
= -\frac{1}{\pi}\;\mathrm{Im}\Big\{
& \int_{0}^{\infty}\widehat{\psi}_m(\omega)
\big(\widehat{\psi}_{m'}(\omega))^*
\exp(\imath \omega \tau )\,d\omega\Big\},\label{eq:crosspsimpsimp}
%= & -\frac{1}{\pi}
%\int_{0}^{\infty}|\widehat{\psi}_m(\omega)|^2\;\sin(\omega\tau)\;d\omega.
%\label{eq:gammapsi}
\end{align}
where $\mathrm{Im}\{z\}$ denotes the imaginary part
of a complex $z$.
\item According to \eqref{eq:linkpsi0Hpsi0}, for $m'=0$ we find, after some simple calculations:
\begin{multline}
\gamma_{\psi_m,\psi_0^\HH}(\tau)=
\frac{1}{\pi}\,\mathrm{Re}\Big\{\sum_{k=0}^\infty
(-1)^k  
\int_{2k\pi}^{2(k+1)\pi}
\widehat{\psi}_m(\omega)
 \\ \times \big(\widehat{\psi}_0(\omega)\big)^*
\exp\big(\imath\omega\,(\frac{1}{2}+\tau+d)\big)\,d\omega\Big\},
%\frac{1}{\pi} \sum_{k=0}^\infty
%(-1)^k \int_{2k\pi}^{2(k+1)\pi}
%|\widehat{\psi}_0(\omega)|^2\;\cos\big(\omega\;(\frac{1}{2}+\tau+d)\big)\;d\omega.
\label{eq:crosspsimpsi0}
\end{multline}
where $\mathrm{Re}\{z\}$ denotes the real part
of a complex $z$.
\end{itemize}
In both cases, we have
\begin{equation}
|\gamma_{\psi_m,\psi_{m'}^\HH}(\tau)|\leq \frac{1}{\pi} \int_{0}^{\infty}|\widehat{\psi}_m(\omega)
\widehat{\psi}_{m'}(\omega)|\;d\omega.
\label{eq:boundcrosspsimpsimp}
\end{equation}
For $M$-band wavelet decompositions, selective filter banks are commonly used. Provided that this selectivity property is satisfied,
the cross term $|\widehat{\psi}_m(\omega)\widehat{\psi}_{m'}(\omega)|$ can be expected to be close
to zero and the upper bound in \eqref{eq:boundcrosspsimpsimp} to take small values when $m\neq m'$.
This fact will
be discussed in Section \ref{se:interbandsimuls} based on numerical results. On the contrary, when $m=m'$, the cross-correlation
functions always need to be evaluated more carefully. In Section \ref{se:examplewav}, we will therefore focus   on the functions:
\begin{align}
\gamma_{\psi_m,\psi_{m}^\HH}(\tau)= & -\frac{1}{\pi}
\int_{0}^{\infty}|\widehat{\psi}_m(\omega)|^2\,\sin(\omega\tau)\;d\omega, \; m\neq 0,
\label{eq:gammapsi}\\
\gamma_{\psi_0,\psi_0^\HH}(\tau)=&
\frac{1}{\pi} \sum_{k=0}^\infty
(-1)^k \nonumber \\ 
\times \int_{2k\pi}^{2(k+1)\pi}
&  |\widehat{\psi}_0(\omega)|^2\,\cos\big(\omega\;(\frac{1}{2}+\tau+d)\big)\;d\omega.
\label{eq:gammaphi}
\end{align}
% It can be noticed that these cross-correlation functions possess some
% \textbf{symmetry properties}. Actually, for all $m\neq
% 0$, $\gamma_{\psi_m,\psi_m^\HH}$ is an odd function whereas
% $\gamma_{\psi_0,\psi_0^\HH}$, viewed as a function of $\tau+d$, is
% symmetric with respect to $-1/2$.

Note that, in this paper, we do not consider interscale correlations. Although expressions of the second-order statistics similar to the intrascale ones can be derived, sequences of wavelet coefficients defined at different resolution levels are generally not cross-stationary \cite{Leporini_D_1999_tit_hig_owpcfa}.

\subsection{Extension to the $2D$ case}
We now consider the analysis of a two-dimensional noise $n$, which
is also assumed to be real, wide-sense stationary with zero-mean and autocovariance
function
\begin{equation}
%\forall (x,y,\tau,\theta) \in \RR^4,\qquad \Gamma_b(\tau,\theta) = \E\{b(x,y) b(x-\tau,y-\theta)\}.
\forall (\ttau,\xx) \in \RR^2 \times \RR^2,\qquad \Gamma_n(\ttau) = \E\{n(\xx+\ttau) n(\xx)\}.
\nonumber
\end{equation}
%Such noises are frequently considered in image
%denoising problems.\\
We can proceed similarly to the previous section. We denote by
$(n_{j,\mm}[\kk])_{\kk\in\ZZ^2}$ the coefficients resulting from a
$2D$ separable $M$-band wavelet decomposition \cite{Steffen_P_1993_tsp_the_rmbwb} of the
noise, in a given subband $(j,\mm)\in \ZZ \times \NMSQ$. The
wavelet coefficients of the dual decomposition are
denoted by $(n^\HH_{j,\mm}[\kk])_{\kk\in\ZZ^2}$.
%Putting emphasis on the cross-covariance between the primal
%and dual wavelet coefficients, we have
We obtain expressions of the covariance fields similar to \eqref{eq:bb1D}-\eqref{eq:bbh1D}:
%For example, putting emphasis on the cross-covariance between the primal
%and dual wavelet coefficients, we have,
for all $j\in\ZZ$, $\mm = (m_1,m_2) \in \NMSQ$, $\mm' =
(m_1',m_2') \in \NMSQ$, $\eell = (\ell_1,\ell_2)\in \ZZ^2$ and
$\kk\in \ZZ^2$,
\begin{align}
\Gamma_{n_{j,\mm},n_{j,\mm'}}[\eell] & =
\E\{n_{j,\mm}[\kk+\eell]n_{j,\mm'}[\kk]\}\nonumber\\
&=\int_{-\infty}^\infty\!\int_{-\infty}^\infty \Gamma_n(x_1,x_2)
\gamma_{\psi_{m_1},\psi_{m_1'}}\left(\frac{x_1}{M^j}-\ell_1\right)\nonumber \\
& \qquad \quad \times \gamma_{\psi_{m_2},\psi_{m_2'}}\left(\frac{x_2}{M^j}-\ell_2\right)\,dx_1 dx_2
\label{eq:bb2D}\\
\Gamma_{n_{j,\mm}^\HH,n_{j,\mm'}^\HH}[\eell] & =
\E\{n_{j,\mm}^\HH[\kk+\eell]n^{\HH}_{j,\mm'}[\kk]\}\nonumber\\
&=\int_{-\infty}^\infty\!\int_{-\infty}^\infty \Gamma_n(x_1,x_2)
\gamma_{\psi_{m_1}^\HH,\psi_{m_1'}^\HH}\left(\frac{x_1}{M^j}-\ell_1\right) \nonumber \\
& \qquad \quad \times \gamma_{\psi_{m_2}^\HH,\psi_{m_2'}^\HH}\left(\frac{x_2}{M^j}-\ell_2\right)\,dx_1 dx_2\\
\Gamma_{n_{j,\mm},n_{j,\mm'}^\HH}[\eell] & =
\E\{n_{j,\mm}[\kk+\eell]n^{\HH}_{j,\mm'}[\kk]\}\nonumber\\
&=
\int_{-\infty}^\infty\!\int_{-\infty}^\infty \Gamma_n(x_1,x_2)
\gamma_{\psi_{m_1},\psi_{m_1'}^\HH}\left(\frac{x_1}{M^j}-\ell_1\right) \nonumber \\
& \qquad \quad  \times \gamma_{\psi_{m_2},\psi_{m_2'}^\HH}\left(\frac{x_2}{M^j}-\ell_2\right)\,dx_1 dx_2.
\label{eq:bbh2D}
\end{align}
From the properties of the correlation functions of the wavelets
and the scaling function as given by Prop. \ref{prop:sym}, it can
be deduced that, when ($m_1 = m_1' = 0$ or $m_1m_1' \neq 0$) and
($m_2 = m_2' = 0$ or $m_2m_2' \neq 0$),
\begin{equation}
\Gamma_{n_{j,\mm},n_{j,\mm'}}
= \Gamma_{n_{j,\mm}^\HH,n_{j,\mm'}^\HH}.
%\E\{n_{j,\mm}[\kk]n_{j,\mm'}[\kk']\}
%= \E\{n_{j,\mm}^\HH[\kk]b^{\HH}_{j,\mm'}[\kk']\}.
\label{eq:bbbhbh}
\end{equation}
Some additional symmetry properties are straightforwardly obtained
from Prop. \ref{prop:sym}. In particular, for all $\mm \in
\NMsSQ$, the cross-covariance $\Gamma_{n_{j,\mm},n^\HH_{j,\mm}}$
is an even sequence.
%and $(\kk,\kk')\in\ZZ^2×\ZZ^2$, we have
%\begin{equation}
%\Gamma_{n_{j,\mm},b^\HH_{j,\mm}}(\eell)
%= \Gamma_{n_{j,\mm},b^\HH_{j,\mm}}(-\eell)
%\E\{n_{j,\mm}[\kk]b^\HH_{j,\mm}[\kk']\}
%= \E\{n_{j,\mm}[\kk']b^\HH_{j,\mm}[\kk]\}.
%\label{eq:bbhbbh}
%\end{equation}
An important consequence of the latter properties concerns the $2\times2$ linear combination of the primal and dual wavelet coefficients
which is often implemented in dual-tree decompositions. As
explained in \cite{Chaux_C_2005_icassp_2D_dtmbwd}, the main advantage of
such a post-processing is to better capture the directional
features in the analyzed image. More precisely, this amounts to
performing the following unitary transform of the detail
coefficients, for $\mm \in \NMsSQ$:
\begin{align}
\forall \kk \in \ZZ^2,\qquad
w_{j,\mm}[\kk] &= \frac{1}{\sqrt{2}}
(n_{j,\mm}[\kk]+n^\HH_{j,\mm}[\kk])\label{eq:transfnoise1}\\
w_{j,\mm}^\HH[\kk] &= \frac{1}{\sqrt{2}}
(n_{j,\mm}[\kk]-n^\HH_{j,\mm}[\kk]).
\label{eq:transfnoise2}
\end{align}
(The transform is usually not applied when $m_1 = 0$ or $m_2=0$.)
The covariances of the transformed fields of noise coefficients $(w_{j,\mm}[\kk])_{\kk\in \ZZ^2}$ and
$(w_{j,\mm}^\HH[\kk])_{\kk\in \ZZ^2}$ then take
the following expressions:
\begin{prop}
For all $\mm \in \NMsSQ$ and $\eell \in \ZZ^2$,
%$(\kk,\kk')\in\ZZ^2× \ZZ^2$,
\begin{align}
\Gamma_{w_{j,\mm},w_{j,\mm}}[\eell] & =
\Gamma_{n_{j,\mm},n_{j,\mm}}[\eell] +
\Gamma_{n_{j,\mm},n_{j,\mm}^\HH}[\eell]
\label{eq:2Drot1}\\
\Gamma_{w_{j,\mm}^\HH,w_{j,\mm}^\HH}[\eell] & =
\Gamma_{n_{j,\mm},n_{j,\mm}}[\eell] -
\Gamma_{n_{j,\mm},n_{j,\mm}^\HH}[\eell]
\label{eq:2Drot2}\\
\Gamma_{w_{j,\mm},w_{j,\mm}^\HH}[\eell] & = 0.
% \E\{n_{j,\mm}[\kk]n_{j,\mm}[\kk']\}&=
% \E\{n_{j,\mm}[\kk]n_{j,\mm}[\kk']\}
% + \E\{n_{j,\mm}[\kk]b^{\HH}_{j,\mm}[\kk']\}
% \label{eq:2Drot1}\\
% \E\{n_{j,\mm}^\HH[\kk]n^{\HH}_{j,\mm}[\kk']\}&=
% \E\{n_{j,\mm}[\kk]n_{j,\mm}[\kk']\}
% - \E\{n_{j,\mm}[\kk]b^{\HH}_{j,\mm}[\kk']\}
% \label{eq:2Drot2}\\
% \E\{n_{j,\mm}[\kk]n^{\HH}_{j,\mm}[\kk']\}&=
% 0.
% \label{eq:2Drot3}
\end{align}
\label{prop:posttransf}
\end{prop}
\begin{proof}
See Appendix \ref{ap:posttransf}.
\end{proof}
This shows that the post-transform not only
provides a better directional analysis of the image of interest but also plays an important role w.r.t. the noise analysis.
Indeed, it allows to completely cancel the correlations between the primal and dual noise coefficient fields
obtained for a given value of $(j,\mm)$.
In turn, this operation introduces some spatial noise correlation in each subband.

For a two-dimensional white noise, $\Gamma_n(\xx) =
\sigma^2\ \delta(\xx)$ and the coefficients
$(n_{j,\mm}[\kk])_{\kk\in\ZZ^2}$ and $(n_{j,\mm'}^\HH[\kk])_{\kk\in\ZZ^2}$
are such that, for all $\eell = (\ell_1,\ell_2)\in \ZZ^2$,
%and $\kk'=(k_1',k_2')\in \ZZ^2$,
\begin{align}
\Gamma_{n_{j,\mm},n_{j,\mm'}}[\eell] & =
\Gamma_{n_{j,\mm}^\HH,n_{j,\mm'}^\HH}[\eell]
%\E\{n_{j,\mm}[\kk]n_{j,\mm'}[\kk']\}
%& = \E\{b^\HH_{j,\mm}[\kk]b^\HH_{j,\mm'}[\kk']\}
= \sigma^2 \delta_{m_1-m_1'} \delta_{m_2-m_2'}
\delta_{\ell_1} \delta_{\ell_2}\label{eq:internoiseb}\\
\Gamma_{n_{j,\mm},n_{j,\mm'}^\HH}[\eell] & =
%\E\{n_{j,\mm}[\kk]n_{j,\mm'}^\HH[\kk']\} &=
\sigma^2\gamma_{\psi_{m_1},\psi_{m_1'}^\HH}(-\ell_1)
\gamma_{\psi_{m_2},\psi_{m_2'}^\HH}(-\ell_2)\,. \label{eq:internoise}
\end{align}
% are white, centered, of variance $\sigma^2$. Furthermore, using
% previous expression, we  easily deduce that
% \begin{equation}
% \E\{n_{j,m,m'}[k,l]n_{j,m,m'}^\HH[\kk']\} =
% \sigma^2\gamma_{\psi_m,\psi_m^\HH}(k'-k)
% \gamma_{\psi_{m'},\psi_{m'}^\HH}(l'-l)\,. \label{eq:internoise}
% \end{equation}

As a consequence of Prop. \ref{prop:sym},
in the case when $\eell=\mathbf{0}$, we  conclude
that, for $(m_1\neq0$ or $m_2\neq0)$ and $\kk \in \ZZ^2$,
the vector $[n_{j,\mm}[\kk]\quad n_{j,\mm}^\HH[\kk]]$ has uncorrelated components with equal variance.
This property holds more generally for 2D noises with separable
covariance functions.

\section{Some asymptotic properties}
\label{se:asympt}
In the previous section, we have shown that the correlations
of the basis functions play a prominent role in the determination of the second-order statistical properties of the noise
coefficients. To estimate the strength of the dependencies
between the coefficients, it is useful to determine
the decay of the correlation functions.
The following result allows to evaluate their decay.
\begin{prop}\label{prop:decaycor}
Let $(N_1,\ldots,N_{M-1}) \in (\NN^*)^{M-1}$ and define $N_0 =
\min_{m\in \NMs} N_m$.
%Assume that, for all $m\in\NM$,
%the function $\widehat{\psi}_m$ is $2N_m+1$ times continuously
%differentiable on $\RR$ and, for all $n\in\{0,\ldots,2N_m+1\}$, its $q$-th order derivatives $\widehat{\psi}_m^{(q)}$ belongs to $\LL^2(\RR)$.
Assume that, for all $m\in\NM$, the function
$|\widehat{\psi}_m|^2$ is $2N_m+1$ times continuously
differentiable on $\RR$ and, for all $q\in\{0,\ldots,2N_m+1\}$,
its $q$-th order derivatives $(|\widehat{\psi}_m|^2)^{(q)}$ belong
to $\LL^1(\RR)$.\footnote{By convention, the derivative of order 0
of a function is the function itself.} Further assume that, for
all $m \neq 0$, $\widehat{\psi}_m(\omega) =O(\omega^{N_m})$ as
$\omega \to 0$. Then, there exists $C \in \RR_+$ such that, for
all $m\in\NM$,
\begin{equation}
\forall \tau \in \RR^*,\qquad
|\gamma_{\psi_m,\psi_m}(\tau)|\leq \frac{C}{|\tau|^{2N_m+1}}
\label{eq:decayasymptcor1}
\end{equation}
and
\begin{equation}
\forall \tau \in \RR^*,\qquad
|\gamma_{\psi_m,\psi_m^\HH}(\tau)|\leq \frac{C}{|\tau|^{2N_m+1}}.
\label{eq:decayasymptcor2}
\end{equation}
% the decay rate of
% $\gamma_{\psi_m,\psi_m}(\tau)$ and $\gamma_{\psi_m,\psi_m^\HH}(\tau)$ is at least $|\tau|^{-2N}$
\end{prop}
\begin{proof}
See Appendix \ref{ap:decaycor}.
\end{proof}
Note that, for all $m\in\NM$,
 the assumptions concerning $|\widehat{\psi}_m|^2$ are satisfied if
$\widehat{\psi}_m$ is $2N_m+1$ times continuously
differentiable on $\RR$ and, for all $q\in\{0,\ldots,2N_m+1\}$, its $q$-th order derivatives $\widehat{\psi}_m^{(q)}$ belong to $\LL^2(\RR)$.
Indeed, if $\widehat{\psi}_m$ is $2N_m+1$ times continuously
differentiable
on $\RR$, so is $|\widehat{\psi}_m|^2$. Leibniz formula
allows us to express its derivative of order $q \in \{0,\ldots,2N_m+1\}$ as
\begin{equation}
(|\widehat{\psi}_m|^2)^{(q)} =\sum_{\ell=0}^q
\binom{q}{\ell}
(\widehat{\psi}_m)^{(\ell)} (\widehat{\psi}_m^*)^{(q-\ell)}.
\label{eq:leibniz}
\nonumber
\end{equation}
Consequently, if for all $\ell \in \{0,\ldots,q\}$, $\widehat{\psi}_m^{(\ell)}
\in \LL^2(\RR)$, then
$(|\widehat{\psi}_m|^2)^{(q)}\in \LL^1(\RR)$.

Note also that, for integrable wavelets, the assumption $\widehat{\psi}_m(\omega) = O(\omega^{N_m})$
as $\omega \to 0$ means that the wavelet $\psi_m$, $m\neq 0$, has $N_m$ vanishing moments.

Therefore, the decay rate of the wavelet correlation functions is
all the more important as the Fourier transforms of the basis
functions $\psi_m$, $m \in \NM$, are regular (\emph{i.e.} the wavelets
have fast decay themselves) and the number of vanishing moments
is large. The latter condition is useful to ensure that Hilbert
transformed functions $\psi_m^\HH$ have regular spectra too. It
must be emphasized that Prop. \ref{prop:decaycor} guarantees that
the asymptotic decay of the wavelet correlation functions is
\emph{at most} $|\tau|^{-2N_m-1}$. A faster decay can be obtained in practice
for some wavelet families. For example, when $\psi_m$ is compactly
supported, $\gamma_{\psi_m,\psi_m}$ also has a compact support. In
this case however, $\psi_m^\HH$ cannot be compactly supported
\cite{Chaux_C_2006_tip_ima_adtmbwt}, so that the bound in \eqref{eq:decayasymptcor2} remains
of interest. Examples will be discussed in more detail in Section~\ref{se:examplewav}.

It is also worth noticing that the obtained upper bounds on the correlation functions allow us to evaluate the decay rate of the covariance sequences of the dual-tree wavelet coefficients of a stationary noise as expressed below.
\begin{prop} \label{prop:decaycov1D}
Let $n$ be a 1D zero-mean wide-sense stationary random process.
Assume that either
$n$ is a white noise or its autocovariance function
is with exponential decay, that is there exist
$A\in \RR_+$ and $\alpha \in \RR_+^*$, such that
\begin{equation}
\forall \tau \in \RR,\qquad
|\Gamma_n(\tau)| \leq A e^{-\alpha |\tau|}.
\label{eq:expdecay}
\nonumber
\end{equation}
Consider also functions $\psi_m$, $m\in \NM$, satisfying the
assumptions of Prop. \ref{prop:decaycor}. Then, there exists
$\widetilde{C} \in \RR_+$ such that for all $j\in \ZZ$, $m\in\NM$
and $\ell \in \ZZ^*$,
\begin{align}
& |\Gamma_{n_{j,m},n_{j,m}}[\ell]| \leq
\frac{\widetilde{C}}{1+|\ell|^{2N_m+1}}
\label{eq:boundcovann}\\
& |\Gamma_{n_{j,m},n_{j,m}^\HH}[\ell]| \leq
\frac{\widetilde{C}}{1+|\ell|^{2N_m+1}}. \label{eq:boundcovannh}
\end{align}
\end{prop}
\begin{proof}
See Appendix \ref{ap:decaycov1D}.
\end{proof}
The decay property of the covariance sequences
readily extends to the 2D case:
\begin{prop}
Let $n$ be a 2D zero-mean wide-sense stationary random field. Assume that either
$n$ is a white noise or its autocovariance function
is with exponential decay, that is there exist
$A\in \RR_+$ and $(\alpha_1,\alpha_2) \in (\RR_+^*)^2$, such that
\begin{equation}
\forall (\tau_1,\tau_2) \in \RR^2,\qquad
|\Gamma_n(\tau_1,\tau_2)| \leq A e^{-\alpha_1 |\tau_1|-\alpha_2 |\tau_2|}.
\label{eq:expdecay2D}
\end{equation}
Consider also functions $\psi_m$, $m\in \NM$, satisfying the
assumptions of Prop. \ref{prop:decaycor}. Then, there exists
$\widetilde{C} \in \RR_+$ such that for all $j\in \ZZ$, $\mm
\in\NMSQ$ and $\eell =(\ell_1,\ell_2) \in \ZZ^2$,
\begin{align}
& |\Gamma_{n_{j,\mm},n_{j,\mm}}[\eell]| \leq \frac{\widetilde{C}}
{(1+|\ell_1|^{2N_m+1})(1+|\ell_2|^{2N_m+1})}
\label{eq:boundcovann2}\\
& |\Gamma_{n_{j,\mm},n_{j,\mm}^\HH}[\eell]| \leq
\frac{\widetilde{C}} {(1+|\ell_1|^{2N_m+1})(1+|\ell_2|^{2N_m+1})}.
\label{eq:boundcovannh2}
\end{align}
Besides, for all $j\in \ZZ$, $\mm \in\NMsSQ$ and $\eell
=(\ell_1,\ell_2) \in \ZZ^2$,
\begin{align}
& |\Gamma_{w_{j,\mm},w_{j,\mm}}[\eell]| \leq \frac{2\widetilde{C}}
{(1+|\ell_1|^{2N_m+1})(1+|\ell_2|^{2N_m+1})}
\label{eq:boundcovaww}\\
& |\Gamma_{w_{j,\mm}^\HH,w_{j,\mm}^\HH}[\eell]| \leq
\frac{2\widetilde{C}}
{(1+|\ell_1|^{2N_m+1})(1+|\ell_2|^{2N_m+1})}.
\label{eq:boundcovawhwh}
\end{align}
\end{prop}
\begin{proof}
Due to the separability of the 2D dual-tree wavelet analysis,
\eqref{eq:boundcovann2} and \eqref{eq:boundcovannh2} are obtained quite similarly to \eqref{eq:boundcovann} and \eqref{eq:boundcovannh}.
The proof of \eqref{eq:boundcovaww} and \eqref{eq:boundcovawhwh}
then follows from \eqref{eq:2Drot1} and \eqref{eq:2Drot2}.
\end{proof}

The two previous propositions provide
upper bounds on the decay rate of the covariance sequences of the dual-tree wavelet coefficients,
when the norm of the lag variable  ($\ell$ or $\eell$) takes large values.
We end this section by providing asympotic results at coarse
resolution (as $j\to \infty$).

\begin{prop} \label{prop:asymptj}
Let $n$ be a 1D zero-mean wide-sense stationary process with
covariance function $\Gamma_n \in \LL^1(\RR) \cap \LL^2(\RR)$.
%Assume that, for all $m\in \NM$, $\psi_m \in \LL^1(\RR)$.
Then,
for all $(m,m') \in \NMSQ$, we have
\begin{align}
\lim_{j\to \infty} \Gamma_{n_{j,m},n_{j,m'}}[\ell]
&= \widehat{\Gamma}_n(0)\,\delta_{m-m'} \delta_{\ell}\label{eq:internoiseb1Dasym}\\
\lim_{j\to \infty}\Gamma_{n_{j,m},n_{j,m'}^\HH}[\ell]
&=\widehat{\Gamma}_n(0)\,
\gamma_{\psi_m,\psi_{m'}^\HH}\left(-\ell\right).
\label{eq:internoise1Dasym}
\end{align}
\end{prop}
\begin{proof}
See Appendix \ref{ap:asymptj}.
\end{proof}
In other words, at coarse resolution in the transform domain,  a stationary noise $n$ with
arbitrary covariance function $\Gamma_n$ behaves like a white
noise with spectrum density $\widehat{\Gamma}_n(0)$. This fact
further emphasizes the interest in studying more precisely the
dual-tree wavelet decomposition of a white noise. Note also that,
by calculating higher order cumulants of the dual-tree wavelet
coefficients and using techniques as in \cite{Leporini_D_1999_tit_hig_owpcfa,Touati_S_2002_jasp_som_rwpdnp}, it could be
proved that, for all $(m,m')\in \NMSQ$ and $(k,k') \in \ZZ^2$,
$[n_{j,m}(k)\quad n_{j,m'}^\HH(k')]$ is asymptotically normal as
$j\to \infty$. Although Prop. \ref{prop:asymptj} has been stated
for 1D random processes, we finally point out that quite similar
results are obtained in the 2D case.

\section{Wavelet families examples}\label{se:examplewav}
For a white noise (see \eqref{eq:internoiseb1D}, \eqref{eq:internoise1D}, \eqref{eq:internoiseb} and \eqref{eq:internoise}) or for arbitrary wide-sense stationary noises analyzed at coarse resolution (cf. Prop. \ref{prop:asymptj}), we have seen that the cross-correlation functions between the primal and dual wavelets taken at integer values are the main features. In order to better evaluate the impact of the wavelet choice, we will now
specify the expressions of these cross-correlations for different
wavelet families.

\subsection{$M$-band Shannon wavelets}
$M$-band Shannon wavelets (also called sinc wavelets in the
literature) correspond to an ideally selective analysis  in the
frequency domain. These wavelets also appear as a limit case for
many wavelet families, e.g. Daubechies's or spline wavelets. We
have then, for all $m \in \NM$,
\begin{equation}
\widehat{\psi}_m(\omega)=
\mathbf{1}_{]-(m+1)\pi,-m\pi]\cup[m\pi,(m+1)\pi[}(\omega),
\nonumber
\end{equation}
where $\mathbf{1}_{\mathbb{S}}$ denotes the characteristic function of the
set $\mathbb{S} \subset \RR$:
\begin{equation}
\mathbf{1}_{\mathbb{S}}(\omega) = \begin{cases}
1 & \mbox{if $\omega \in \mathbb{S}$}\\
0 & \mbox{otherwise.}
\end{cases}
\nonumber
\end{equation}
In this case, \eqref{eq:gammaphi} reads: $\forall \tau \in \RR,$
\begin{align}
\gamma_{\psi_0,\psi_0^\HH}(\tau)&=\frac{1}{\pi}\int_{0}^{\pi}\cos{\Big((\frac{1}{2}+d+ \tau)\omega\Big)} d\omega\nonumber\\
& = \begin{cases}
\displaystyle \frac{(-1)^d \cos{(\pi \tau )}}{\pi(\frac{1}{2}+ d + \tau)} & \mbox{if $\tau \neq -d - \frac{1}{2}$}\\
1 & \mbox{otherwise.}
\end{cases}
\nonumber
\end{align}
For $m\in \NMs$, \eqref{eq:gammapsi} leads to $\forall \tau \in \RR,$
\begin{align}
\gamma_{\psi_m,\psi_m^\HH}(\tau)&=-\frac{1}{\pi}\int_{m\pi}^{(m+1)\pi}\sin{(\omega \tau)} d\omega
\nonumber\\
& = \begin{cases}
\displaystyle \frac{\cos\big((m+1)\pi\tau\big) - \cos(m\pi\tau)}{\pi \tau} & \mbox{if $\tau \neq 0$}\\
0 & \mbox{otherwise.}
\end{cases}
\nonumber
\end{align}
We deduce from the two previous expressions that,
for all $\ell \in \ZZ$,
\begin{align}
\gamma_{\psi_0,\psi_0^\HH}(\ell)&=\frac{(-1)^{(d+\ell)}}{\pi(d+\ell+\frac{1}{2})}
\label{eq_intercorr_shannon0}\\
\forall{m\neq0},\qquad \nonumber \\
\gamma_{\psi_m,\psi_m^\HH}(\ell)&=\begin{cases}
\displaystyle (-1)^{(m+1)\ell}\frac{1-(-1)^\ell}{\pi \ell} & \mbox{if\ $\ell \neq 0$}\\
0 & \mbox{otherwise}.
\end{cases}
\label{eq_intercorr_shannon}
\end{align}
We can remark that, for all $(m,m') \in \NMsSQ$,
\begin{equation}
\gamma_{\psi_m,\psi_m^\HH}(\ell) = (-1)^{(m'-m)\ell}\gamma_{\psi_{m'},\psi_{m'}^\HH}(\ell)
\label{eq:symShanMey}
\end{equation}
and
$\gamma_{\psi_m,\psi_m^\HH}(\ell) = 0$, when $\ell$ is odd.
Besides, the correlation sequences decay pretty slowly as
$\ell^{-1}$. We also note that, as the functions $\psi_m$,
$m\in\NM$, have non-overlapping spectra,
\eqref{eq:bb1D}-\eqref{eq:bbh1D} (resp.
\eqref{eq:bb2D}-\eqref{eq:bbh2D}) allow us to conclude that,
dual-tree noise wavelet coefficients corresponding respectively to
subbands $(j,m)$ and $(j,m')$ with $m\neq m'$ (resp. $(j,m_1,m_2)$
and $(j,m_1',m_2')$ with $m_1 \neq m_1'$ or $m_2 \neq m_2'$) are
perfectly uncorrelated.

\subsection{Meyer wavelets} \label{ssec:ondmeyer} These wavelets
\cite{Meyer_Y_1991_book_ond_o1}, \cite[p.
116]{Daubechies_I_1992_book_ten_lw} are also band-limited but with
smoother transitions than Shannon wavelets. The scaling function
is consequently defined as
\begin{equation}
\widehat{\psi}_0(\omega)=
\begin{cases}
1 & \mbox{if $0\leq |\omega| \leq \pi(1-\epsilon)$}\\
\displaystyle W\Big(\frac{|\omega|}{2\pi\epsilon}-\frac{1-\epsilon}{2\epsilon}\Big) & \mbox{if $\pi(1-\epsilon)\leq |\omega| \leq \pi(1+\epsilon)$}\\
0 & \mbox{otherwise,}
\end{cases}
\label{eq:Meyer0}
\end{equation}
where $0< \epsilon \leq 1/(M+1)$ and
\begin{equation}
\forall\, \theta \in [0,1], \qquad
W(\theta)=\cos\Big(\frac{\pi}{2}\nu(\theta)\Big)
\nonumber
\end{equation}
 with $\nu\;:\;
[0,1] \to [0,1]$ such that
\begin{align}
&\nu(0) = 0\\
&\forall\, \theta \in
[0,1], \quad \nu(1-\theta) = 1-\nu(\theta).
\nonumber
\end{align}
Then, it can be noticed that
\begin{equation}
\forall\, \theta \in [0,1],\qquad
W^2(1-\theta) = 1-W^2(\theta).
\label{eq:propwindow}
\end{equation}
 A common choice for the
$\nu$ function is \cite[p. 119]{Daubechies_I_1992_book_ten_lw}:
\begin{equation}
\forall\, \theta \in [0,1],\qquad
\nu(\theta)=\theta^4(35-84\,\theta+70\,\theta^2-20\,\theta^3).
\label{eq:Meyerexample}
\end{equation}
For $m\in \{1,\ldots,M-2\}$, the associated $M$-band wavelets are given by \eqref{eq:Meyerm}
\begin{figure*}[!t]
\begin{equation}
\widehat{\psi}_m(\omega) = \begin{cases}
\displaystyle e^{\imath \eta_m(\omega)} W\Big(\frac{m+\epsilon}{2\epsilon}-\frac{|\omega|}{2\pi\epsilon}\Big)
& \mbox{if $(m-\epsilon)\pi \leq |\omega| \leq (m+\epsilon)\pi$}\\
e^{\imath \eta_m(\omega)} & \mbox{if $(m+\epsilon)\pi \leq |\omega| \leq (m+1-\epsilon)\pi$}\\
\displaystyle e^{\imath \eta_m(\omega)} W\Big(\frac{|\omega|}{2\pi \epsilon}-\frac{m+1-\epsilon}{2\epsilon}\Big)
& \mbox{if $(m+1-\epsilon)\pi \leq |\omega| \leq (m+1+\epsilon)\pi$}\\
0 & \mbox{otherwise}
\end{cases}
\label{eq:Meyerm}
\end{equation}
\end{figure*}
while, for the last wavelet, we have \eqref{eq:MeyerM}.
\begin{figure*}
\begin{equation}
\widehat{\psi}_{M-1}(\omega) = \begin{cases}
\displaystyle e^{\imath \eta_{M-1}(\omega)}
W\Big(\frac{M-1+\epsilon}{2\epsilon}-\frac{|\omega|}{2\pi\epsilon}\Big)
& \mbox{if $(M-1-\epsilon)\pi \leq |\omega| \leq (M-1+\epsilon)\pi$}\\
e^{\imath \eta_{M-1}(\omega)}
&\mbox{if $(M-1+\epsilon)\pi \leq |\omega| \leq M(1-\epsilon)\pi$}\\
\displaystyle e^{\imath  \eta_{M-1}(\omega)} W\Big(\frac{|\omega|}{2\pi \epsilon M}-\frac{1-\epsilon}{2\epsilon}\Big)
& \mbox{if $M(1-\epsilon) \pi < |\omega| \leq M(1+\epsilon)\pi$}\\
0 & \mbox{otherwise.}
\end{cases}
\label{eq:MeyerM}
\end{equation}
\hrulefill
\vspace*{4pt}
\end{figure*}
Hereabove, the phase functions $\eta_m$, $m\in \NMs$, are odd
functions and we have
\begin{multline}
\forall \omega \in (M\pi,M(1+\epsilon)\pi),\qquad\\
\eta_{M-1}(\omega) = -\eta_{M-1}(2M\pi-\omega) \mod 2\pi.
\nonumber
\end{multline}
In addition, for the orthonormality condition to be satisfied, the following
recursive equations must hold:
\begin{multline}
\forall \omega \in ((m-\epsilon)\pi,(m+\epsilon)\pi),\qquad\\
\eta_m(\omega-2m\pi)-\eta_{m-1}(\omega-2m\pi) \\= \eta_m(\omega)-\eta_{m-1}(\omega) + \pi 
\mod 2\pi.
\nonumber
\end{multline}
by setting: $\forall \omega \in \RR$, $\eta_0(\omega) = 0$. Generally, linear phase solutions to the previous equation are chosen \cite{Tennant_B_2003_sol_omcbwcp}.

Using the above expressions, the cross-correlations between the Meyer basis functions and their dual counterparts are derived in Appendix  \ref{ap:Meyer}.
It can be deduced from these results that: $\forall \; \ell \in
\ZZ$,
\begin{equation}
\gamma_{\psi_0,\psi_0^\HH}(\ell) =
\frac{(-1)^{d+\ell}}{\pi(d+\ell+\frac{1}{2})}
-(-1)^{d+\ell}\,I_\epsilon\Big(d+\ell+\frac{1}{2}\Big),
\label{eq_inter_meyer0}
\end{equation}
where
\begin{equation}
\forall x \in \RR, \; I_\epsilon(x) = 2 \epsilon\int_0^{1}W^2\Big(\frac{1+\theta}{2}\Big)\,
\sin{(\pi\epsilon x\theta)}\,d\theta.
\end{equation}
For the wavelets, we have when $m \in \{1,\ldots,M-2\}$:
\begin{multline}
\gamma_{\psi_m,\psi_m^\HH}(\ell) =\\
\begin{cases}
\displaystyle (-1)^{(m+1)\ell} \big(1-(-1)^\ell\big) \left(
\frac{1}{\pi \ell} - I_\epsilon(\ell)\right) & \mbox{if $\ell \neq 0$}\\
0 & \mbox{otherwise}
\end{cases}
\label{eq_inter_meyerm}
\end{multline}
whereas
\begin{multline}
\gamma_{\psi_{M-1},\psi_{M-1}^\HH}(\ell) =\\
\begin{cases}
\displaystyle (-1)^{M\ell} \big( \frac{1-(-1)^\ell}{\pi\ell} + (-1)^{\ell} I_\epsilon(\ell)-  I_{M\epsilon}(\ell) \big)& \mbox{if $\ell \neq 0$}\\
0 & \mbox{otherwise.}
\end{cases}
%\begin{cases}
%\displaystyle (-1)^{M\ell} \frac{1-(-1)^\ell}{\pi\ell}+(-1)^{(M+1)\ell} I_\epsilon(\ell)- (-1)^{M\ell} I_{M\epsilon}(\ell) & \mbox{if $\ell \neq 0$}\\
%0 & \mbox{otherwise.}
%\end{cases}
\label{eq_inter_meyerM}
\end{multline}
% \begin{multline}
% \gamma_{\psi_1,\psi_1^\HH}(q) = \frac{\cos(2 \pi \epsilon
% q)-(-1)^q
% \cos(\pi \epsilon q)}{\pi q}+2(-1)^{q+1}\epsilon \int_{0}^{1}
% \gamma^2(\theta)  \; \sin(\pi\epsilon(1-2\theta) q)\; d\theta\\ +\; 4\epsilon \int_{0}^{1}\gamma^2(\theta)
% \sin(2\pi\epsilon(1-2\theta)q)d\theta.
% \label{eq_inter_meyer1}
% \end{multline}
Similarly to Shannon wavelets, for $(m,m') \in \{1,\ldots,M-2\}^2$, \eqref{eq:symShanMey} holds and
$\gamma_{\psi_m,\psi_m^\HH}(\ell) = 0$, when $\ell$ is odd.
As expected, we observe that the previous  cross-correlations converge point-wise to the expressions
given for Shannon wavelets  in \eqref{eq_intercorr_shannon0}
and \eqref{eq_intercorr_shannon}, as we let $\epsilon \to 0$.

Besides, let us make the following assumption: $W^2$ is $2q+2$ times continuously differentiable on $[0,1]$ with $q\in \NN^*$
and, for all $\ell\in \{0,\ldots,2q-1\}$,
$(W^2)^{(\ell)}(1) = 0$. This assumption is typically satisfied by the window defined by \eqref{eq:Meyerexample} with $q=4$. From \eqref{eq:propwindow}, it can be further noticed that, for all $\ell \in \{1,\ldots,q+1\}$, $(W^2)^{(2\ell)}(1/2) = 0$.
Then, when $x\neq 0$, it is readily checked by integrating by part that
\begin{multline}
\int_0^{1}W^2\Big(\frac{1+\theta}{2}\Big)\,
\sin{(\pi\epsilon x \theta )}\,d\theta
= \frac{1}{2\pi \epsilon x} \\+\frac{(-1)^{q-1}(W^2)^{(2q)}(1)}{2^{2q}(\pi \epsilon x)^{2q+1}}
\cos(\pi \epsilon x)\\+ \frac{(-1)^{q}(W^2)^{(2q+1)}(1)}{2^{2q+1}(\pi \epsilon x)^{2q+2}}
\sin(\pi \epsilon x)\\+ \frac{(-1)^{q+1}}{2^{2q+2}(\pi \epsilon x)^{2q+2}}
\int_0^{1}(W^2)^{(2q+2)}\Big(\frac{1+\theta}{2}\Big)\,
\sin{(\pi\epsilon x \theta)}\,d\theta.
\nonumber
\end{multline}
This shows that, as $|x| \to \infty$,
\begin{equation}
I_\epsilon(x) = \frac{1}{\pi x} +
\frac{(-1)^{q-1}(W^2)^{(2q)}(1)}{2^{2q-1}\pi^{2q+1} \epsilon^{2q} x^{2q+1} }
\cos(\pi \epsilon x) + O(x^{-2q-2}).
\label{eq:expandIepsilon}
\end{equation}
For example, for the taper function defined by \eqref{eq:Meyerexample}, we get
\begin{equation}
I_\epsilon(x) = \frac{1}{\pi x} -\frac{385875}{4\pi^7 \epsilon^8 x^9}
\cos(\pi \epsilon x) + O(x^{-10}).
\nonumber
\end{equation}
Combining \eqref{eq:expandIepsilon} with \eqref{eq_inter_meyer0}, \eqref{eq_inter_meyerm}
and \eqref{eq_inter_meyerM} allows us to see that
the cross-correlation sequences decay as $\ell^{-2q-1}$ when $|\ell| \to \infty$.
Eq. \eqref{eq:expandIepsilon} also indicates that the decay tends to be faster when
$\epsilon$ is large, which is consistent with intuition since the basis functions are then better localized in time.
Note that, as shown by \eqref{eq:Meyerm} and \eqref{eq:MeyerM}, under the considered differentiability assumptions, $|\widehat{\psi}_m|^2$ is $2q-1$ times continuously
differentiable on $\RR$
whereas $\widehat{\psi}_m(\omega) = 0$ for $m\in \NMs$ and
$|\omega| < m-\epsilon$. Prop. \ref{prop:decaycor} then guarantees a decay rate at least equal to $|\ell|^{-2q+1}$ (here, $N_m = q-1$).
In this case, we see that the decay rate derived from \eqref{eq:expandIepsilon} is more acurate than the decay given by Prop. \ref{prop:decaycor}.

\subsection{Wavelet families derived from wavelet packets}
\subsubsection{General form}
One can generate $M$-band orthonormal wavelet bases from dyadic
orthonormal wavelet packet decompositions corresponding to an
equal subband analysis. We are consequently limited to scaling
factors $M$ which are power of $2$. More precisely, let
$(\psi_m)_{m\in \NN}$ be the considered wavelet packets
\cite{Coifman_R_1992_tit_ent_babbs}, for all $P \in \NN^*$ an
orthonormal $M$-band wavelet decomposition is obtained using the
basis functions $(\psi_m)_{0\leq m < M}$ with $M = 2^P$. In this
case, the basis functions satisfy the following two-scale
relations: for all $m\in \NN$,
\begin{align}
\sqrt{2} \widehat{\psi}_{2m}(2\omega) &= A_0(\omega) \widehat{\psi}_{m}(\omega)
\label{eq:packet}\\
\sqrt{2} \widehat{\psi}_{2m+1}(2\omega) &= A_1(\omega)\widehat{\psi}_{m}(\omega),
\label{eq:packetb}
\end{align}
where $A_0$ and $A_1$ are the frequency responses of the low-pass
and high-pass filters of the associated two-band para-unitary
synthesis filter bank. We can infer the following result.
\begin{prop} \label{prop:recurpacket}
For all $\tau \in \RR$ and $m\in \NN^*$, we have
\begin{align}
\gamma_{\psi_{2m},\psi_{2m}^\HH}(\tau)&=\gamma_{a_0}[0] \gamma_{\psi_m,\psi_m^\HH}(2\tau) \nonumber
\\ + \sum_{k=1}^\infty 
\gamma_{a_0}[k]\Big( & \gamma_{\psi_m,\psi_m^\HH}(2\tau+k)+\gamma_{\psi_m,\psi_m^\HH}(2\tau-k)\Big)
\label{eq:interpair}\\
\gamma_{\psi_{2m+1},\psi_{2m+1}^\HH}(\tau)&=\gamma_{a_1}[0] \gamma_{\psi_m,\psi_m^\HH}(2\tau)\nonumber
\\ +
\sum_{k=1}^\infty
\gamma_{a_1}[k]\Big( & \gamma_{\psi_m,\psi_m^\HH}(2\tau+k)+\gamma_{\psi_m,\psi_m^\HH}(2\tau-k)\Big),
\label{eq:interimpair}
\end{align}
where, for all $\epsilon \in \{0,1\}$, $(\gamma_{a_\epsilon}[k])_{k\in
\ZZ}$ is the autocorrelation of the impulse response $(a_\epsilon[k])_{k\in \ZZ}$
of the filter with frequency response $A_\epsilon$:
\begin{equation}
\forall k \in \ZZ,\qquad \gamma_{a_\epsilon}[k] = \sum_{q=-\infty}^\infty a_\epsilon[q]\,a_\epsilon[q-k].
\nonumber
\end{equation}
\end{prop}
\begin{proof}
See Appendix \ref{ap:recurpacket}.
\end{proof}
It is important to note that \eqref{eq:interpair}
and \eqref{eq:interimpair} are not valid for $m=0$.
These two relations  define recursive equations for the
calculation of the cross-correlations
$(\gamma_{\psi_{m},\psi_{m}^\HH})_{m> 1}$, provided that
$\gamma_{\psi_{1},\psi_{1}^\HH}$ has been calculated first.

For this specific class of $M$-band wavelet decompositions, it is possible to relate the decay properties of the cross-correlation functions to the number of vanishing moments of the underlying dyadic wavelet analysis.
\begin{prop} \label{prop:binasympt}
Assume that the filters with frequency response $A_0$ and $A_1$ are FIR and $A_1$ has a zero of order $N\in \NN^*$ at frequency 0 (or, equivalently, $A_0$ has a zero of order $N$ at frequency $1/2$). Then,
 there exists $C_0 \in \RR_+$ such that
\begin{equation}
\forall \tau \in \RR^*,\qquad
|\gamma_{\psi_{0},\psi_{0}^\HH}(\tau)|
\leq C_0 |\tau|^{-2N-1}.
\label{eq:boundcorwavpack0}
\end{equation}
In addition, for all $m\in \NN^*$, let
$(\epsilon_1,\epsilon_2,\ldots,\epsilon_r)\in \{0,1\}^r$, $r\in \NN^*$, be the digits in the binary representation of $m$, that is
\begin{equation}
m = \sum_{i=1}^r \epsilon_i 2^{i-1}.
\label{eq:binary}
\end{equation}
Then, there exists $C_m \in \RR_+$ such that
\begin{equation}
\forall \tau \in \RR^*,\qquad
|\gamma_{\psi_{m},\psi_{m}^\HH}(\tau)|
\leq C_m|\tau|^{-2N(\sum_{i=1}^r \epsilon_i)-1}.
\label{eq:boundcorwavpack}
\end{equation}
\end{prop}
\begin{proof}
The filters of the underlying dyadic multiresolution being FIR (Finite Impulse Response), the wavelet packets are compactly supported. Consequently, their Fourier transforms are infinitely differentiable, their derivatives of any order belonging to $\LL^2(\RR)$.
% \forall k t^k \psi_m(t) \in \LL^2([-A,A]) \subset \LL^1([-A,A])$
In addition, the binary representation of $m \in \NN^*$ being given by \eqref{eq:binary}, Eqs.~\eqref{eq:packet} and
\eqref{eq:packetb} yield
\begin{equation}
\widehat{\psi}_{m}(\omega) =
\widehat{\psi}_{0}\Big(\frac{\omega}{2^P}\Big)\prod_{i=1}^r
\left(\frac{1}{\sqrt{2}}\,
A_{\epsilon_i}\Big(\frac{\omega}{2^i}\Big)\right)\prod_{i=r+1}^P
\left(\frac{1}{\sqrt{2}}\,
A_{0}\Big(\frac{\omega}{2^i}\Big)\right)
\nonumber
\end{equation}
that is $H_m(\omega) = \prod_{i=1}^P
A_{\epsilon_i}(2^{P-i}\omega)$. Moreover, by assumption
$A_1(\omega) = O(\omega^N)$ as $\omega \to 0$, whereas $A_0(0) =
\sqrt{2}$ and $|\widehat{\psi}_{0}(0)|=1$. This shows that, when
$m \neq 0$, $\widehat{\psi}_{m}(\omega) = O(\omega^{N(\sum_{i=1}^r
\epsilon_i)})$ as $\omega \to 0$. From \eqref{eq:decayasymptcor2},
we deduce the upper bound in \eqref{eq:boundcorwavpack}.
Furthermore, by applying Prop. \ref{prop:decaycor} when $M=2$, we
have then $N_0 = N_1 = N$ and  \eqref{eq:boundcorwavpack0} is
obtained.
\end{proof}
We see that the cross-correlation
$\gamma_{\psi_{m},\psi_{m}^\HH}$ decays all the more rapidly as the number of 1's  in the binary representation of $m$ is large.\footnote{The characterization of the sum of digits of integers remains an open problem in number theory\cite{Drmota_M_2005_jlms_sum_dfs,Allouche_J_2007_jnt_sum_sdcbd}.}

\subsubsection{The particular case of Walsh-Hadamard transform}
 The case $M=2$ corresponds to Haar wavelets. In contrast with  Shannon wavelets, these wavelets lay emphasis on
time/spatial localization. We consequently have:
\begin{align}
  \widehat{\psi}_0(\omega)& = \mathrm{sinc}(\frac{\omega}{2})\; e^{-\imath\frac{w}{2}} \label{eq_ond_haar0}\\
\quad \widehat{\psi}_1(\omega)& =\imath \;
\mathrm{sinc}(\frac{\omega}{4}) \; \sin(\frac{\omega}{4}) \;
e^{-\imath \frac{w}{2}},
\label{eq_ond_haar1}
\end{align}
where
\begin{equation}
\mathrm{sinc}(\omega) =
\begin{cases}
\displaystyle \frac{\sin(\omega)}{\omega} & \mbox{if $\omega \neq 0$}\\
1 & \mbox{otherwise.}
\end{cases}
\nonumber
\end{equation}
After some calculations which are provided in Appendix \ref{ap:haar},
we obtain for all $\tau \in \RR$,
\begin{multline} \pi
\gamma_{\psi_0,\psi_0^\HH}(\tau)=\sum_{k=0}^\infty (-1)^k 
\Big(\frac{1}{2}\; S_k(3+2d+2\tau)\\-S_k(1+2d+2\tau) +\frac{1}{2}\;
S_k(-1+2d+2\tau) \Big),
\label{eq_inter_haar0}
\end{multline}
where, for all $k\in \mathbb{N}$ and for all $x\in \RR$,
\begin{equation}
S_k(x) = x\; \int_{k\pi x}^{(k+1)\pi x} \mathrm{sinc}(u)\;du.
\nonumber
\end{equation}
Furthermore, we have (adopting the  convention:
$0\ln(0)=0$):
\begin{multline}
\pi
\gamma_{\psi_1,\psi_1^\HH}(\tau)=6\,\tau\ln|\tau|\,+\,(\tau+1)\ln|\tau+1|\,+\,(\tau-1)\ln|\tau-1| \\
-\,4\;\Big(\tau+\frac{1}{2}\Big)\ln\Big|\tau+\frac{1}{2}\Big|\,-\,4\;\Big(\tau-\frac{1}{2}\Big)\ln\Big|\tau-\frac{1}{2}\Big|.
\label{eq_inter_haar1}
\end{multline}

For $M = 2^P$ with $P > 1$, the cross-correlations $\gamma_{\psi_{m},\psi_{m}^\HH}$, $m
\in \{2,\ldots,2^P-1\}$, can be determined in a recursive manner thanks
to Prop. \ref{prop:recurpacket}. For Walsh-Hadamard wavelets, we have
\begin{align}
\forall \epsilon \in\{0,1\},\;\forall k\in \ZZ,\;
\gamma_{a_\epsilon}[k] = \begin{cases}
1  & \mbox{if $k = 0$}\\
\displaystyle \frac{(-1)^\epsilon}{2} & \mbox{if $|k| = 1$}\\
0 & \mbox{otherwise}
\end{cases}
\end{align}
and, consequently,
for all $m\neq 0$ and $\tau\in \RR$,
%\begin{align}
%\gamma_{\psi_{2m},\psi_{2m}^\HH}(\tau)&=
%\gamma_{\psi_m,\psi_m^\HH}(2\tau) \nonumber \\
%& +\frac{1}{2}\Big(\gamma_{\psi_m,\psi_m^\HH}(2\tau+1)+\gamma_{\psi_m,\psi_m^\HH}(2\tau-1)\Big)
%\label{eq:haarpair}\\
%\gamma_{\psi_{2m+1},\psi_{2m+1}^\HH}(\tau)&=
%\gamma_{\psi_m,\psi_m^\HH}(2\tau)  \nonumber \\ & -\frac{1}{2}\Big(\gamma_{\psi_m,\psi_m^\HH}(2\tau+1)+\gamma_{\psi_m,\psi_m^\HH}(2\tau-1)\Big).
%\label{eq:haarimpair}
%\end{align}
\begin{multline}
\gamma_{\psi_{2m},\psi_{2m}^\HH}(\tau)= \gamma_{\psi_m,\psi_m^\HH}(2\tau) \\
+\frac{1}{2}\Big(\gamma_{\psi_m,\psi_m^\HH}(2\tau+1)+\gamma_{\psi_m,\psi_m^\HH}(2\tau-1)\Big)
\label{eq:haarpair}
\end{multline}
\begin{multline}
\gamma_{\psi_{2m+1},\psi_{2m+1}^\HH}(\tau)= \gamma_{\psi_m,\psi_m^\HH}(2\tau)\\  -\frac{1}{2}\Big(\gamma_{\psi_m,\psi_m^\HH}(2\tau+1)+\gamma_{\psi_m,\psi_m^\HH}(2\tau-1)\Big).
\label{eq:haarimpair}
\end{multline}
From \eqref{eq_inter_haar1}, it can be noticed that
$\gamma_{\psi_1,\psi_1^\HH}(\tau)=1/(8\pi \tau^3)+ O(\tau^{-5})$
when $|\tau|>2$, which corresponds to a faster asymptotic decay
than with Shannon wavelets. The asymptotic behaviour of
$\gamma_{\psi_{m},\psi_{m}^\HH}(\tau)$, $m>2$, can also be deduced
from \eqref{eq_inter_haar1}, \eqref{eq:haarpair} and
\eqref{eq:haarimpair}. The expressions given in Table
\ref{tab:asympthaar} are in perfect agreement with the decay rates
predicted by Prop. \ref{prop:binasympt}.

\subsection{Franklin wavelets}
Franklin wavelets
\cite{Franklin_1928_math-annalen_set_cof,Stromberg_J_1983_cha_mfsohossrubhs}  correspond to a
dyadic orthonormal basis of spline wavelets of order 1 \cite[p.
146 sq.]{Daubechies_I_1992_book_ten_lw}. With the Haar wavelet,
they form a special case of Battle-Lemari{\'e} wavelets
\cite{Battle_G_1987_cmp_blo_scolf,Lemarie_P_1988_nou_bolr}. The
Fourier transforms of the scaling function and the mother wavelet
are given by:
\begin{align}
\widehat{\psi}_0(\omega)&=\left(\frac{3}{1+2 \cos^2(\omega/2)}\right)^{1/2}\,\mathrm{sinc}^2{\Big(\frac{\omega}{2}\Big)}\\
\widehat{\psi}_1(\omega)&=-\left(\frac{3 (1+2 \sin^2(\omega/4))}{\big(1+2 \cos^2(\omega/2)\big) \big(1+2 \cos^2(\omega/4)\big)}\right)^{1/2} \nonumber \\
& \times \sin^2\Big(\frac{\omega}{4}\Big)\,\mathrm{sinc}^2{\Big(\frac{\omega}{4}\Big)}
\exp{(-\imath\frac{\omega}{2})}.
\label{eq:psifranklin}
\end{align}
The expression of the cross-correlation of the scaling functions readily
follows from \eqref{eq:gammaphi}:
\begin{equation}
\forall \tau\in\RR,\qquad
\gamma_{\psi_0,\psi_0^\HH}(\tau) = \frac{6}{\pi}  \sum_{k=0}^\infty (-1)^kT_k(1+2d+2\tau),
\nonumber
\end{equation}
where, for all $k\in \NN$ and $x\in\RR$,
\begin{equation}
T_k(x) = \int_{k\pi}^{(k+1)\pi}
\frac{\mathrm{sinc}^4(u)}{1+2 \cos^2(u)}\,\cos(ux)\,du.
\nonumber
\end{equation}
The expression of the cross-correlation of the mother wavelet can
be deduced from \eqref{eq:gammapsi} and \eqref{eq:psifranklin} and
resorting to numerical methods for the computation of the
resulting integral, but it is also possible to obtain a series
expansion of the cross-correlation as shown next.

Taking the square modulus of \eqref{eq:psifranklin}, we find
\begin{equation}
2|\widehat{\psi}_1(2\omega)|^2 = |\widetilde{A}_1(\omega)|^2 \, |\widehat{\chi}(\omega)|^2,
\label{eq:Franklinwavscal}
\end{equation}
where
\begin{align}
& \widetilde{A}_1(\omega)=\left(\frac{6 \big(2- \cos(\omega)\big)}{\big(1+2
\cos^2(\omega)\big) \big(2+\cos(\omega)\big)}\right)^{1/2}, \nonumber \\
& \widehat{\chi}(\omega)=
\Big(\frac{\sin^2{(\omega/2)}}{\omega/2} \Big) ^2\,.
\nonumber
\end{align}
Let $(\widetilde{a}_1[k])_{k\in \ZZ}$ (resp. $\chi$) be the sequence
(resp. function) whose Fourier transform is $\widetilde{A}_1$
(resp. $\widehat{\chi}$).
Similarly to \eqref{eq:interimpair}, \eqref{eq:Franklinwavscal} leads to the following relation
\begin{multline}
\forall \tau \in \RR,\qquad
\gamma_{\psi_1,\psi_1^\HH}(\tau)=\gamma_{\widetilde{a}_1}[0]\,\gamma_{\chi,\chi^\HH}(2\tau)\\+\sum_{k=1}^\infty
\gamma_{\widetilde{a}_1}[k]
\Big(\gamma_{\chi,\chi^\HH}(2\tau+k)+\gamma_{\chi,\chi^\HH}(2\tau-k)\Big),
\label{eq:relpsi1}
\end{multline}
where $(\gamma_{\widetilde{a}_1}[k])_{k\in \ZZ}$ denotes the autocorrelation of the sequence $(\widetilde{a}_1[k])_{k\in \ZZ}$.

We have then to determine $\gamma_{\chi,\chi^\HH}$ and $(\gamma_{\tilde{a}_1}[k])_{k\in \NN}$.
First, it can be shown (see Appendix
\ref{ap:splines} for more detail) that
\begin{multline}
3\pi \gamma_{\chi,\chi^\HH}(\tau)=
q_0 \tau^3 \ln|\tau| \\ + \sum_{p=1}^4 q_p \big( (\tau+p)^3 \ln|\tau+p| + (\tau-p)^3 \ln|\tau-p| \big),
\label{eq:gammatildpsi1}
\end{multline}
where
\begin{equation}
q_0 = -\frac{35}{16},\quad
q_1 = \frac{7}{4},\quad
q_2 = -\frac{7}{8},\quad
q_3 = \frac{1}{4},\quad
q_4 = -\frac{1}{32}.
\nonumber
\end{equation}
% \begin{align}
% Q_0(\tau) & = -\frac{35}{16} \tau^2\\
% Q_1(\tau) & = \frac{7}{4}(\tau+1)^2\\
% Q_2(\tau) & = -\frac{7}{8}(\tau+2)^2\\
% Q_3(\tau) & = \frac{1}{4}(\tau+3)^2\\
% Q_4(\tau) &= -\frac{1}{32}(\tau+4)^2
% \end{align}

Secondly, the sequence $(\gamma_{\tilde{a}_1}[k])_{k\in \NN}$ can be deduced
from $|\widetilde{A}_1(\omega)|^2$ by using $z$-transform inversion techniques
 (calculations are provided in Appendix \ref{ap:splines}). This leads to $\forall k \in \NN,\qquad$
\begin{equation}
\begin{cases}
\displaystyle \gamma_{\tilde{a}_1}[2k] =
\frac{2\sqrt{3}}{9}(2-\sqrt{3})^{k}\big(7(-1)^{k}+4(2-\sqrt{3})^{k}\big)\\
\displaystyle \gamma_{\tilde{a}_1}[2k+1] = \frac{8\sqrt{3}}{9}(2-\sqrt{3})^{k}
 \big((-1)^{k}(1-\sqrt{3})-(2-\sqrt{3})^{k+1}\big).
\end{cases}
\label{eq:expgammaat1}
\end{equation}
Equations \eqref{eq:relpsi1}, \eqref{eq:gammatildpsi1} and \eqref{eq:expgammaat1} thus allow an accurate numerical evaluation
of $\gamma_{\psi_1, \psi_1^\HH}$. Since
\begin{equation}
\gamma_{\chi,\chi^\HH}(\tau)\sim -3/(2\pi\tau^5) \qquad \mbox{as $|\tau|\to \infty$}
\label{eq:asymptgammachi}
\end{equation}
and
\begin{equation}
\gamma_{\tilde{a}_1}[k] = O((2-\sqrt{3})^{k/2})\qquad \mbox{as $k\to \infty$}
\label{eq:asymptgammaat1}
\end{equation}
the convergence of the series in  \eqref{eq:relpsi1} is indeed pretty fast.

From Prop. \ref{prop:decaycor}, we further deduce that
$\gamma_{\psi_0, \psi_0^\HH}(\tau)$ and $\gamma_{\psi_1,
\psi_1^\HH}(\tau)$ decay as $|\tau|^{-5}$ (here, we have $N_0 =
N_1 = 2$). The decay rate of $\gamma_{\psi_1, \psi_1^\HH}$ can be
derived more precisely from \eqref{eq:relpsi1}. Indeed, we have
\begin{align}
& |\tau|^5 \sum_{k=-\infty}^\infty
|\gamma_{\widetilde{a}_1}[k]| |\gamma_{\chi,\chi^\HH}(2\tau-k)| \nonumber \\
& \leq \frac{1}{2} \sum_{k=-\infty}^\infty
|\gamma_{\widetilde{a}_1}[k]| \big(|2\tau-k|^5+|k|^5\big)  |\gamma_{\chi,\chi^\HH}(2\tau-k)|\nonumber\\
& \leq \Big(\sup_{u\in \RR} (|u|^5 |\gamma_{\chi,\chi^\HH}(u)|)+
\sup_{u\in \RR} |\gamma_{\chi,\chi^\HH}(u)|\Big)
\sum_{k=-\infty}^\infty (1+|k|^5)|\gamma_{\widetilde{a}_1}[k]| \nonumber \\
& < \infty,
\end{align}
where the convexity of $|.|^5$ has been used in the first inequality
and the last inequality is a consequence of \eqref{eq:asymptgammachi}
and \eqref{eq:asymptgammaat1}. It can be deduced from the dominated convergence theorem that
\begin{align}
\lim_{|\tau| \to \infty} \tau^5 \gamma_{\psi_1,\psi_1^\HH}(\tau)
= &\sum_{k=-\infty}^\infty
\gamma_{\widetilde{a}_1}[k] \lim_{|\tau| \to \infty}\tau^5 \gamma_{\chi,\chi^\HH}(2\tau-k)\nonumber\\
= & -\frac{3}{64\pi} \sum_{k=-\infty}^\infty \gamma_{\widetilde{a}_1}[k]
= -\frac{3}{64\pi} |\widetilde{A}_1(0)|^2 \nonumber \\
= & -\frac{1}{32\pi}.
\nonumber
\end{align}
Finally, we would like to note that similar expressions can be derived for higher order spline wavelets although the calculations become tedious.

\section{Experimental results}
\label{sec:applinum}

\subsection{Results based on theoretical expressions}
At first, we provide numerical evaluations of the expressions
of the cross-correlation sequences obtained in the previous section when the lag variable (denoted by $\ell$) varies in $\{0,1,2,3\}$. The cross-correlations for lag values in $\{-3,-2,-1\}$ can be deduced from the symmetry properties shown in Section \ref{se:2ndexp1}. We notice that
cubic spline wavelets \cite{Unser_M_1999_spm_spl_pfsp} have not been studied in Section \ref{se:examplewav}, so that the their cross-correlation values have to be computed directly from  (\ref{eq:gammapsi}) and (\ref{eq:gammaphi}).
The results concerning the dyadic case are given in
Table~\ref{tab:simulsdya}. They show
that the cross-correlations between the noise coefficients at the output of a
dual-tree analysis can take significant values (up to 0.64). We also observe that the wavelet choice has a clear influence on the magnitude
of the correlations. Indeed, while Meyer wavelet leads to
results close to the Shannon wavelet, the correlations are weaker
for the Haar wavelet. As expected, spline wavelets yield intermediate cross-correlation values between the Meyer and the Haar cases. \\
Our next results concern the $M$-band case with $M \ge 3$.
Due to the properties of the cross-correlations, the study can be simplified as explained below.
%\begin{description}
\begin{itemize}
\item Shannon wavelets: due to \eqref{eq:symShanMey}, the $M$-band cross-correlations are, up to a possible sign change, equal to the dyadic case cross-correlations (see Table~\ref{tab:simulsdya}).
% for $m\in \{1,...,M-2\}$,
% $\gamma_{\psi_{m+1},\psi_{m+1}^\HH}(\ell)=(-1)^{\ell}\gamma_{\psi_{m},
% \psi_{m}^\HH}(\ell)$. That's why we decided to consider the
% $4$-band case and complementary results for a higher number of
% bands, are easy to deduce.
\item Meyer wavelets: still due to \eqref{eq:symShanMey},
the first $M-2$ cross-correlations of the wavelets are easily deduced from the first one. So, we only need to specify
$\gamma_{\psi_0,\psi_0^\HH}$,
$\gamma_{\psi_1, \psi_1^\HH}$ and
$\gamma_{\psi_{M-1},\psi_{M-1}^\HH}$.
Tables \ref{tab:simulsM1} and \ref{tab:simulsM2} give the related values when $M$ ranges from 3 to 8, the $\epsilon$ parameter being set to its  possible maximum value $(M+1)^{-1}$.
\item Walsh-Hadamard wavelets: when $M=2^{P+1}$, $P\in \NN^*$, $(\psi_m)_{0\le m < M/2}$ is the set of basis functions of the $(M/2)$-band wavelet decomposition. In this way, the results in Table \ref{tab:simulsM3} allow us to evaluate the cross-correlation values for $M \in \{2,4,8\}$.
\end{itemize}
%\end{description}
As shown in Tables \ref{tab:simulsM1} and \ref{tab:simulsM2},
the cross-correlations in the Meyer case remain significant, their magnitudes being even slightly increased as the number of subbands becomes larger. Table \ref{tab:simulsM3} shows that the cross-correlation of Walsh-Hadamard wavelets are much smaller and that they are close to zero when the subband index $m$ is large.
%, **which is consistent with their short support**.

\subsection{Monte Carlo simulations}

A second approach for computing the cross-correlations consists in carrying out a Monte Carlo study. More precisely, a realization of a white standard Gaussian noise  sequence of length
$L=M^J\lfloor\frac{2^{14}}{M^J}\rfloor $ (with $J=3$)
is drawn and its $1D$ dual-tree decomposition
over $J$ resolution levels is performed. Then, the cross-covariances for each subband can be estimated by their classical sample estimates. In our experiments, average values of these cross-correlations are computed over 100 runs.\\
This Monte Carlo study allows us to validate the theoretical expressions we have obtained for several wavelet families in Section \ref{se:examplewav}. In addition,  this approach  can be applied to wavelets whose Fourier transforms do not take a simple form. For instance, we are able to compute the cross-correlation values for symlets \cite{Daubechies_I_1992_book_ten_lw}[p.259] associated to filters of length 8 as well as for $4$-band compactly supported wavelets (here designated as AC) associated to 16-tap filters \cite{Alkin_O_1995_tsp_des_embclpprp}.\\
Table \ref{tab:numdya} shows the estimations of the cross-correlations obtained in the dyadic case, while the results in the $M$-band case with $M\ge 3$ are listed in Tables \ref{tab:numM1} and \ref{tab:numM2}. By comparing these results with the ones in Tables \ref{tab:simulsM3}, \ref{tab:simulsM1} and
\ref{tab:simulsM2}, a good agreement is observed between the theoretical values and the estimated ones for Shannon, Meyer and cubic spline wavelets. For less regular wavelets such as Franklin or Haar wavelets, the agreement remains quite good at coarse resolution ($j=3$) but, at fine resolution ($j=1$), it appears that the correlations are stronger in practice than predicted by the theory. The fact that we use a discrete decomposition instead of the classical analog wavelet framework may account for these differences. Indeed, we
use the implementation of the $M$-band dual-tree decomposition described in \cite{Chaux_C_2006_tip_ima_adtmbwt}, which requires some digital prefilters.  The selectivity of these filters is inherited from the frequency selectivity of the scaling function.
As a side effect, the noise is colored by these prefilters.\\
Some comments can also be made concerning symlets 8
and 4-band AC wavelets. We see that the symlets behave very similarly to Franklin wavelets whereas AC wavelets provide intermediate correlation magnitudes between the $M$-band Meyer and Hadamard cases.

\subsection{Inter-band cross-correlations} \label{se:interbandsimuls}
Although the cross-correlations between primal/dual basis functions corresponding to different subbands have not been much investigated in the previous sections, we provide in this part some  numerical evaluations for them.\\
More precisely, we are interested in studying
$(\gamma_{\psi_m,\psi_{m'}^\HH}(\ell))_{\ell\in\ZZ}$ with $m\neq
m'$, which represents the inter-band cross-correlations. We are
able to compute them thanks to \eqref{eq:crosspsimpsimp} and
\eqref{eq:crosspsimpsi0}.
Numerical results are given in Table \ref{tab:simulsinterb}.\\
%; we verify that Eq. \eqref{eq:boundcrosspsimpsimp} is always verified.\\
Some symmetry properties can be observed, which can be deduced from  \eqref{eq:crosspsimpsimp}, \eqref{eq:crosspsimpsi0} and the specific form of the considered wavelet functions.
Most interestingly, it can be noticed that the inter-band cross-correlations often have a significantly smaller amplitude than the
corresponding intra-band cross-correlations.
%is often unsignificant with respect to $\gamma_{\psi_m,\psi_{m}^\HH}(\ell)$ values;
As expected,
the more frequency-selective the decomposition filters,
the more negligible the values of the inter-band cross-correlations.

\subsection{Two-dimensional experiment}

We aim here at comparing the obtained theoretical expressions of
the two-dimensional cross-covariances with Monte Carlo evaluations
of these second-order statistics. We consider a two-dimensional
3-band Meyer dual-tree wavelet decomposition of a white standard
Gaussian field of size $756\times 756$. The Monte Carlo study is
carried out over $10000$ realizations. The decomposition is
performed over $J=2$ resolution levels and the results are
provided at the coarsest resolution. The covariance fields are
depicted in Fig. \ref{fig:inter2D} as well as the ones derived
from \eqref{eq:internoise},
\eqref{eq_inter_meyer0}-\eqref{eq_inter_meyerM}. For more
readibility, a dashed separation line between the subbands has
been added (for a $3$-band decomposition, $9$ covariance fields
$(\Gamma_{n_{j,\mm},n_{j,\mm}^\HH}[\eell])_{\eell\in\ZZ^2}$ have
to be computed when $\mm \in \{0,1,2\}^2$). We compute these
fields for $\eell \in \{0,1,2,3\}^2$, thus resulting in $16$
covariance values for each subband. Succinctly, each small
gray-scaled square represents the intensity of the
cross-covariance in a given subband $\mathbf{m}$ at spatial
position $\eell$. Comparing theoretical results with numerical
ones (left and right sides of Fig. \ref{fig:inter2D},
respectively), it can be noticed that they are quite similar. In
addition, we observe that, due to the separability of the
covariance fields and \eqref{eq:internoise0}, for all $\mm =
(m_1,m_2)$ and $\eell = (\ell_1,\ell_2)$,
$(\Gamma_{n_{j,\mm},n_{j,\mm}^\HH}[\eell])_{\eell\in\ZZ^2}$
vanishes when either ($m_1 \neq 0$ and $\ell_1 = 0$) or ($m_2 \neq
0$ and $\ell_2 = 0$).

\section{Conclusion}\label{se:conclu}
In this paper, we have investigated the covariance properties of the $M$-band dual-tree wavelet coefficients of wide-sense stationary 1D and 2D random processes.
We have stated a number of results helping to better understand the structure of the correlations introduced by this frame decomposition. These results may be useful in the design of efficient denoising rules using dual-tree wavelet decompositions, when the noise is additive and stationary. In particular, if a pointwise estimator is applied to the pair of primal/dual coefficients at the same location and in the same subband, we have seen that the related components of the noise are uncorrelated. On the contrary, if a block-based estimator is used to take advantage of some spatial neighborhood  of the primal and dual coefficients around some given position
in a subband, noise correlations generally must be taken into account. Recently, this fact has been exploited in the design of an efficient image denoising method using Stein's principle, yielding state-of-the-art performance for multichannel image denoising \cite{Chaux_C_2006_icassp_new_eid2ddtmbd,Chaux_C_2007_tsp_non_sbemid}.
In future work, it would be interesting to extend our analysis to
other classes of random processes. In particular, a similar study
could be undertaken for self-similar processes
\cite{Abry_P_1995_wspp,Wornell_G_1992_tit_wav_brcsssafm} and
processes with stationary increments
\cite{Krim_H_1995_tit_mul_acnp,Averkamp_R_tit_1998_som_dpcwtrp}.

Finally, we would like to note that the expressions of the cross-correlations between the primal and dual wavelets which have been derived in this paper may be of interest for other problems.  Indeed, let
\[
T = \begin{bmatrix}
D\\D^\HH
\end{bmatrix}
\]
denote the dual-tree wavelet decomposition where
$D$ (resp. $D^\HH$) is the primal (resp. dual) wavelet decomposition. The studied cross-correlations then characterize the ``off-diagonal'' terms of the operator
\[
T T^* = \begin{bmatrix}
I & D (D^\HH)^*\\
D^\HH D^* & I
\end{bmatrix},
\]
where $A^*$ denotes the adjoint of a bounded linear operator $A$.
The operator $T T^*$ is encountered in the solution of some inverse problems.

\begin{appendices}

\section{Proof of Proposition \ref{prop:corrgen}} \label{ap:corrgen}
The $M$-band wavelet coefficients of the noise are given by
\begin{align}
\forall m \in \NM,&\forall k\in \ZZ, \qquad \nonumber \\ &n_{j,m}[k] =
\int_{-\infty}^\infty n(x)\;
\frac{1}{M^{j/2}}\psi_m(\frac{x}{M^j}-k)\,dx \nonumber\\
&n_{j,m}^\HH[k] = \int_{-\infty}^\infty n(x)\;
\frac{1}{M^{j/2}}\psi_m^\HH(\frac{x}{M^j}-k)\,dx. \nonumber
\end{align}
For all $(m,m')\in \NMSQ$ and $(k,k') \in \ZZ^2$, we have then
\begin{multline}
\E\{n_{j,m}[k]n_{j,m'}[k']\} =
\int_{-\infty}^\infty\int_{-\infty}^\infty
\E\{n(x)n(x')\} \\ \times \frac{1}{M^{j/2}}\psi_m(\frac{x}{M^j}-k)
\frac{1}{M^{j/2}}\psi_{m'}(\frac{x'}{M^j}-k')dx\,dx'\,.
\nonumber
\end{multline}
After the variable change $\tau = x-x'$, using the definition
of the autocovariance of the noise in \eqref{eq:covb}, we find
that
\begin{multline}
\E\{n_{j,m}[k]n_{j,m'}[k']\} =
\int_{-\infty}^\infty
\Gamma_n(\tau) \\ \Big(\int_{-\infty}^\infty\frac{1}{M^{j/2}}\psi_m(\frac{x}{M^j}-k)
\frac{1}{M^{j/2}}\psi_{m'}(\frac{x-\tau}{M^j}-k')dx\Big)d\tau
\nonumber
\end{multline}
which readily yields
\begin{equation}
\E\{n_{j,m}[k]n_{j,m'}[k']\} =
\int_{-\infty}^\infty
\Gamma_n(\tau) \gamma_{\psi_m,\psi_{m'}}(\frac{\tau}{M^j}+k'-k)\,d\tau.
\nonumber
\end{equation}
Note that, in the above derivations, permutations of the integral symbols/expectation have been performed. For these operations to be valid, some technical conditions are required.
For example, Fubini's theorem \cite[p. 164]{Rudin_W_1987_book_rea_ca} can be invoked provided that
\begin{equation}
\int_{-\infty}^\infty
\Gamma_{|n|}(\tau) \gamma_{|\psi_m|,|\psi_{m'}|}(\frac{\tau}{M^j}+k'-k)\,d\tau < \infty,
\nonumber
\end{equation}
where $\Gamma_{|n|}$ is the autocovariance of $|n|$.\\
Relations \eqref{eq:bhbh1D} and
\eqref{eq:bbh1D} follow from similar arguments.

\section{Proof of Proposition \ref{prop:sym}}\label{ap:sym}
For all $(m,m')\in \NM^2$, $\forall \tau \in \RR,\qquad$
\begin{equation}
\frac{1}{2\pi}\int_{-\infty}^\infty \widehat{\psi}_{m}^\HH(\omega)\big(\widehat{\psi}_{m'}^\HH(\omega))^* e^{\imath \omega \tau}\,d\omega
= \gamma_{\psi_m^\HH,\psi_{m'}^\HH}(\tau).
\label{eq:FourierinvgammaH}
\nonumber
\end{equation}
Since the Fourier transform defines an isometry on $\LL^2(\RR)$,
 it can be deduced from \eqref{eq:FourierinvgammaH} that $\gamma_{\psi_m^\HH,\psi_{m'}^\HH}$ is in
$\LL^2(\RR)$ and its Fourier transform is
$\omega \mapsto \widehat{\psi}_{m}^\HH(\omega)\big(\widehat{\psi}_{m'}^\HH(\omega))^*$.
\footnote{\label{f:L2gamma} As $\{\psi_{m'}^\HH(t-k),k\in \ZZ\}$ is an orthonormal family
of $\LL^2(\RR)$, we have
$|\widehat{\psi}_{m'}^\HH(\omega)|\le 1
$
and $\widehat{\psi}_{m}^\HH\big(\widehat{\psi}_{m'}^\HH)^* \in \LL^2(\RR)$.}
According to \eqref{eq:Hilbertcond} and \eqref{eq:linkpsi0Hpsi0}, when $m=m'=0$ or
$mm'\neq 0$, the latter function is equal to
$\omega \mapsto \widehat{\psi}_{m}(\omega)
\big(\widehat{\psi}_{m'}(\omega))^*$, thus showing that
$\gamma_{\psi_m^\HH,\psi_{m'}^\HH} =
\gamma_{\psi_m,\psi_{m'}}$.
The equality of the covariance sequences defined by \eqref{eq:bb1D}
and \eqref{eq:bhbh1D} straightforwardly follows.

%{eq:symmmplm}
When $m m'\neq 0$,
the Fourier transform of $\gamma_{\psi_m,\psi_{m'}^\HH}$
is equal to
$\omega \mapsto \imath\,\mathrm{sign}(\omega)
\widehat{\psi}_{m}(\omega) \widehat{\psi}_{m'}^*(\omega)$
whose conjuguate is the Fourier transform
of $-\gamma_{\psi_{m'},\psi_m^\HH}$.
This proves \eqref{eq:symmmplm},
%$\widehat{\psi}_{m}(\omega)|^2$, which is an odd function.
%We deduce that $\gamma_{\psi_m,\psi_{m}^\HH}$ is an odd function too.
which combined with  \eqref{eq:bbh1D} leads to
\begin{multline}
\forall \ell \in \ZZ,\\
\Gamma_{n_{j,m},n_{j,m'}^\HH}[\ell] =
-\int_{-\infty}^\infty\Gamma_n(x)
\gamma_{\psi_m',\psi_{m}^\HH}\left(-\frac{x}{M^j}+\ell\right)dx.
\nonumber
\end{multline}
After a variable change and using the fact that $\Gamma_n$ is an even function, we obtain \eqref{eq:symmmplmc}.
% \begin{equation}
% \Gamma_{n_{j,m},n_{j,m}^\HH}(\ell)
% = -\Gamma_{n_{j,m},n_{j,m}^\HH}(-\ell).
% \end{equation}

Consider now the Fourier transform $\omega \mapsto
\widehat{\psi}_0(\omega)(\widehat{\psi}_0^\HH(\omega))^*$
of $\gamma_{\psi_0,\psi_0^\HH}$.
For all $\omega \geq 0$, there exists $k\in \NN$
such as $\omega \in [2k\pi,2(k+1)\pi)$ and, from
\eqref{eq:linkpsi0Hpsi0}, we get
\begin{align}
\widehat{\psi}_0(\omega)(\widehat{\psi}_0^\HH(\omega))^*
& =(-1)^k e^{\imath(d+\frac{1}{2})\omega}\;|\widehat{\psi}_0(\omega)|^2\nonumber \\
&= e^{\imath(2d+1)\omega}\widehat{\psi}_0(-\omega)(\widehat{\psi}_0^\HH(-\omega))^*.
\nonumber
\end{align}
For symmetry reasons, the equality between the first and last terms extends to all $\omega \in \RR$. Coming back to the time domain, we find
\begin{equation}
\forall \tau \in \RR,\qquad
\gamma_{\psi_0,\psi_0^\HH}(\tau) =
\gamma_{\psi_0,\psi_0^\HH}(-\tau-2d-1).
\nonumber
\end{equation}
This shows the symmetry of
$\gamma_{\psi_0,\psi_0^\HH}$ w.r.t. $-d-1/2$.
Eq. \eqref{eq:bbh1D} then yields
\begin{align}
\forall \ell \in \ZZ, \nonumber \\
 \Gamma_{n_{j,0},n_{j,0}^\HH}[\ell]&=
\int_{-\infty}^\infty \Gamma_n(x)
\gamma_{\psi_0,\psi_{0}^\HH}\left(-\frac{x}{M^j}+\ell-2d-1\right)dx\nonumber\\
& = \int_{-\infty}^\infty
\Gamma_n(x)
\gamma_{\psi_0,\psi_{0}^\HH}\left(\frac{x}{M^j}+\ell-2d-1\right)dx\nonumber\\
& = \Gamma_{n_{j,0},n_{j,0}^\HH}[-\ell+2d+1].
\nonumber
\end{align}

\section{White noise case}
\label{ap:whitenoise} Recall that a white noise is not a process
with finite variance, but a generalized random process
\cite{Guelfand_I_1962_book_dis,Yaglom_A_1987_book_cor_dsrfbr}. As
such, some caution must be taken in the application of
\eqref{eq:bb1D}-\eqref{eq:bbh1D}. More precisely, if $n$ is a
white noise, its autocovariance can be viewed as the limit as
$\epsilon >0$ tends to 0 of
\begin{equation}
\Gamma_{n^\epsilon}(\tau) = \frac{\sigma^2}{\sqrt{2\pi}\epsilon} \exp(-\frac{\tau^2}{2 \epsilon^2}),\qquad
 \tau \in \RR.
 \nonumber
\end{equation}
Formula \eqref{eq:bbh1D} can then be used, yielding for all $(m,m')\in \NM^2$ and $(j,\ell)\in \ZZ^2$,
\begin{multline}
\Gamma_{n_{j,m}^\epsilon,n_{j,m'}^{\epsilon\HH}}[\ell] = \\
\sigma^2\int_{-\infty}^\infty  \frac{1}{\sqrt{2\pi}}
\exp(-\frac{x^2}{2})\,
\gamma_{\psi_m,\psi_{m'}^\HH}\left(\frac{\epsilon
x}{M^j}-\ell\right)dx.
\nonumber
\end{multline}
Since $\psi_m$ and $\psi_{m'}^\HH$ are in $\LL^2(\RR)$, $\gamma_{\psi_m,\psi_{m'}^\HH}$ is a bounded continuous  function. By applying Lebesgue dominated convergence theorem, we deduce that
\begin{align}
&\Gamma_{n_{j,m},n_{j,m'}^\HH}[\ell] = \lim_{\epsilon\to 0}
\Gamma_{n_{j,m}^\epsilon,n_{j,m'}^{\epsilon\HH}}[\ell] \nonumber \\
& =\sigma^2 
\int_{-\infty}^\infty \frac{1}{\sqrt{2\pi}}
\exp(-\frac{x^2}{2})\,
\lim_{\epsilon \to 0}\gamma_{\psi_m,\psi_{m'}^\HH}\left(\frac{\epsilon x}{M^j}-\ell\right)dx\nonumber\\
 & = \sigma^2  \gamma_{\psi_m,\psi_{m'}^\HH} (-\ell) \int_{-\infty}^\infty  \frac{1}{\sqrt{2\pi}} \exp(-\frac{x^2}{2})dx
\nonumber
\end{align}
which leads to \eqref{eq:internoise1D}. Equations \eqref{eq:internoiseb1D} are similarly obtained by further noticing that, due to the orthonormality property, $\gamma_{\psi_m,\psi_{m'}}(-\ell) = \gamma_{\psi_m^\HH,\psi_{m'}^\HH}(-\ell) = \delta_{m-m'}\delta_\ell$.

\section{Proof of Proposition \ref{prop:posttransf}} \label{ap:posttransf}
From \eqref{eq:transfnoise1} and \eqref{eq:transfnoise2} defining
the unitary transform applied to the detail noise coefficients
$(n_{j,\mm}[\kk])_{\kk \in \ZZ^2}$ and $(n_{j,\mm}^\HH[\kk])_{\kk
\in \ZZ^2}$:
\begin{align}
\E\{w_{j,\mm}[\kk] & w_{j,\mm}[\kk']\} = \nonumber \\
& \frac{1}{2} \Big( \E\{n_{j,\mm}[\kk]n_{j,\mm}[\kk']\} + \E\{n_{j,\mm}[\kk]n_{j,\mm}^\HH[\kk']\} \nonumber \\ &+ \E\{n_{j,\mm}^\HH[\kk]n_{j,\mm}[\kk']\} +  \E\{n_{j,\mm}^\HH[\kk]n_{j,\mm}^\HH[\kk']\} \Big).
\nonumber
\end{align}
Using \eqref{eq:bbbhbh} and the evenness of $\Gamma_{n_{j,\mm},n^\HH_{j,\mm}}$, one can easily deduce \eqref{eq:2Drot1}.
Concerning \eqref{eq:2Drot2}, we proceed in the same way, taking into account the relation:
\begin{align}
\E\{w_{j,\mm}^\HH[\kk] & w_{j,\mm}^\HH[\kk']\} = \nonumber \\
& \frac{1}{2} \Big( \E\{n_{j,\mm}[\kk]n_{j,\mm}[\kk']\} - \E\{n_{j,\mm}[\kk]n_{j,\mm}^\HH[\kk']\} \nonumber \\ &- \E\{n_{j,\mm}^\HH[\kk]n_{j,\mm}[\kk']\} +  \E\{n_{j,\mm}^\HH[\kk]n_{j,\mm}^\HH[\kk']\} \Big).
\nonumber
\end{align}
Finally, noting that
\begin{align}
\E\{w_{j,\mm}[\kk]& w_{j,\mm}^\HH[\kk']\} = \nonumber \\ 
& \frac{1}{2} \Big( \E\{n_{j,\mm}[\kk]n_{j,\mm}[\kk']\} - \E\{n_{j,\mm}[\kk]n_{j,\mm}^\HH[\kk']\} \nonumber \\ &+ \E\{n_{j,\mm}^\HH[\kk]n_{j,\mm}[\kk']\} - \E\{n_{j,\mm}^\HH[\kk]n_{j,\mm}^\HH[\kk']\} \Big)
\nonumber
\end{align}
and, invoking the same arguments, we see that $w_{j,\mm}[\kk]$ and $w_{j,\mm}^\HH[\kk']$ are uncorrelated random variables.

\section{Proof of Proposition \ref{prop:decaycor}}
\label{ap:decaycor}
Since $\psi_m \in \LL^2(\RR)$, we have
\begin{multline}
\forall \tau\in\RR,\\
\gamma_{\psi_m,\psi_m}(\tau) = \frac{1}{2\pi}
\int_{-\infty}^\infty |\widehat{\psi}_m(\omega)|^2 e^{\imath \omega\tau}\,d\omega.
\nonumber
\end{multline}
% As $\widehat{\psi}_m$ is $2N_m+1$ times continuously
% differentiable
% on $\RR$, so is $|\widehat{\psi}_m|^2$. Leibniz formula
% allows us to express its derivative of order $n \in \{0,\ldots,2N_m+1\}$ as
% \begin{equation}
% (|\widehat{\psi}_m|^2)^{(q)} =\sum_{\ell=0}^n
% \binom{n}{\ell}
% (\widehat{\psi}_m)^{(\ell)} (\widehat{\psi}_m^*)^{(n-\ell)}.
% \label{eq:leibniz}
% \end{equation}
% Furthermore, for all $\ell \in \{0,\ldots,n\}$, $\widehat{\psi}_m^{(\ell)}
% \in \LL^2(\RR)$ and, consequently,
% $(|\widehat{\psi}_m|^2)^{(q)}\in \LL^1(\RR)$.
Furthermore, $|\widehat{\psi}_m|^2$ is $2N_m+1$ times continuously
differentiable and for all $q \in \{0,\ldots,2N_m+1\}$,
$(|\widehat{\psi}_m|^2)^{(q)}\in \LL^1(\RR)$.
It can be deduced \cite{Schwartz_L_1997_book_ana_atm}[p. 158--159] that
\begin{multline}
\forall \tau\in\RR,\\
(-\imath\tau)^{2N_m+1} \gamma_{\psi_m,\psi_m}(\tau) = \frac{1}{2\pi}
\int_{-\infty}^\infty (|\widehat{\psi}_m|^2)^{(2N_m+1)}(\omega)\,e^{\imath \omega\tau}\,d\omega
\nonumber
\end{multline}
which leads to
\begin{multline}
\forall \tau\in\RR,\\
|\tau|^{2N_m+1} |\gamma_{\psi_m,\psi_m}(\tau)| \leq \frac{1}{2\pi}
\int_{-\infty}^\infty \big|(|\widehat{\psi}_m|^2)^{(2N_m+1)}(\omega)\big|\,d\omega.
\label{eq:boundasymptpsimpsim}
\end{multline}

Let us now consider the cross-correlation
functions $\gamma_{\psi_m,\psi_m^\HH}$ with $m \neq 0$.
% \footnote{Proceeding as for
% $\gamma_{\psi_m,\psi_m}$ would not give the appropriate decay rate due to the fact
% that the Fourier transform of $\gamma_{\psi_m,\psi_m^\HH}$
% is the function $\omega \mapsto \imath\,\mathrm{sign}(\omega) |\widehat{\psi}_m(\omega)|^2$
% which can only be proved to be $2N-1$ times continuously
% differentiable on $\RR$.}
% Applying Lemma \ref{eq:lempsimnot0} to $|\widehat{\psi}_m|^2$, we get
% for all $\tau \in \RR$,
% \begin{align}
% &\int_0^\infty (|\widehat{\psi}_m|^2)^{(2N+1)}(\omega) \cos(\omega\tau)\,d\omega\nonumber\\
% = & \sum_{k=0}^{N} (-1)^{k+1} (|\widehat{\psi}_m|^2)^{(2N-2k)}(0)\;\tau^{2k}+(-1)^{N} \tau^{2N+1} \int_0^\infty |\widehat{\psi}_m(\omega)|^2 \sin(\omega \tau)\,d\omega.
% \label{eq:lemmaconsmn0}
% \end{align}
% The fact that $|\widehat{\psi}_m(\omega)|^2 = O(\omega^{2N})$ as $\omega \to 0$,
%  shows that, for all $q \in \{0,\ldots,2N-1\}$, $(|\widehat{\psi}_m|^2)^{(q)}(0) = 0$.
% Eq. \eqref{eq:lemmaconsmn0} thus reduces to
% \begin{align}
% &\int_0^\infty (|\widehat{\psi}_m|^2)^{(2N+1)}(\omega) \cos(\omega\tau)\,d\omega\nonumber\\
% = & -(|\widehat{\psi}_m|^2)^{(2N)}(0)+(-1)^{N} \tau^{2N+1} \int_0^\infty |\widehat{\psi}_m(\omega)|^2 \sin(\omega \tau)\,d\omega.
% \end{align}
% Using \eqref{eq:gammapsi}, we deduce that
% \begin{equation}
% \forall \tau\in\RR,\qquad
% |\tau|^{2N+1} |\gamma_{\psi_m,\psi_m^\HH}(\tau)| \leq
% \frac{1}{\pi}
% \Big(\int_0^\infty |(|\widehat{\psi}_m|^2)^{(2N+1)}(\omega)|\,d\omega+
% |(|\widehat{\psi}_m|^2)^{(2N)}(0)|\Big).
% \end{equation}
Similarly, when $m\neq 0$, we have
\begin{multline}
\forall \tau\in\RR,\\
\gamma_{\psi_m,\psi_m^\HH}(\tau) = \frac{1}{2\pi}
\int_{-\infty}^\infty \alpha(\omega)|\widehat{\psi}_m(\omega)|^2 e^{\imath \omega\tau}\,d\omega,
\label{eq:gammapsimpsimHfreq}
\end{multline}
where $\alpha(\omega) = \imath\,\mathrm{sign}(\omega)$.
The function $\omega \mapsto \alpha(\omega) |\widehat{\psi}_m(\omega)|^2$ is $2N_m+1$ times continuously
differentiable on $\RR^*$, where its derivative of order $q \in \{0,\ldots,2N_m+1\}$ is
\begin{equation}
(\alpha|\widehat{\psi}_m|^2)^{(q)} = \alpha\; (|\widehat{\psi}_m|^2)^{(q)}.
\label{eq:dernalphapsi}
\end{equation}
%Combining \eqref{eq:leibniz} with
Due to the fact that $|\widehat{\psi}_m(\omega)|^2 =
O(\omega^{2N_m})$ as $\omega \to 0$, we have for all $q \in
\{0,\ldots,2N_m-1\}$, $(|\widehat{\psi}_m|^2)^{(q)}(0) = 0$. From
\eqref{eq:dernalphapsi}, we deduce that the function
$(\alpha|\widehat{\psi}_m|^2)^{(q)}$ admits limits on the left
side and on the right side of 0, which are both equal to 0. This
allows to conclude that $\alpha|\widehat{\psi}_m|^2$ is $2N_m-1$
times continuously differentiable on $\RR$, its $2N_m-1$ first
derivatives vanishing at 0. Besides, $(\alpha
|\widehat{\psi}_m|^2)^{(2N_m-1)}$ is continuously differentiable
on $(-\infty,0]$ and on $[0,\infty)$ ($(\alpha
|\widehat{\psi}_m|^2)^{(2N_m)}$ may be discontinuous at 0). Using
the same arguments as for $\gamma_{\psi_m,\psi_m}$, this allows us
to claim that
\begin{align}
& \forall \tau\in\RR,\nonumber \\
& (-\imath \tau)^{2N_m} \gamma_{\psi_m,\psi_m^\HH}(\tau) = \frac{1}{2\pi}
\int_{-\infty}^\infty \alpha(\omega) (|\widehat{\psi}_m|^2)^{(2N_m)}(\omega) e^{\imath \omega\tau}\,d\omega\nonumber\\
& = -\frac{1}{\pi}
\int_0^\infty (|\widehat{\psi}_m|^2)^{(2N_m)}(\omega) \sin(\omega\tau)\,d\omega.
\label{eq:tau2NpsimpsimH}
\end{align}
We can note that
$\lim_{\omega \to \infty}  (|\widehat{\psi}_m|^2)^{(2N_m)}(\omega) \in \RR$
as it is equal to $(|\widehat{\psi}_m|^2)^{(2N_m)}_+(0)+\int_{0}^\infty (|\widehat{\psi}_m|^2)^{(2N_m+1)}(\nu)\,d\nu$
where $(|\widehat{\psi}_m|^2)^{(2N_m)}_+(0)$ denotes the right-hand side derivative of order $2N_m$ of $|\widehat{\psi}_m|^2$ at 0.
Since $(|\widehat{\psi}_m|^2)^{(2N_m)}\in \LL^1([0,\infty))$,
the previous limit is necessarily zero. Using this fact and integrating by part
in \eqref{eq:tau2NpsimpsimH}, we find that, for all $\tau \in \RR$,
\begin{multline}
\tau \int_0^\infty (|\widehat{\psi}_m|^2)^{(2N_m)}(\omega) \sin(\omega\tau)\,d\omega
= (|\widehat{\psi}_m|^2)^{(2N_m)}_+(0)\\ + 
\int_0^\infty (|\widehat{\psi}_m|^2)^{(2N_m+1)}(\omega) \cos(\omega\tau)\,d\omega.
\nonumber
\end{multline}
Combining this expression with \eqref{eq:tau2NpsimpsimH}, we deduce that
\begin{multline}
\forall \tau\in\RR,\\
|\tau|^{2N_m+1} |\gamma_{\psi_m,\psi_m^\HH}(\tau)| \leq
\frac{1}{\pi}
\Big(\int_0^\infty |(|\widehat{\psi}_m|^2)^{(2N_m+1)}(\omega)|\,d\omega \\+
|(|\widehat{\psi}_m|^2)^{(2N_m)}_+(0)|\Big).
\label{eq:boundasymptpsimpsimh}
\end{multline}
% \begin{equation}
% \forall \tau\in\RR,\qquad
% |\tau|^{2N} |\gamma_{\psi_m,\psi_m^\HH}(\tau)| \leq \frac{1}{2\pi}
% \int_{-\infty}^\infty \big|(|\widehat{\psi}_m|^2)^{(2N)}(\omega)\big|\,d\omega.
% \label{eq:boundasymptpsimpsimh}
% \end{equation}

Let us now study the case when $m=0$. Eq. \eqref{eq:gammapsimpsimHfreq} still holds, but as shown
by  \eqref{eq:linkpsi0Hpsi0}, $\alpha$ takes a more complicated form:
\begin{multline}
\forall k \in \ZZ ,\; \forall \omega \in [2k\pi,2(k+1)\pi),\\
\alpha(\omega) =
\begin{cases}
(-1)^k e^{\imath(d+\frac{1}{2})\omega} & \mbox{if $k \geq 0$}\\
(-1)^{k+1} e^{\imath(d+\frac{1}{2})\omega} & \mbox{otherwise.}
\end{cases}
\nonumber
\end{multline}
So, the function $\alpha$ as well as its derivatives of any order
now exhibit discontinuities at $2 k\pi$ where $k\in \ZZ^*$. However, from \eqref{eq:twoscalef} and the low-pass condition
$\widehat{\psi}_0(0) = 1$, we have, for all $m \neq 0$,
\begin{equation}
H_m(\omega) = O(\omega^{N_m}), \qquad \mbox{as $\omega \to 0$.}
\nonumber
\end{equation}
As a consequence of the para-unitary condition \eqref{eq:paraunitarity},
we get
\begin{equation}
\sum_{m=0}^{M-1} |H_m(\omega)|^2 = M
\nonumber
\end{equation}
and
\begin{equation}
\sum_{p=0}^{M-1} |H_0(\omega+p\frac{2\pi}{M})|^2 = M
\nonumber
\end{equation}
which allows to deduce that
\begin{equation}
\forall p \in \NMs,\qquad H_0(\omega+p\frac{2\pi}{M}) =
O(\omega^{N_0}).
\nonumber
\end{equation}
From \eqref{eq:twoscalef}, it can be concluded that
\begin{equation}
\forall k \in \ZZ^*,\; \widehat{\psi}_0(\omega+2k\pi) = O(\omega^{N_0}),
\; \mbox{as $\omega \to 0$.}
\label{eq:equivpsi02kpi}
\end{equation}
The derivatives of order $q \in \{0,\ldots,2N_0+1\}$ of $\alpha |\widehat{\psi}_0|^2$ over
$\RR \setminus \{2k\pi ,k\in \ZZ^*\}$ are given by
\begin{equation}
(\alpha |\widehat{\psi}_0|^2)^{(q)} = \sum_{\ell = 0}^q
\binom{q}{\ell}
(\alpha)^{(\ell)} (|\widehat{\psi}_0|^2)^{(q-\ell)},
\label{eq:derivnpsi0H}
\nonumber
\end{equation}
where
\begin{multline}
\forall k \in \ZZ ,\; \forall \omega \in (2k\pi,2(k+1)\pi),\\
\alpha^{(\ell)}(\omega) =
\begin{cases}
(-1)^k \imath^\ell (d+\frac{1}{2})^\ell\;e^{\imath(d+\frac{1}{2})\omega} & \mbox{if $k \geq 0$}\\
(-1)^{k+1} \imath^\ell (d+\frac{1}{2})^\ell\;e^{\imath(d+\frac{1}{2})\omega} & \mbox{otherwise.}
\end{cases}
\nonumber
\end{multline}
We deduce that, for all $q \in \{0,\ldots,2N_0+1\}$,
$(\alpha |\widehat{\psi}_0|^2)^{(q)} \in \LL^1(\RR)$.
Furthermore, combining \eqref{eq:equivpsi02kpi} with \eqref{eq:derivnpsi0H}
allows us to show that, for all $q \in \{0,\ldots,2N_0-1\}$, the derivative
of order $q$ of
$\alpha |\widehat{\psi}_0|^2$ at $2k\pi$, $k\in \ZZ^*$, is defined and equal to 0. Consequently,
$\alpha |\widehat{\psi}_0|^2$ is $2N_0-1$ times continuously differentiable on $\RR$ while $(\alpha |\widehat{\psi}_0|^2)^{(2N_0-1)}$ is continuously differentiable on
$\cup_{k\in \ZZ} (2k\pi,2(k+1)\pi)$. Similarly to the case $m\neq 0$, this leads to
\begin{align}
& \forall \tau\in\RR,\nonumber \\
& (-\imath \tau)^{2N_0} \gamma_{\psi_0,\psi_0^\HH}(\tau) = \frac{1}{2\pi}
\int_{-\infty}^\infty (\alpha |\widehat{\psi}_0|^2)^{(2N_0)}(\omega)\,e^{\imath \omega\tau}\,d\omega\nonumber \\
& = \frac{1}{2\pi} \sum_{k=-\infty}^\infty
\int_{2k\pi}^{2(k+1)\pi} (\alpha |\widehat{\psi}_0|^2)^{(2N_0)}(\omega)\,e^{\imath \omega\tau}\,d\omega.
\end{align}
By integration by part, we deduce that
\begin{align}
& \forall \tau\in\RR,\nonumber \\
&(-\imath \tau)^{2N_0+1} \gamma_{\psi_0,\psi_0^\HH}(\tau) = \frac{1}{2\pi} \nonumber \\
& \qquad \times \Big(\int_{-\infty}^\infty (\alpha |\widehat{\psi}_0|^2)^{(2N_0+1)}(\omega)\,e^{\imath \omega\tau}\,d\omega
+ \beta\Big)
\label{eq:tauN+1gammapsi0psi0H}\\
&\beta=\sum_{k\in \ZZ^*} \big((\alpha |\widehat{\psi}_0|^2)^{(2N_0)}_+(2k\pi)
- (\alpha |\widehat{\psi}_0|^2)^{(2N_0)}_-(2k\pi)\big)\,e^{\imath 2\pi k\tau},
\label{eq:seriesbeta}
\end{align}
where $(\alpha |\widehat{\psi}_0|^2)^{(2N_0)}_+(\omega_0)$ (resp.
$(\alpha |\widehat{\psi}_0|^2)^{(2N_0)}_-(\omega_0)$) denotes the right-side (resp.
left-side) derivative of order $2N_0$ of $\alpha |\widehat{\psi}_0|^2$ at $\omega_0 \in \RR$.\footnote{The series in \eqref{eq:seriesbeta} is convergent since all the other terms in \eqref{eq:tauN+1gammapsi0psi0H} are finite.}
We conclude that
\begin{multline}
\forall \tau\in\RR, \quad
|\tau|^{2N_0+1}|\gamma_{\psi_0,\psi_0^\HH}(\tau)| =\\ \frac{1}{2\pi}
\Big(\int_{-\infty}^\infty \big|(\alpha |\widehat{\psi}_0|^2)^{(2N_0+1)}(\omega)\big|\,d\omega
 + |\beta|\Big).
\label{eq:boundasymptpsi0psi0h}
\end{multline}
In summary, we have proved that \eqref{eq:decayasymptcor1}
and  \eqref{eq:decayasymptcor2} hold,
the constant $C$ being chosen equal to the maximum value of the left-hand side terms in the inequalities \eqref{eq:boundasymptpsimpsim}, \eqref{eq:boundasymptpsimpsimh} and
\eqref{eq:boundasymptpsi0psi0h}.

\section{Proof of Proposition \ref{prop:decaycov1D}}\label{ap:decaycov1D}
Let $m\in \NM$. Since $\psi_m$ is a unit norm function of
$\LL^2(\RR)$, the function $\gamma_{\psi_m,\psi_m^\HH}$ is upper
bounded by 1. As $\gamma_{\psi_m,\psi_m^\HH}$ further satisfies
\eqref{eq:decayasymptcor2}, it can be deduced that
\begin{equation}
\forall \tau \in \RR,\qquad
|\gamma_{\psi_m,\psi_m^\HH}(\tau)| \leq \frac{1+C}{1+|\tau|^{2N_m+1}}.
\label{eq:upboundbisgammapsimpsimh}
\end{equation}
The same upper bound holds for $\gamma_{\psi_m,\psi_m}$.

For a white noise, the property then appears as a straightforward consequence
of the latter inequality and Eqs.~\eqref{eq:internoiseb1D} and
\eqref{eq:internoise1D}.

Let us next turn our attention to processes with exponentially
decaying covariance sequences. From \eqref{eq:bbh1D},
\eqref{eq:expdecay} and \eqref{eq:upboundbisgammapsimpsimh}, we
deduce that
\begin{multline}
\forall \ell \in \ZZ,\qquad\\ |\Gamma_{n_{j,m},n_{j,m}^\HH}[\ell]|
\leq A (1+C) \int_{-\infty}^\infty \frac{e^{-\alpha
|x|}}{1+|M^{-j} x -\ell|^{2N_m+1}}\,dx. \label{eq:ACmaj}
\end{multline}
As the left-hand side of \eqref{eq:ACmaj} corresponds to an even function of $\ell$, without loss of generality, it can be assumed that
$\ell \geq 0$. We can decompose the above integral as
\begin{multline}
\int_{-\infty}^\infty \frac{e^{-\alpha |x|}}{1+|M^{-j} x -\ell|^{2N_m+1}}\,dx \\
= \int_0^\infty \frac{e^{-\alpha x}}{1+(M^{-j} x +\ell)^{2N_m+1}}\,dx \\
+ \int_0^\infty \frac{e^{-\alpha x}}{1+|M^{-j} x -\ell|^{2N_m+1}}\,dx \,.
\nonumber
\end{multline}
The first integral in the right-hand side can be upper bounded as follows
\begin{multline}
 \int_0^\infty \frac{e^{-\alpha x}}{1+(M^{-j} x +\ell)^{2N_m+1}}\,dx \\
\leq (1+\ell^{2N_m+1})^{-1}\int_0^\infty e^{-\alpha x}\,dx \\ = \alpha^{-1}(1+\ell^{2N_m+1})^{-1}.
\nonumber
\end{multline}
Let $\epsilon \in (0,1)$ be given. The second integral can be decomposed as
\begin{multline}
\int_0^\infty \frac{e^{-\alpha x}}{1+|M^{-j} x -\ell|^{2N_m+1}}\,dx \\
= \int_0^{\epsilon M^j\ell} \frac{e^{-\alpha x}}{1+(\ell-M^{-j} x)^{2N_m+1}}\,dx \\
+ \int_{\epsilon M^j\ell}^\infty \frac{e^{-\alpha x}}{1+|M^{-j} x-\ell|^{2N_m+1}}\,dx.
\nonumber
\end{multline}
Furthermore, we have
\begin{align}
\int_0^{\epsilon M^j\ell} & \frac{e^{-\alpha x}}{1+(\ell-M^{-j} x)^{2N_m+1}}\,dx \nonumber \\
&\leq (1+(1-\epsilon)^{2N_m+1} \ell^{2N_m+1})^{-1} \int_0^{\epsilon M^j\ell} e^{-\alpha x}\,dx\nonumber\\
&\leq \alpha^{-1} (1-\epsilon)^{-2N_m-1} (1+\ell^{2N_m+1})^{-1}\\
\int_{\epsilon M^j\ell}^\infty & \frac{e^{-\alpha x}}{1+|M^{-j} x-\ell|^{2N_m+1}}\,dx \nonumber \\ &\leq \int_{\epsilon M^j\ell}^\infty e^{-\alpha x}\,dx
= \alpha^{-1} e^{-\alpha \epsilon M^j\ell}.
\nonumber
\end{align}
From the above inequalities, we obtain
\begin{multline}
\forall \ell \in \NN^*,\\ 
|\Gamma_{n_{j,m},n_{j,m}^\HH}[\ell]|
\leq A (1+C)
\alpha^{-1}\big((1+(1-\epsilon)^{-2N_m-1})\\\times (1+\ell^{2N_m+1})^{-1}+e^{-\alpha
\epsilon M^j\ell}\big).
\nonumber
\end{multline}
As $\lim_{\ell\to \infty} (1+\ell^{2N_m+1})e^{-\alpha \epsilon M^j\ell}=0$, it readily follows
that there exists $\widetilde{C} \in \RR_+$ such that
\eqref{eq:boundcovannh} holds.

The left-hand side of \eqref{eq:ACmaj} being also an upper bound
for $|\Gamma_{n_{j,m},n_{j,m}}[\ell]|$, $\ell \neq 0$,
\eqref{eq:boundcovann} is proved at the same time.

\section{Proof of Proposition \ref{prop:asymptj}}
\label{ap:asymptj}
Let us prove \eqref{eq:internoise1Dasym}, the proof of
\eqref{eq:internoiseb1Dasym} being quite similar.
We first note that $\widehat{\psi}_m(\widehat{\psi}_{m'}^\HH)^*$ and therefore $\gamma_{\psi_m,\psi_{m'}^\HH}$
belong to $\LL^2(\RR)$ (see footnote \ref{f:L2gamma}).
% , since
% $\psi_m \in \LL^1(\RR)$ and $\psi_{m'}^\HH \in \LL^2(\RR)$, $\gamma_{\psi_m,\psi_{m'}^\HH} \in \LL^2(\RR)$.
Applying Parseval's equality to \eqref{eq:bbh1D}, we obtain
for all $\ell \in \ZZ$,
\begin{align}
& \Gamma_{n_{j,m},n_{j,m'}^\HH}[\ell] \nonumber \\
&= \frac{1}{2\pi}
\int_{-\infty}^\infty
\widehat{\Gamma}_n(\omega)\,M^j\widehat{\psi}_m^*(M^j\omega)
\widehat{\psi}_{m'}^\HH(M^j \omega)e^{\imath M^j  \ell \omega} d\omega\nonumber \\
& = \frac{1}{2\pi} \int_{-\infty}^\infty \widehat{\Gamma}_n\Big(\frac{\omega}{M^j}\Big)\,\widehat{\psi}_m^*(\omega)
\widehat{\psi}_{m'}^\HH(\omega)e^{\imath \ell \omega} d\omega.
\nonumber
\end{align}
As $\Gamma_n \in \LL^1(\RR)$, the spectrum density $\widehat{\Gamma}_n$ is a bounded continuous function. According to Lebesgue dominated convergence theorem,
\begin{align}
& \lim_{j\to\infty}\Gamma_{n_{j,m},n_{j,m'}^\HH}[\ell] \nonumber \\ &=
\frac{1}{2\pi} \int_{-\infty}^\infty
\lim_{j\to\infty}\widehat{\Gamma}_n\Big(\frac{\omega}{M^j}\Big)\,\widehat{\psi}_m^*(\omega)
\widehat{\psi}_{m'}^\HH(\omega)e^{\imath \ell \omega} d\omega\nonumber\\
&= \frac{\widehat{\Gamma}_n(0)}{2\pi} \int_{-\infty}^\infty \widehat{\psi}_m^*(\omega)
\widehat{\psi}_{m'}^\HH(\omega)e^{\imath \ell \omega} d\omega = \widehat{\Gamma}_n(0)
\gamma_{\psi_m,\psi_{m'}^\HH}(-\ell).
\nonumber
\end{align}

\section{Cross-correlations for Meyer wavelets}
\label{ap:Meyer}
Substituting \eqref{eq:Meyer0} in \eqref{eq:gammaphi}, we obtain, for all
$\tau \in \RR$,
\begin{align}
\gamma_{\psi_0,\psi_0^\HH}(\tau)=&\frac{1}{\pi}\Big(\int_{0}^{\pi(1-\epsilon)}\cos{\Big(\omega(d+\frac{1}{2}+ \tau)\Big)} d\omega \nonumber \\ + \int_{\pi(1-\epsilon)}^{\pi(1+\epsilon)} & W^2\big(\frac{\omega}{2\pi\epsilon}-\frac{1-\epsilon}{2\epsilon}\big)\cos{\Big(\omega(d+\frac{1}{2}+ \tau)\Big)}d\omega\Big)\nonumber \\
= &(1-\epsilon)\mathrm{sinc}\Big(\pi(1-\epsilon)(d+\frac{1}{2}+ \tau)\Big)
\nonumber \\ + \epsilon \int_{-1}^1 & 
W^2\Big(\frac{1+\theta}{2}\Big) \cos{\Big(\pi(\epsilon \theta +1)\big(d+\frac{1}{2}+ \tau\big)\Big)} d\theta.
\label{eq:meyer0interm}
\end{align}
Using \eqref{eq:propwindow}, we get
\begin{multline}
\int_{-1}^0
W^2\Big(\frac{1+\theta}{2}\Big) \cos{\Big(\pi(\epsilon \theta +1)\big(d+\frac{1}{2}+ \tau\big)\Big)} d\theta \\ =
\int_0^1  \cos{\Big(\pi(\epsilon \theta -1)\big(d+\frac{1}{2}+ \tau\big)\Big)} d\theta\\
- \int_0^1 W^2\Big(\frac{1+\theta}{2}\Big) \cos{\Big(\pi(\epsilon \theta -1)\big(d+\frac{1}{2}+ \tau\big)\Big)} d\theta.
\label{eq:int-10int01}
\end{multline}
This allows us to rewrite \eqref{eq:meyer0interm} as
\begin{multline}
\gamma_{\psi_0,\psi_0^\HH}(\tau)= \mathrm{sinc}\Big(\pi(d+\frac{1}{2}+ \tau)\Big)
-\sin{\Big(\pi \big(d+\frac{1}{2}+ \tau\big)\Big)} \\
\times I_\epsilon\Big(d+\frac{1}{2}+ \tau\Big).
\end{multline}
After simplification, \eqref{eq_inter_meyer0} follows.

According to \eqref{eq:gammapsi} and \eqref{eq:Meyerm}, we have
for all $m \in \{1,\ldots,M-2\}$ and $\tau \in \RR^*$,
\begin{align}
& \gamma_{\psi_m,\psi_m^\HH}(\tau) \nonumber \\
=&-\frac{1}{\pi}\Big(
\int_{\pi(m-\epsilon)}^{\pi(m+\epsilon)}W^2\Big(
\frac{m+\epsilon}{2\epsilon}-\frac{\omega}{2\pi\epsilon}\Big)\sin(\omega \tau) d\omega \nonumber \\ & +
\int_{\pi(m+\epsilon)}^{\pi(m+1-\epsilon)}\sin(\omega \tau) d\omega \nonumber\\ &+
\int_{\pi(m+1-\epsilon)}^{\pi(m+1+\epsilon)}W^2\Big(\frac{\omega}{2\pi\epsilon}
-\frac{m+1-\epsilon}{2\epsilon}\Big)\sin(\omega \tau) d\omega\Big)\nonumber\\
= & \frac{\cos\big(\pi(m+1-\epsilon)\tau\big)-\cos\big(\pi(m+\epsilon)\tau\big)}{\pi\tau} \nonumber \\ & +\epsilon
\int_{-1}^1 W^2\Big(\frac{1+\theta}{2}\Big)  \sin\big(\pi(\epsilon \theta-m) \tau\big)d\theta\nonumber\\
& -\epsilon\int_{-1}^1 W^2\Big(\frac{1+\theta}{2}\Big) \sin\big(\pi(\epsilon \theta+m+1) \tau\big)d\theta.
\nonumber
\end{align}
By proceeding similarly to \eqref{eq:meyer0interm}-\eqref{eq:int-10int01}, we find
\begin{equation}
\gamma_{\psi_m,\psi_m^\HH}(\tau)
= \big(\cos(\pi(m+1)\tau)-\cos(\pi m\tau)\big)
\Big(\frac{1}{\pi\tau}- I_\epsilon(\tau)\Big).
\nonumber
\end{equation}
When $\tau$ is an integer, this expression further simplifies in \eqref{eq_inter_meyerm}.

Finally, when $m=M-1$, we have, for all $\tau \in \RR^*$,
\begin{align}
& \gamma_{\psi_{M-1},\psi_{M-1}^\HH}(\tau) \nonumber \\
=&-\frac{1}{\pi}\Big(
\int_{\pi(M-1-\epsilon)}^{\pi(M-1+\epsilon)}W^2\Big(
\frac{M-1+\epsilon}{2\epsilon}-\frac{\omega}{2\pi\epsilon}\Big)\sin(\omega \tau) d\omega \nonumber \\ & +
\int_{\pi(M-1+\epsilon)}^{\pi M(1-\epsilon)}\sin(\omega \tau) d\omega \nonumber\\ &+
\int_{\pi M(1-\epsilon)}^{\pi M(1+\epsilon)}W^2\Big(\frac{\omega}{2\pi\epsilon M}
-\frac{1-\epsilon}{2\epsilon}\Big)\sin(\omega \tau) d\omega\Big)\nonumber\\
= & \frac{\cos\big(\pi M(1-\epsilon)\tau\big)-\cos\big(\pi(M-1+\epsilon)\tau\big)}{\pi\tau} \nonumber \\ & + \epsilon \int_{-1}^1 W^2\Big(\frac{1+\theta}{2}\Big) \sin\big(\pi(\epsilon \theta-M+1) \tau\big)d\theta\nonumber\\
& -\epsilon M \int_{-1}^1 W^2\Big(\frac{1+\theta}{2}\Big) \sin\big(\pi M(\epsilon \theta+1) \tau\big)d\theta\nonumber\\
= & \frac{\cos\big(\pi M
\tau\big)-\cos\big(\pi(M-1)\tau\big)}{\pi\tau}+\cos\big(\pi(M-1)\tau\big)
I_\epsilon(\tau) \nonumber \\
 & -\cos(\pi M \tau)  I_{M\epsilon}(\tau).
 \nonumber
\end{align}
This yields \eqref{eq_inter_meyerM}.

\section{Proof of Proposition \ref{prop:recurpacket}}
\label{ap:recurpacket}
Let $m\in \NN^*$.
Given  (\ref{eq:gammapsi}), \eqref{eq:packet} leads to
\begin{align}
-\pi \gamma_{\psi_{2m},\psi_{2m}^\HH}(\tau)= &
\int_{0}^{\infty}|\widehat{\psi}_{2m}(\omega)|^2\;\sin(\omega\tau)\;d\omega\nonumber\\
= & \int_{0}^{\infty}|A_0(\omega)|^2 |\widehat{\psi}_{m}(\omega)|^2\;\sin(2\omega\tau)\;d\omega
\label{eq:packetA0}
\end{align}
Furthermore, we have
\begin{align}
|A_0(\omega)|^2&=\sum_k \gamma_{a_0}[k] \exp(-\imath k \omega)\nonumber\\
&= \gamma_{a_0}[0] + 2 \sum_{k=1}^{\infty} \gamma_{a_0}[k] \cos(k \omega).
\nonumber
\end{align}
Combining this equation with \eqref{eq:packetA0} and using classical trigonometric
equalities, we obtain
\begin{equation}
\begin{split}
-\pi \gamma_{\psi_{2m},\psi_{2m}^\HH}(\tau) = &\, \gamma_{a_0}[0] \int_0^\infty|\widehat{\psi}_{m}(\omega)|^2 \sin(2 \omega \tau)d \omega\\
+ \sum_{k=1}^\infty & \gamma_{a_0}[k] \Big( \int_0^\infty  |\widehat{\psi}_{m}(\omega)|^2 \sin\big((2\tau-k) \omega\big) d \omega \\
&+ \int_0^\infty  |\widehat{\psi}_{m}(\omega)|^2 \sin\big((2\tau+k)\omega\big) d\omega \Big)
\end{split}
\label{eq:paquetimp}
\nonumber
\end{equation}
which, again invoking Relation \eqref{eq:gammapsi}, yields \eqref{eq:interpair}.
Eq. \eqref{eq:interimpair} can be proved similarly starting from \eqref{eq:packetb}.

\section{Cross-correlations for Haar wavelet}
\label{ap:haar}
Knowing the expression of the Fourier transform of the Haar scaling function
in \eqref{eq_ond_haar0} and using the cross-correlation formula \eqref{eq:gammaphi}, we obtain:
\begin{align}
& \forall \tau\in \RR,\quad
\gamma_{\psi_{0},\psi_{0}^\HH}(\tau) \nonumber \\
&=\frac{1}{\pi} \sum_{k=0}^\infty
(-1)^k \int_{2k\pi}^{2(k+1)\pi}
\text{sinc}^2(\frac{\omega}{2})\;\cos\big(\omega\;(\frac{1}{2}+\tau+d)\big)\;d\omega\nonumber\\
&=\frac{2}{\pi} \sum_{k=0}^\infty
(-1)^k \int_{k\pi}^{(k+1)\pi}\frac{\sin^2{(\nu)}}{\nu^2}\cos\big(\nu\;(1+2\tau+2d)\big)\;d\nu.
\label{eq:Haargammapsi0inter}
\end{align}
By integration by part, we find: for all $(\alpha,\beta,\eta)\in \RR^3$,
\begin{align}
& \int_\alpha^\beta\frac{\sin^2(\omega)}{\omega^2}\cos(\eta \omega)d\omega \nonumber \\ =&\frac{\sin^2(\alpha)\cos(\eta \alpha)}{\alpha}-\frac{\sin^2(\beta)\cos(\eta \beta)}{\beta} \nonumber \\ & +\frac{1}{4}(2+\eta)\int_\alpha^\beta \frac{\sin\big((2+\eta)\omega\big)}{\omega}d\omega-\frac{\eta}{2}\int_\alpha^\beta \frac{\sin(\eta \omega)}{\omega}d\omega \nonumber\\& +\frac{1}{4}(2-\eta)\int_\alpha^\beta \frac{\sin\big((2-\eta )\omega\big)}{\omega}d\omega\nonumber\\
=&\frac{\sin^2(\alpha)\cos(\eta \alpha)}{\alpha}-\frac{\sin^2(\beta)\cos(\eta \beta)}{\beta} \nonumber \\ & +\frac{1}{4}(\eta+2)\int_{\alpha(\eta+2)}^{\beta(\eta+2)} \frac{\sin(\nu)}{\nu}d\nu -\frac{\eta}{2}\int_{\alpha\eta}^{\beta\eta} \frac{\sin(\nu)}{\nu}d\nu \nonumber\\ & +\frac{1}{4}(\eta-2)\int_{\alpha(\eta-2)}^{\beta(\eta-2)} \frac{\sin(\nu)}{\nu}d\nu.
\nonumber
\end{align}
Combining this result with \eqref{eq:Haargammapsi0inter} leads to \eqref{eq_inter_haar0}.

On the other hand, according to \eqref{eq_ond_haar1} and \eqref{eq:gammapsi},
we have
\begin{multline}
\forall \tau \in \RR,\\
\gamma_{\psi_{1},\psi_{1}^\HH}(\tau)=-\frac{1}{\pi}
\int_{0}^{\infty}\textrm{sinc}^2(\frac{\omega}{4})\;\sin^2{(\frac{\omega}{4})}\;\sin(\omega\tau)\;d\omega\,.
\nonumber
\end{multline}
In \cite[p.459]{Gradshteyn_I_2000_book_tab_isp},
an expression of $\int_0^{\infty}\frac{\sin^2{(\alpha x)}\sin^2{(\beta x)}\sin{(2\eta x)}dx}{x^2}$
with $(\alpha,\beta,\eta)\in (\RR_+^*)^3$ is given. Using this relation yields
\eqref{eq_inter_haar1} when $\tau > 0$. The general expression for $\tau \in \RR$ follows from the oddness of $\gamma_{\psi_{1},\psi_{1}^\HH}$.

\section{Cross-correlation for the Franklin wavelet}
\label{ap:splines}

We have, for all $\tau \in \RR$,
\begin{align*}
\gamma_{\chi,\chi^\HH}(\tau)&=-\frac{1}{\pi}\int_{0}^{\infty}
|\widehat{\chi}(\omega)|^2 \sin(\omega \tau) d\omega \nonumber \\
&=-\frac{2}{\pi}\int_{0}^{\infty} \frac{\sin^8(\omega)}{\omega^4}
\sin(2\omega \tau) d\omega.
\end{align*}
After two successive integrations by part, we obtain
\begin{align}
\gamma_{\chi,\chi^\HH}(\tau)=&-\frac{4}{3\pi}\Big(4\int_{0}^{\infty}\frac{\sin^7(\omega)\cos(\omega)\sin(2\omega
\tau)}{\omega^3}d\omega \nonumber \\
+ & \tau\int_{0}^{\infty}\frac{\sin^8(\omega)\cos(2\omega
\tau)}{\omega^3}d\omega\Big)\nonumber \\
=&-\frac{2}{3\pi}\Big(28\int_{0}^{\infty}\frac{\sin^6(\omega)\cos^2(\omega)\sin(2\omega
\tau)}{\omega^2}d\omega \nonumber\\
- & 2(2+\tau^2)\int_{0}^{\infty}\frac{\sin^8(\omega)\sin(2\omega
\tau)}{\omega^2}d\omega\nonumber\\
&+16\tau
\int_{0}^{\infty}\frac{\sin^7(\omega)\cos(\omega)\cos(2\omega\tau)}{\omega^2}d\omega\Big).
\label{eq:gachi}
\end{align}
Standard trigonometric manipulations allow us to write:
\begin{align}
\sin^6(\omega)\cos^2(\omega)\sin(2\omega \tau) =
%&\frac{1}{4}\sin^4(\omega) \sin^2(2\omega) \sin(2\omega \tau) \nonumber \\=
&\frac{1}{8} \sin^4(\omega)\Big(\sin(2\tau\omega) \nonumber \\
- \frac{1}{2}
\sin\big(2&(\tau+2)\omega\big) - \frac{1}{2} \sin \big(2(\tau-2)\omega\big)\Big) \nonumber \\
\sin^8(\omega)\sin(2\omega \tau) =
%&\frac{1}{8}\sin^4(\omega) \big(\cos(4\omega)-4\cos(2\omega)+3\big)\sin(2\omega \tau) \nonumber \\=
&\frac{1}{16}\sin^4(\omega) \Big(\sin \big(2(\tau+2)\omega\big)\nonumber \\
+ \sin & \big(2(\tau-2)\omega\big) -
\sin\big(2(\tau+1)\omega\big)\nonumber \\
- &4 \sin \big(2(\tau-1)\omega\big)+ 6\sin(2 \tau\omega)\Big)\nonumber \\
\sin^7(\omega)\cos(\omega)\cos(2\omega\tau) =
%& \frac{1}{8}\sin^4(\omega)\big(2\sin(2\omega)-\sin(4\omega)\big)\cos(2\omega\tau)\nonumber \\=
& \frac{1}{16}\sin^4(\omega) \Big(\sin \big(2(\tau-2)\omega\big) \nonumber \\
 - \sin & \big(2(\tau+2)\omega\big) +
2\sin \big(2(\tau+1)\omega\big)\nonumber\\
- & 2 \sin \big(2(\tau-1)\omega\big)\Big). \nonumber
\end{align}
Inserting these expressions in (\ref{eq:gachi}) yields
\begin{multline}
3\pi \gamma_{\chi,\chi^\HH}(\tau)=
Q_0(\tau) J(\tau) - Q_1(\tau) J(\tau+1)-Q_1(-\tau) J(\tau-1)\\
+ Q_2(\tau) J(\tau+2)+ Q_2(-\tau) J(\tau-2),
\label{eq:gammatildpsi1temp}
\end{multline}
where  (see \cite[p. 459]{Gradshteyn_I_2000_book_tab_isp})
\begin{align}
\forall x \in \RR,\quad J(x)&=2\int_0^\infty \frac{\sin^4(\omega)}{\omega^2}\sin(2\omega x) d\omega \nonumber \\
&=-\frac{3}{2}x\ln|x|+(1+x)\ln|1+x| \nonumber \\
& -(1-x)\ln|1-x|-\frac{2+x}{4}\ln|2+x|\nonumber \\
& +\frac{2-x}{4}\ln|2-x|
\nonumber
\end{align}
% where $\forall c>0$ (see \cite[p. 459]{Gradshteyn2000})
% \begin{align}
% \mathcal{I}(c)&=\int_0^\infty \frac{\sin^4(\omega)}{\omega^2}\sin(c\omega) d\omega \nonumber \\
% &=-\frac{3}{8}c\ln(c)-\frac{4+c}{16}\ln(4+c)+\frac{4-c}{16}\ln|4-c|+\frac{2+c}{4}\ln(2+c)-\frac{2-c}{4}\ln|2-c|,
% \end{align}
and
\begin{equation}
Q_0(\tau) = \frac{3}{4}\tau^2-2,\;
Q_1(\tau) = \frac{\tau^2}{2}+2\tau+1,\;
Q_2(\tau) = \frac{1}{8}(\tau+4)^2.
\nonumber
\end{equation}
Simple algebra allows us to prove that \eqref{eq:gammatildpsi1temp} is equivalent to \eqref{eq:gammatildpsi1}.

On the other hand,
$|\widetilde{A}_1(\omega)|^2$ can be viewed as the frequency response of a
non causal stable digital filter whose transfer function is
\begin{align*}
P_{\widetilde{A}_1}(z) &= \displaystyle \frac{6(2-\frac{z+z^{-1}}{2})}
{\big(1+2\big(\frac{z+z^{-1}}{2}\big)^2)(2+\frac{z+z^{-1}}{2})} \nonumber
\\
&= \frac{2\sqrt{3}}{9}\Big(\frac{4(2+\sqrt{3})}{z+2+\sqrt{3}}-\frac{4(2-\sqrt{3})}{z+2-\sqrt{3}}\nonumber \\
+& \frac{7(2+\sqrt{3})-4(1+\sqrt{3})z}{z^2+2+\sqrt{3}}-\frac{7(2-\sqrt{3})-4(1-\sqrt{3})z}{z^2+2-\sqrt{3}}\Big).
\end{align*}
We next expand $P_{\widetilde{A}_1}(z)$ in Laurent series on
the holomorphy domain containing the unit circle, that is
\begin{equation}
\mathcal{D}_{P_{\widetilde{A}_1}} = \Big\{ z \in \mathbb{C}\;\mid\;
\frac{\sqrt{3}-1}{\sqrt{2}} < |z| < \frac{\sqrt{3}+1}{\sqrt{2}}\Big\}.
\nonumber
\end{equation}
We thus deduce from the partial fraction decomposition of $P_{\widetilde{A}_1}(z)$ that
\begin{multline}
P_{\widetilde{A}_1}(z) =
\frac{2\sqrt{3}}{9}\Big(4\sum_{k=-\infty}^\infty (-1)^k (2-\sqrt{3})^{|k|} z^{-k} \\
+ 7\sum_{k=-\infty}^\infty (-1)^k (2-\sqrt{3})^{|k|} z^{-2k}\\
+ 4(1-\sqrt{3}) \sum_{k=0}^\infty (-1)^k (2-\sqrt{3})^k
\big(z^{2k+1}+z^{-2k-1}\big)\Big).
\nonumber
\end{multline}
By identifiying the latter expression with
$\sum_{k=-\infty}^\infty \gamma_{\tilde{a}_1}[k] z^{-k}$, \eqref{eq:expgammaat1} is obtained.

\end{appendices}

%\bibliographystyle{IEEEtran}
%\bibliography{abbr,Reference_Base_JabRef,Reference_Base_JabRef_Perso_CC}

\onecolumn
\pagebreak
\thispagestyle{empty}
\listoffigures
\listoftables
\addtocounter{page}{-1}

\newpage
%%%%%%%%%%%%%%%%%%%%%%%%%%%%%%%%%%%%%%%%%%%%%%%%%%%%%%%%%%%%%%%%%%%%%%%%%%%%%%%%%%%%%%%%%%%%%%%%%%%%%%%%
%figures and tables
%%%%%%%%%%%%%%%%%%%%%%%%%%%%%%%%%%%%%%%%%%%%%%%%%%%%%%%%%%%%%%%%%%%%%%%%%%%%%%%%%%%%%%%%%%%%%%%%%%%%%%%%

\begin{figure}[htbp]
\centering
\input{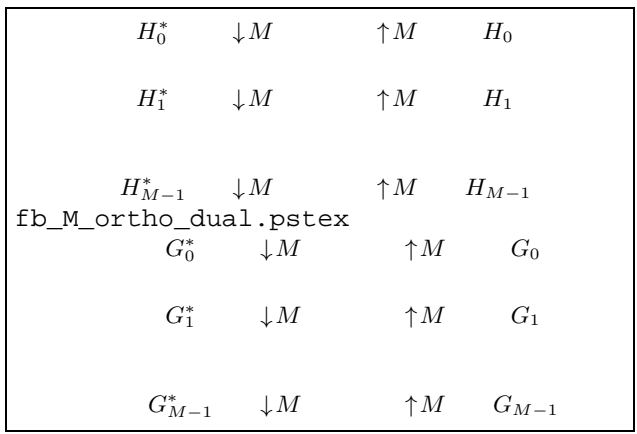}
\caption{A pair of primal (top) and dual (bottom) analysis/synthesis $M$-band para-unitary filter banks.}
\label{fig:Mband}
\end{figure}

\begin{figure}[ht]
\includegraphics[width=16cm]{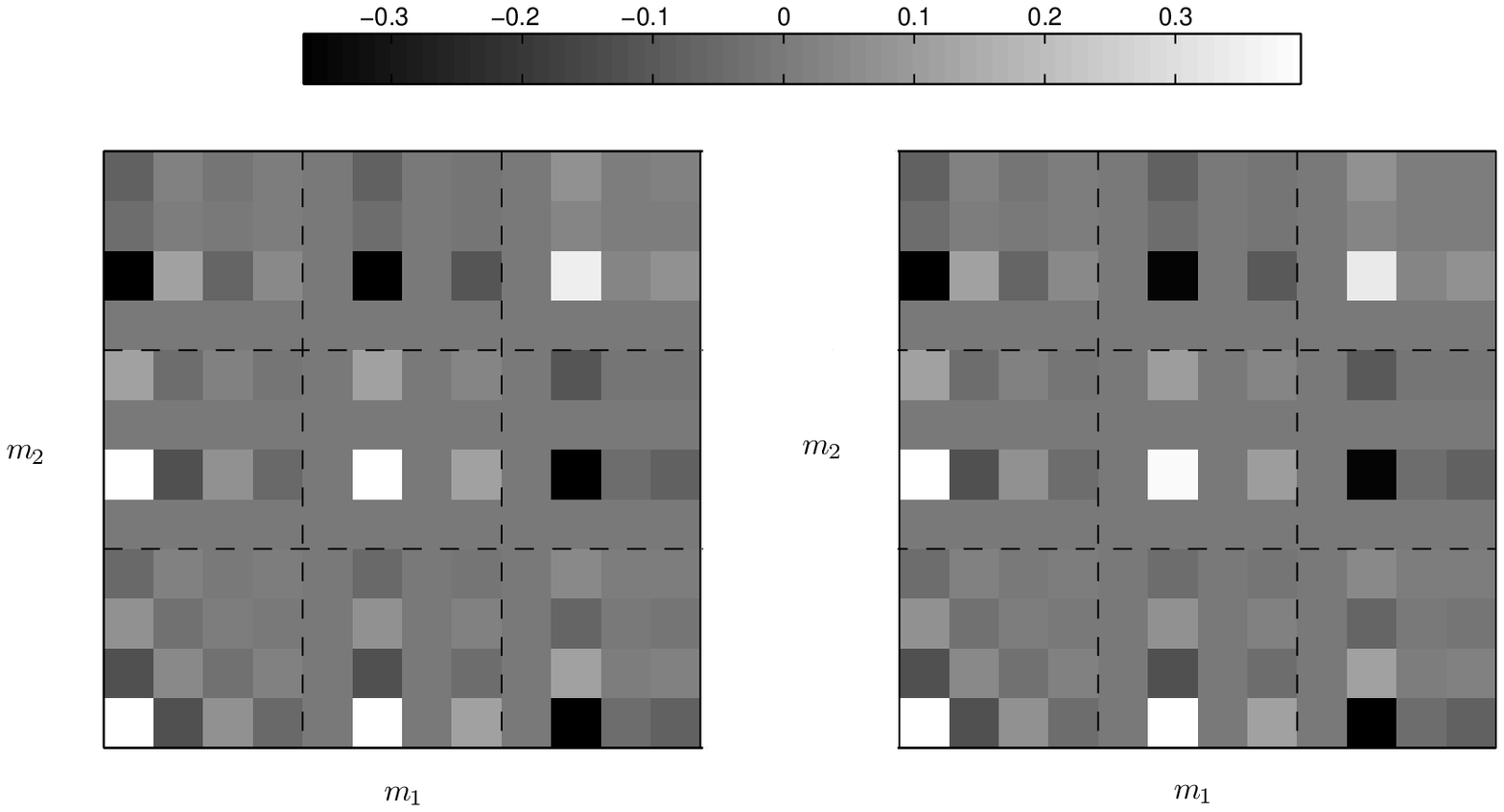}
\caption{$2D$ cross-correlations using $3$-band Meyer wavelets. Theoretical results (left); experimental results (right).\label{fig:inter2D}}
\end{figure}

\begin{table}[htbp]
\hspace{-0.6cm}
%\begin{center}
\begin{tabular}{|c||c|c|c|c|c|c|c|c|c|c|c|c|}
\hline
 $m$  & 1 & 2 & 3 & 4 & 5 & 6 & 7 & 8 & 9 & 10 & 11 & 12\\
\hline
{$\pi \gamma_{\psi_m,\psi_m^\HH}(\tau)$}
& $\frac{1}{2^3 \tau^3}$
& $\frac{1}{2^5 \tau^3}$
& $-\frac{3}{2^7 \tau^5}$
& $\frac{1}{2^7 \tau^3}$
& $-\frac{3}{2^9 \tau^5}$
& $-\frac{3}{2^{11} \tau^5}$
& $\frac{45}{2^{14} \tau^7}$
& $\frac{1}{2^9 \tau^3}$
& $-\frac{3}{2^{11} \tau^5}$
& $-\frac{3}{2^{13} \tau^5}$
& $\frac{45}{2^{16} \tau^7}$
& $-\frac{3}{2^{15} \tau^5}$
\\
\hline
\end{tabular}
\caption{Asymptotic form of $\gamma_{\psi_{m},\psi_{m}^\HH}(\tau)$ as $|\tau|\to \infty$ for Walsh-Hadamard wavelets.}
\label{tab:asympthaar}
%\end{center}
\end{table}

\begin{table*}[h]
%\begin{center}
\hspace*{-0.8cm}
{\footnotesize
\begin{tabular}{| c|| c| c| c| c|| c| c| c| }
\cline{2-8}
\multicolumn{1}{ c |}{ }
 & \multicolumn{4}{|c||}{$\gamma_{\psi_0,\psi_0^\HH}$}  & \multicolumn{3}{|c|}{$\gamma_{\psi_1,\psi_1^\HH}$} \\
\hline
Wavelets\; $\setminus$ \;$\ell$ & 0 & 1  & 2 & 3  & 1  & 2 & 3   \\
\hline
\hline
Shannon & 0.63662 &-0.21221 &0.12732 & $-9.0946\times10^{-2}$ & 0.63662 & 0 & 0.21221  \\
\hline
%Meyer $\epsilon=\frac{1}{20}$ & 0.6365 & -0.2119 & 0.1268 & -0.0902 & 0.6393 & -0.0086 & 0.2202  \\
%\hline
%Meyer  $\epsilon=1/10$ & 0.63622 & -0.21100 & 0.12532 & $-8.8166\times10^{-2}$  & 0.63260 & $-4.7224\times10^{-3}$ &  0.20054  \\
%\hline
%Meyer  $\epsilon=1/6$ & 0.63550  & -0.20887  & 0.12184 & $-8.3418\times10^{-2}$ & 0.62555 & $-1.2581\times10^{-2}$ &  0.18177\\
%\hline
Meyer  $\epsilon=1/3$ & 0.63216 & -0.19916 & 0.10668 & $-6.4166\times10^{-2}$ & 0.59378 & $-4.1412\times10^{-2}$ & 0.11930\\
\hline Splines order 3 & 0.62696  & -0.18538  &
8.8582$\times10^{-2}$  & -4.6179$\times10^{-2}$  & 0.55078 & -5.8322$\times10^{-2}$  & 8.2875$\times10^{-2}$ \\
\hline
Splines order 1 & 0.60142  & -0.12891  & $3.4815\times10^{-2}$  & $-9.2967\times10^{-3}$  & 0.38844  & $-5.7528\times10^{-2}$  & $1.8659\times10^{-2}$ \\
\hline
Haar & 0.51288 & $-1.1338\times10^{-2}$ & $-1.0855\times10^{-3}$ & $-2.6379\times10^{-4}$ & 0.10816 & $5.6994\times10^{-3}$ &  $1.5610\times10^{-3}$  \\
\hline
\end{tabular}
\caption{Theoretical cross-correlation values in the dyadic case
($d=0$).\label{tab:simulsdya}} }
%\end{center}
\end{table*}

\begin{table*}[h]
\begin{center}
{\scriptsize
\begin{tabular}{| c|| c| c| c| c|| c| c| c| }
\cline{2-8}
\multicolumn{1}{ c|}{}
 & \multicolumn{4}{|c||}{$\gamma_{\psi_0,\psi_0^\HH}$}  & \multicolumn{3}{|c|}{$\gamma_{\psi_1,\psi_1^\HH}$} \\
\hline
Wavelets\; $\setminus$ \;$\ell$ & 0 & 1  & 2 & 3  & 1  & 2 & 3   \\
\hline
\hline
Meyer $3$-band $\epsilon=1/4$ & 0.63411 & -0.20478 & 0.11530 & -7.4822$\times10^{-2}$ & 0.62662 & 0 & 0.18391 \\
\hline
%Meyer $3\sim4$-band $\epsilon=1/10$ & 0.63622 & -0.21100 & 0.12532 & -8.8166$\times10^{-2}$ & 0.63501 & 0 & 0.20742 \\
%\hline
Meyer $4$-band $\epsilon=1/5$ & 0.63501 & -0.20742 & 0.11950 & -8.0293$\times10^{-2}$  &  0.63020 & 0 & 0.19367  \\
\hline Meyer $5$-band $\epsilon=1/6$ & 0.63550 & -0.20887 & 0.12184 &  -8.3419$\times10^{-2}$ & 0.63216  & 0  & 0.19917  \\
\hline Meyer $6$-band $\epsilon=1/7$ & 0.63580 & -0.20975 & 0.12327 & -8.5357$\times10^{-2}$ & 0.63334 & 0 &  0.20255  \\
\hline Meyer $7$-band $\epsilon=1/8$ & 0.63599 & -0.21033 & 0.12421 & -8.6637$\times10^{-2}$ & 0.63411 &  0 &  0.20478 \\
\hline Meyer $8$-band $\epsilon=1/9$ &  0.63612 & -0.21072 & 0.12486 &  -8.7525$\times10^{-2}$ & 0.63463& 0 &  0.20632 \\
%\hline Meyer $12$-band $\epsilon=1/13$ & 0.63638 & -0.21149 & 0.12614 & -8.9291$\times10^{-2}$ & 0.63567 & 0 & 0.20936 \\
%\hline Meyer $16$-band $\epsilon=1/17$ & 0.63648 & -0.21170 & 0.12663 &  -8.9975$\times10^{-2}$ & 0.63606 & 0 & 0.21054 \\
\hline
%Meyer $4$-band $\epsilon=1/10$ & 0.63622 & -0.21100 & 0.12532 & -8.8166$\times10^{-2}$  &  0.63501 & 0 & 0.20742  \\
\end{tabular}
 \caption{Theoretical values for the first two cross-correlation sequences in the $M$-band Meyer case ($d=0$).\label{tab:simulsM1}} }
\end{center}
\end{table*}

\begin{table*}[h]
\begin{center}
{\footnotesize

\begin{tabular}{| c|| c| c| c| }
\cline{2-4}
\multicolumn{1}{ c|}{}
 & \multicolumn{3}{|c|}{$\gamma_{\psi_{M-1},\psi_{M-1}^\HH}$} \\
\hline
Wavelets\; $\setminus$ \;$\ell$  & 1  & 2 & 3 \\
\hline
\hline
%Shannon $4$-band   & 0.63662 & 0 & 0.21221 \\
%\hline
Meyer $3$-band $\epsilon=1/4$  & -0.58918 &  -6.0378$\times10^{-2}$ & -0.11965 \\
\hline
%Meyer $3$-band $\epsilon=1/10$  & -0.62864 &  -1.2302$\times10^{-2}$ & -0.19002 &  - & - & - \\
%\hline
Meyer $4$-band $\epsilon=1/5$ & 0.58555 & -7.0840$\times10^{-2}$ &  0.11961  \\
\hline
%Meyer $4$-band $\epsilon=1/10$ & -0.63501 & 0 & -0.20742  & 0.62316 & -2.2327$\times10^{-2}$ &  0.17705  \\
%\hline
Meyer $5$-band $\epsilon=1/6$ &  -0.58278 & -7.7359$\times10^{-2}$ & -0.11940 \\
\hline
Meyer $6$-band $\epsilon=1/7$ &  0.58063 & -8.1773$\times10^{-2}$ & 0.11914 \\
\hline
Meyer $7$-band $\epsilon=1/8$ &  -0.57893 & -8.4944$\times10^{-2}$ & -0.11888 \\
\hline
Meyer $8$-band $\epsilon=1/9$ &  0.57755 & -8.7324$\times10^{-2}$ & 0.11863 \\
\hline
%Meyer $12$-band $\epsilon=1/13$ & 0.57396 & -9.2855$\times10^{-2}$ & 0.11783 \\
%\hline
%Meyer $16$-band $\epsilon=1/17$ &  0.57195 & -9.5587$\times10^{-2}$ & 0.11730 \\
%\hline
\end{tabular}
\caption{Theoretical values for the last 
cross-correlation sequence in the
$M$-band Meyer case ($d=0$).\label{tab:simulsM2}}
 \vskip 0.3cm
\begin{tabular}{| c|| c| c| c| }
\hline
$\ell$ & 1  & 2 & 3  \\
\hline
\hline
$\gamma_{\psi_2,\psi_2^\HH}$ & $6.0560\times10^{-2}$ & $1.5848\times10^{-3}$ & $4.0782\times10^{-4}$ \\
\hline
$\gamma_{\psi_3,\psi_3^\HH}$ & $-4.9162\times10^{-2}$  &  $-3.0109\times10^{-4}$ & $-3.4205\times10^{-5}$  \\
\hline
$\gamma_{\psi_4,\psi_4^\HH}$ & 3.2069$\times 10^{-2}$ &  4.0952$\times 10^{-4}$  & 1.0319$\times10^{-4}$\\
\hline
$\gamma_{\psi_5,\psi_5^\HH}$ & -2.8899$\times 10^{-2}$ & -8.0753$\times 10^{-5}$ & -8.7950$\times10^{-6}$\\
\hline
$\gamma_{\psi_6,\psi_6^\HH}$ & -2.4899$\times 10^{-2}$ & -2.6077$\times 10^{-5}$ & -2.4511$\times 10^{-6}$ \\
\hline
$\gamma_{\psi_7,\psi_7^\HH}$ & 2.4297$\times 10^{-2}$ &  1.0608$\times 10^{-5}$ & 4.8118$\times10^{-7}$\\
\hline
%$\gamma_{\psi_8,\psi_8^\HH}$ & 1.6496$\times 10^{-2}$  & 1.0330$\times 10^{-4}$ & 2.5877$\times10^{-5}$ \\
%\hline
%$\gamma_{\psi_9,\psi_9^\HH}$ & -1.5677$\times 10^{-2}$ & -2.0585$\times 10^{-5}$  & -2.2149$\times10^{-6}$\\
%\hline
%$\gamma_{\psi_{10},\psi_{10}^\HH}$ & -1.4535$\times 10^{-2}$ & -6.6739$\times 10^{-6}$ & -6.1769$\times 10^{-7}$ \\
%\hline
%$\gamma_{\psi_{11},\psi_{11}^\HH}$ & 1.4373$\times 10^{-2}$ & 2.7471$\times 10^{-6}$ & 1.2193$\times10^{-7}$\\
%\hline
%$\gamma_{\psi_{12},\psi_{12}^\HH}$ & -1.2477$\times 10^{-2}$ & -1.8280$\times 10^{-6}$ & -1.5939$\times 10^{-7}$ \\
%\hline
%$\gamma_{\psi_{13},\psi_{13}^\HH}$ & 1.2425$\times 10^{-2}$ & 7.8561$\times 10^{-7}$  & 3.2144$\times10^{-8}$\\
%\hline
%$\gamma_{\psi_{14},\psi_{14}^\HH}$ & 1.2159$\times 10^{-2}$ & 3.0568$\times 10^{-7}$ & 9.7868$\times10^{-9}$\\
%\hline
%$\gamma_{\psi_{M-1},\psi_{M-1}^\HH}$ & -1.2138$\times 10^{-2}$ & -1.8744$\times 10^{-7}$ & -3.2532$\times10^{-9}$\\
%\hline
\end{tabular}
\caption {Theoretical cross-correlation values in the Walsh-Hadamard case.}\label{tab:simulsM3}}
\end{center}
\end{table*}

\begin{table*}[h]
%\begin{center}
\hspace*{-0.2cm}
{\scriptsize
\begin{tabular}{| c| c|| c| c| c| c|| c| c| c| c| }
\cline{3-10}
\multicolumn{2}{c|}{}
& \multicolumn{4}{|c||}{$\gamma_{\psi_0,\psi_0^\HH}$}  & \multicolumn{4}{|c|}{$\gamma_{\psi_1,\psi_1^\HH}$} \\
\hline
Wavelets & $j$\,$\setminus$\,$\ell$ & 0 & 1  & 2 & 3 & 0 & 1  & 2 & 3   \\
\hline
\hline 
& $ 1$ & 0.63538 & -0.21134 & 0.12586 & -9.1515$\times 10^{-2}$ & 9.97$\times 10^{-6}$ & 0.63680 & -1.7137$\times 10^{-4}$ & 0.21165 \\
\cline{2-10}
Shannon & $ 2$ & 0.63558 & -0.21347  &  0.12970 & -8.7908$\times 10^{-2}$ & 2.6426$\times 10^{-6}$& 0.63404 & 7.0561$\times 10^{-4}$ & 0.21210 \\
\cline{2-10}
 & $ 3$ & 0.63467 & -0.20732 & 0.13168 & -9.0116$\times 10^{-2}$ & -1.0078$\times 10^{-4}$ &  0.63846 & -1.2410$\times 10^{-3}$ &  0.20975\\
%\cline{2-10} & $ 4$ & 0.63821 & -0.20390 & 0.12547 &-8.8057$\times 10^{-2}$ & 1.5294$\times 10^{-4}$ &  0.63309 &3.8740$\times 10^{-3}$ & 0.20538\\
%\hline \hline & $ 1$ & 0.63492 & -0.21008 & 0.12380 & -8.8689$\times10^{-2}$ & 7.3183$\times 10^{-5}$ & 0.63433 & -1.7092$\times10^{-3}$ & 0.20456\\
%\cline{2-10}
%Meyer & $ 2$ & 0.63510 & -0.21220 & 0.12763 & -8.5061$\times10^{-2}$ & -4.5612$\times 10^{-5}$ & 0.62985 & -4.0012$\times 10^{-3}$ & 0.20031 \\
%\cline{2-10} $\epsilon=1/10$ & $ 3$ & 0.63437 & -0.20621 & 0.12975 & -8.7416$\times10^{-2}$ & -3.5753$\times 10^{-4}$ & 0.63444 & -6.2461$\times 10^{-3}$ & 0.19799\\
%\cline{2-10} & $ 4$ & 0.63735 & -0.20225 &  0.12303 & -8.4845$\times10^{-2}$ & 8.0967$\times 10^{-4}$ & 0.62882 & -4.7238$\times 10^{-4}$ & 0.19335 \\
%\hline \hline
%& $ 1$ & 0.63415 & -0.20790 & 0.12028 & -8.3897$\times10^{-2}$ & 1.2349$\times 10^{-4}$ & 0.63003 & -4.2080$\times 10^{-3}$ & 0.19303\\
%\cline{2-10}
%Meyer & $ 2$ & 0.63432 & -0.21001 & 0.12409 & -8.0273$\times10^{-2}$ & -7.2837$\times10^{-5}$ & 0.62263 & -1.1889$\times10^{-2}$ & 0.18143 \\
%\cline{2-10}
% $\epsilon=1/6$ & $ 3$ & 0.63355 & -0.20400 &  0.12618 & -8.2600$\times10^{-2}$ & -4.098$\times10^{-4}$ & 0.62718 & -1.4093$\times10^{-2}$ & 0.17914\\
%\cline{2-10}
%& $ 4$ & 0.63682 & -0.20033 & 0.11976 & -8.0314$\times10^{-2}$ & 4.3778$\times10^{-4}$ & 0.62178 & -8.6652$\times10^{-3}$ & 0.17475 \\
\hline 
\hline
& $ 1$ & 0.63091 & -0.19828 & 0.10517 & -6.4650$\times10^{-2}$ & 1.8257$\times10^{-5}$ & 0.61092 & -1.2433$\times10^{-2}$ &  0.15307 \\
\cline{2-10}
Meyer & $ 2$ & 0.63112  & -0.20043 & 0.10903 & -6.1060$\times10^{-2}$ & -7.5431$\times10^{-6}$ & 0.59115 & -4.0881$\times10^{-2}$ &  0.11888 \\
\cline{2-10}
$\epsilon=1/3$ & $ 3$ & 0.62971 & -0.19391 & 0.11084 & -6.3378$\times10^{-2}$ & 4.0868$\times10^{-4}$ & 0.59522 & -4.2624$\times10^{-2}$ & 0.11651 \\
%\cline{2-10} & $ 4$ & 0.63419 &  -0.19138  &  0.10535 & -6.1768$\times10^{-2}$ & -1.2745$\times10^{-3}$ & 0.59018 & -3.8016$\times10^{-2}$ & 0.11381\\
\hline \hline & $ 1$ & 0.62587 & -0.18459 & 8.7088$\times10^{-2}$
&
-4.6635$\times10^{-2}$ & -1.4511$\times10^{-4}$ & 0.58458 & -1.2651$\times10^{-2}$ & 0.12557 \\
\cline{2-10} Splines & $ 2$ & 0.62606 & -0.18679 & 9.1068$\times10^{-2}$ & -4.3124$\times10^{-2}$ & 1.9483$\times10^{-4}$ & 0,54841 & -5.8083$\times10^{-2}$ & 8.2386$\times10^{-2}$\\
\cline{2-10}
 order $3$ & $ 3$ & 0.62398 & -0.17984 & 9.2793$\times10^{-2}$ &  -4.5682$\times10^{-2}$& 1.2400$\times10^{-3}$ & 0.55204 & -5.9368$\times10^{-2}$ &  8.0105$\times10^{-2}$\\
%\cline{2-10} & $ 4$ & 0.62945 & -0.17798 & 8.7595$\times10^{-2}$ & -4.4227$\times10^{-2}$ & -2.8437$\times10^{-3}$ &  0.54744 & -5.4354$\times10^{-2}$ & 7.8836$\times10^{-2}$\\
\hline \hline & $ 1$ & 0.60016 & -0.12749  &  3.2975$\times10^{-2}$ & -9.7419$\times10^{-3}$ & -4.5287$\times10^{-4}$ & 0.47691 & 1.6224$\times10^{-2}$ & 6.9681$\times10^{-2}$ \\
\cline{2-10} Splines& $ 2$ &  0.60059 & -0.13045  & 3.7613$\times10^{-2}$ &  -6.5441$\times10^{-3}$ & 6.6358$\times10^{-4}$ & 0.38507 & -5.7502$\times10^{-2}$ & 1.8042$\times10^{-2}$ \\
\cline{2-10}
order $1$ & $ 3$ & 0.59771 & -0.12303  &  3.9388$\times10^{-2}$ &  -9.3208$\times10^{-3}$ & 2.2725 $\times10^{-3}$ &  0.38958 & -5.8143$\times10^{-2}$ & 1.6160$\times10^{-2}$ \\
%\cline{2-10} & $ 4$ & 0.60524 & -0.12249  & 3.4723$\times10^{-2}$ & -8.8249$\times10^{-2}$ & -5.4485$\times10^{-3}$ &  0.38649 & -5.3529$\times10^{-2}$ & 1.6641$\times10^{-2}$ \\
\hline \hline &
$ 1$ & 0.50297 & -3.3557$\times10^{-3}$ & -1.1706$\times10^{-3}$ & 2.7788$\times10^{-4}$ & 3.8368$\times10^{-4}$ &  0.22455 & 7.2451$\times10^{-2}$ & 4.6418$\times10^{-2}$\\
\cline{2-10}
Haar& $ 2$ & 0.50966 & -1.0083$\times10^{-2}$ & 7,2357$\times10^{-6}$ & 1.5087$\times10^{-3}$ & -1.2135$\times10^{-3}$ & 9.9745$\times10^{-2}$ & 5.1371$\times10^{-3}$ & 1.0847$\times10^{-3}$ \\
\cline{2-10}
& $ 3$ & 0.51023 & -8.3267$\times10^{-3}$ & 2.7936$\times10^{-3}$ & 7.0343$\times10^{-5}$ & 1.2329$\times10^{-3}$ &  0.10703 & 6.7651$\times10^{-3}$ & 2.2422$\times10^{-3}$ \\
%\cline{2-10} & $ 4$ & 0.51266 & -6.5010$\times10^{-3}$ & 1.4258$\times10^{-4}$ & -1.7978$\times10^{-3}$ & 8.2842$\times10^{-4}$ & 0.10734 & 7.0549$\times10^{-3}$ & 3.4539$\times10^{-4}$\\
\hline \hline  & $ 1$ & 0.59822 & -0.12059 &  2.3566$\times10^{-2}$  & -3.3325$\times10^{-3}$ & -5.0189$\times10^{-4}$  & 0.46392 & 2.1155$\times10^{-2}$ & 6.1137$\times10^{-2}$ \\
\cline{2-10} Symlets 8 & $ 2$ & 0.59899 & -0.12432 & 2.8865$\times10^{-2}$ & -2.8960$\times10^{-4}$ & 6.7795$\times10^{-4}$  & 0.36368  & -5.7692$\times10^{-2}$  & 9.7533$\times10^{-3}$ \\
\cline{2-10}
& $ 3$ & 0.59654 & -0.11703 & 3.0357$\times10^{-2}$ & -2.8071$\times10^{-3}$ & 1.8568$\times10^{-3}$ &  0.37012 & -5.8376$\times10^{-2}$ & 6.9416$\times10^{-3}$\\
 %\cline{2-10}
 %& $ 4$ & 0.60533 & -0.11777 &  2.6040$\times10^{-2}$ & -2.2219$\times10^{-3}$ & -6.1397$\times10^{-3}$ & 0.36796 & -5.4626$\times10^{-2}$ & 7.8162$\times10^{-3}$\\
  \hline
\end{tabular}
\caption{Cross-correlation estimates
in the dyadic case ($d=0$).\label{tab:numdya}} }
%\end{center}
\end{table*}
\begin{table*}[h]
%\begin{center}
\hspace*{-0.7cm}
{\scriptsize
\begin{tabular}{| c| c|| c| c| c| c|| c| c| c| c| }
\cline{3-10}
\multicolumn{2}{c|}{}
& \multicolumn{4}{|c||}{$\gamma_{\psi_0,\psi_0^\HH}$}  & \multicolumn{4}{|c|}{$\gamma_{\psi_1,\psi_1^\HH}$} \\
\hline
Wavelets& $j$\,$\setminus$\,$\ell$ & 0 & 1  & 2 & 3 & 0 & 1  & 2 & 3   \\
\hline
\hline Meyer & 1 & 0.63337 & -0.20549 & 0.11431 & -7.1877$\times10^{-2}$ & -6.8977$\times10^{-4}$ & 0.62533 & -1.3630$\times10^{-4}$ & 0.18236\\
\cline{2-10}  $3$-band & 2 & 0.63284 & -0.19932 & 0.11938 & -7.5331$\times10^{-2}$ & -1.7781$\times10^{-4}$ & 0.63013 & 1.2830$\times10^{-3}$ & 0.18409 \\
\cline{2-10} $\epsilon=1/4$ & 3 & 0.63886 & -0.19987 & 0.11763 & -6.6380$\times10^{-2}$ & -3.9622$\times10^{-4}$ & 0.61503 & 8.4042$\times10^{-4}$ & 0.17519\\
%\cline{2-10}  & 4 & 0.64835 & -0.19520 & 0.12189 & -6.5060$\times10^{-2}$ & -3.7264$\times10^{-4}$ & 0.62701 & 3.4949$\times10^{-3}$ & 0.17445\\
%\cline{2-10} & 4 & 0.62927 & -8.7242 & 1.9135$\times10^{-2}$ & 1.1516$\times10^{-3}$ & 1.2756$\times10^{-2}$ & 0.37426 & -3.7663$\times10^{-2}$ & 6.2396$\times10^{-3}$\\
 \hline \hline
Meyer  & 1 & 0.63383 & -0.20856 & 0.12176 & -7.7150$\times10^{-2}$
& 2.1961$\times10^{-5}$ & 0.62739 & 7.6636$\times10^{-4}$ &
0.19339\\ 
\cline{2-10}
  $4$-band & 2 & 0.63648 & -0.19903 & 0.11757 & -7.7337$\times10^{-2}$ & 4.8821$\times10^{-4}$ & 0.62676 & 3.8876$\times10^{-3}$ & 0.18683 \\
\cline{2-10} $\epsilon=1/5$ & 3 & 0.64642 & -0.19651 & 0.12202 & -6.9984$\times10^{-2}$ & 2.3054$\times10^{-3}$ & 0.63384 & -1.6254$\times10^{-3}$ & 0.19233\\
%\cline{2-10}& 4 & 0.66629 & -0.18141 & 0.12252 & -8.6081$\times10^{-2}$ & 7.1824$\times10^{-4}$ & 0.61528 & 2.4230$\times10^{-2}$ & 0.19744 \\
\hline \hline
 Meyer & 1 & 0.63338 & -0.20818 & 0.12534 & -8.0594$\times10^{-2}$ & 8.6373$\times10^{-4}$ & 0.62902 & 8.3871$\times10^{-4}$ & 0.1981\\ \cline{2-10}
  $5$-band &  2 & 0.64020 & -0.20288 & 0.12135 & -7.3844$\times10^{-2}$& 5.3607$\times10^{-4}$ & 0.62230 & 4.6651$\times10^{-4}$ & 0.19093 \\
\cline{2-10}  $\epsilon=1/6$ & 3 & 0.6566 & -0.19609 & 0.12891 & -7.6061$\times10^{-2}$ & -2.8654$\times10^{-3}$ & 0.62281 & -4.7324$\times10^{-3}$ & 0.19364 \\
%\cline{2-10} & 4 & 0.69661 & -0.21187 & 0.12674 & -4.7990$\times10^{-2}$ & -8.1119$\times10^{-4}$ & 0.60454 & -2.0467$\times10^{-2}$ & 0.15662 \\
\hline \hline Meyer & 1 & 0.63403 & -0.20818 & 0.12711 &
-8.2124$\times10^{-2}$ & 4.5293$\times10^{-4}$ & 0.63229 &
-1.9919$\times10^{-3}$ & 0.20228 \\ \cline{2-10}
   $6$-band & 2 &  0.64471 & -0.20716 & 0.13141 & -8.4914$\times10^{-2}$& 3.7150$\times10^{-4}$ & 0.62450 & 6.5942$\times10^{-4}$ & 0.20313\\
\cline{2-10} $\epsilon=1/7$ & 3 & 0.66409 & -0.19532 & 0.14401 & -9.3486$\times10^{-2}$ & 2.0490$\times10^{-3}$ & 0.63619 & 1.5614$\times10^{-2}$ & 0.17595 \\
%\cline{2-10}  & 4 & 0.70801 & -0.17813 & 0.13165 & -3.2517$\times10^{-2}$ & -1.3687$\times10^{-2}$ &  0.63063 & 6.4898$\times10^{-3}$ & 0.12762 \\
\hline  \hline Meyer & 1 & 0.63323 & -0.20781 & 0.12663 & -8.3335$\times10^{-2}$ & 1.5731$\times10^{-3}$ & 0.63528 & -8.6821$\times10^{-4}$ & 0.20509\\
 \cline{2-10}
$7$-band & 2 &  0.64286 & -0.20057 & 0.12881 & -8.1995$\times10^{-2}$ & -1.6505$\times10^{-4}$ & 0.62782 & -7.9119$\times10^{-3}$ & 0.20007 \\
\cline{2-10} $\epsilon=1/8$   & 3 & 0.68445 & -0.1845 & 0.12065 & -9.0295$\times10^{-2}$ & -5.9955$\times10^{-3}$ & 0.62572 & -5.3033$\times10^{-2}$ & 0.17409 \\
%\cline{2-10} & 4 & 0.74529 & -7.2234$\times10^{-2}$ & 2.4492$\times10^{-2}$ & -1.5505$\times10^{-2}$ & -4.0564$\times10^{-2}$ & 0.52884 & 8.3525$\times10^{-3}$ & 1.1663$\times10^{-2}$ \\
\hline  \hline Meyer & 1 & 0.63426 & -0.20592 & 0.12928 & -8.6766$\times10^{-2}$ & -2.1756$\times10^{-4}$ & 0,63658 & -1.3977$\times10^{-3}$ & 0.20385\\
 \cline{2-10}
$8$-band  & 2 & 0.64743 & -0.19970 & 0.12725 & -7.7096$\times10^{-2}$ & 1.4856$\times10^{-3}$ &  0.63725 & -2.4313$\times10^{-3}$ &  0.20396 \\
\cline{2-10} $\epsilon=1/9$ & 3 & 0.69342 & -0.20505 & 0.11257 & -6.0075$\times10^{-2}$ & -3.6363$\times10^{-3}$ & 0.61590 & 1.3830$\times10^{-2}$ & 0.22112 \\
%\cline{2-10}  & 4 & 0.72053 & -8.9478$\times10^{-2}$ & -6.8857$\times10^{-2}$ & 0.15407 & 2.7255$\times10^{-2}$ & 0.39475 & -3.7583$\times10^{-3}$ & -0.13973 \\
 \hline \hline
  & 1 & 0.59148 & -0.11001 & 1.9635$\times10^{-2}$ & 2.4318$\times10^{-3}$
  &-6.6559$\times10^{-6}$  &  0.36856 & -6.0858$\times10^{-2}$ &  8.4608$\times10^{-5}$ \\ 
\cline{2-10}
 AC & 2 & 0.59855 & -0.10412  & 1.6012$\times10^{-2}$ & 1.8921$\times10^{-4}$& -7.1462$\times10^{-3}$ & 0.37379 & -5.8026$\times10^{-2}$ &  -4.4309$\times10^{-3}$ \\
\cline{2-10} $4$-band & 3 & 0.60057 & -9.5335$\times10^{-2}$ & 2.0094$\times10^{-2}$  & 7.6430$\times10^{-3}$ & 2.5313$\times10^{-3}$ & 0.37514 & -5.6207$\times10^{-2}$ & 6.8164$\times10^{-3}$\\
\hline  \hline
 & & \multicolumn{4}{|c||}{$\gamma_{\psi_2,\psi_2^\HH}$}  & \multicolumn{4}{|c|}{$\gamma_{\psi_3,\psi_3^\HH}$} \\
 \hline
\hline
  & 1 & -1.9012$\times10^{-4}$ & -0.34054 & 5.5692$\times10^{-2}$ & 4.6899$\times10^{-5}$& -5.5011$\times10^{-5}$ & 0.36755 & 4.1274$\times10^{-2}$ & 5.6594$\times10^{-2}$ \\ \cline{2-10}
 AC & 2 &  1.0139$\times10^{-3}$ & -0.32275 & 5.4137$\times10^{-2}$ & -6.7903$\times10^{-3}$& 3.6460$\times10^{-3}$ & 0.18371 & -4.1645$\times10^{-2}$ & 6.8637$\times10^{-3}$ \\
\cline{2-10} $4$-band & 3 & 6.8587$\times10^{-3}$ & -0.32199 & 4.5083$\times10^{-2}$ & -9.7023$\times10^{-3}$ & 8.3037$\times10^{-3}$  & 0.19070 &  -3.7675$\times10^{-2}$ & -4.5919$\times10^{-4}$ \\
%\cline{2-10} & 4 & -1.9111$\times10^{-2}$  &  -0.32216 & 5.8864$\times10^{-2}$ & -1.5327$\times10^{-2}$ &-9.3368$\times10^{-3}$ & 0.17971 & -2.4932$\times10^{-2}$ & 1.3781$\times10^{-3}$ \\
\hline
\hline
   & 1 & 2.4712$\times 10^{-4}$ & 0.20479& 6.9476$\times 10^{-2}$ & 4.4200$\times 10^{-2}$ & -1.8669$\times 10^{-4}$&	-6.1810$\times 10^{-2}$ & -1.2677$\times 10^{-3}$&	2.4199$\times 10^{-5}$\\
\cline{2-10}
 Hadamard &
   2 & 3.5680$\times 10^{-3}$&	5.9530$\times 10^{-2}$& -5.3171$\times 10^{-3}$&4.3827$\times 10^{-3}$ & 6.2437$\times 10^{-4}$&-5.0635$\times 10^{-2}$&4.6773$\times 10^{-3}$&	-8.7358$\times 10^{-3}$ \\
\cline{2-10}
   & 3 & 1.1391$\times 10^{-2}$&5.9541$\times 10^{-2}$&	8.3376$\times 10^{-4}$&	-1.4604$\times 10^{-3}$ & 1.9009$\times 10^{-3}$&-5.5798$\times 10^{-2}$&	-5.7086$\times 10^{-3}$&-1.1253$\times 10^{-2}$ \\ 
\hline
\end{tabular}
\caption{Cross-correlation estimates
in the $M$-band case ($d=0$).\label{tab:numM1}}}
%\end{center}
\end{table*}
\begin{table*}[h]
\begin{center}
{\scriptsize
\begin{tabular}{| c| c|| c| c| c| c| }
\cline{3-6}
\multicolumn{2}{c|}{}
& \multicolumn{4}{|c|}{$\gamma_{\psi_{M-1},\psi_{M-1}^\HH}$} \\
\hline
Wavelet & $j$\,$\setminus$\,$\ell$ & 0 & 1  & 2 & 3 \\
\hline
\hline & 1 & 5.2467$\times10^{-6}$ & 0.63606 & 2.0952$\times10^{-3}$ & 0.21261 \\
\cline{2-6} Shannon  & 2 & 5.9145$\times10^{-6}$ & 0.63592 & -4.1893$\times10^{-3}$ & 0.21083 \\
\cline{2-6} $4$-band & 3 & -1.2667$\times10^{-4}$ & 0.62746 & -5.7616$\times10^{-3}$ & 0.2020 \\
%\cline{2-6} & 4 & -1.9758$\times10^{-3}$ & 0.62216 & 1.0411$\times10^{-2}$ & 0.19837 \\
\hline  \hline Meyer & 1 & 4.1334$\times10^{-4}$ & -0.60986 & -2.5395$\times10^{-2}$ & -0.16095 \\
\cline{2-6} $3$-band & 2 & 3.9059$\times10^{-4}$ & -0.58694 & -6.1089$\times10^{-2}$ & -0.11754 \\
\cline{2-6} $\epsilon=1/4$ & 3 & 3.5372$\times10^{-3}$ & -0.5879 & -5.1057$\times10^{-2}$ & -0.11499 \\
%\cline{2-6}  & 4 & -2.6603$\times10^{-3}$ & -0.57888 & -6.0512$\times10^{-2}$ & -0.11594 \\
 \hline \hline Meyer & 1 & 3.9730$\times10^{-4}$ & 0.60845 & -2.9936$\times10^{-2}$ & 0.16111 \\
\cline{2-6} $4$-band & 2 & -1.3788$\times10^{-3}$ & 0.58530 & -7.5797$\times10^{-2}$ &  0.11985 \\
\cline{2-6}  $\epsilon=1/5$ & 3 & 1.0644$\times10^{-3}$ &  0.57418 & -7.6790$\times10^{-2}$ & 0.10690 \\
%\cline{2-6}  & 4 & -6.6879$\times10^{-3}$ & 0.57189 & -6.4296$\times10^{-2}$ & 0.11240 \\
\hline  \hline Meyer & 1 & -7.2077$\times10^{-6}$  & -0.60862  &  -3.4588$\times10^{-2}$ &  -0.16162\\
\cline{2-6} $5$-band  & 2 & -3.2301$\times10^{-3}$ & -0.58482 & -8.6826$\times10^{-2}$ & -0.11844 \\
\cline{2-6} $\epsilon=1/6$ & 3 & -8.8877$\times10^{-3}$ & -0.56937 & -9.3811$\times10^{-2}$ & -0.11512 \\
%\cline{2-6}  & 4 & 4.0974$\times10^{-3}$ & -0.58715 & -7.9756$\times10^{-2}$ & -0.11369 \\
\hline \hline Meyer & 1 & 8.2632$\times10^{-4}$ & 0.60806 & -3.7209$\times10^{-2}$ & 0.16215\\
\cline{2-6} $6$-band  & 2 & -1.2448$\times10^{-3}$ & 0.58023 & -8.3257$\times10^{-2}$ &  0.11022 \\
\cline{2-6} $\epsilon=1/7$ & 3 & 5.5425$\times10^{-3}$ & 0.58196 & -8.4671$\times10^{-2}$ &  0.12368 \\
%\cline{2-6}  & 4 & -2.1938$\times10^{-2}$ & 0.59596 & -7.0229$\times10^{-2}$ & 4.9730$\times10^{-2}$ \\
\hline \hline Meyer & 1 & 2.7863$\times10^{-4}$ & -0.60863 & -3.9804$\times10^{-2}$ & -0.16443 \\
\cline{2-6} $7$-band & 2 & -5.9703$\times10^{-3}$ & -0.57749 & -9.9056$\times10^{-2}$ & -0.11228 \\
\cline{2-6}  $\epsilon=1/8$& 3 & 1.8490$\times10^{-3}$ & -0.58901 & -6.4289$\times10^{-2}$ & -0.13516 \\
%\cline{2-6}  & 4 & -5.2757$\times10^{-3}$ & -0.48108 & -3.2683$\times10^{-2}$ & -0,025974 \\
\hline \hline Meyer & 1 & -2.5084$\times10^{-4}$ & 0.60811  & -4.1611$\times10^{-2}$ & 0.16612 \\
\cline{2-6}  $8$-band & 2 & 1.0345$\times10^{-3}$ &  0.57216  & -9.4172$\times10^{-2}$ &  0.12014 \\
\cline{2-6} $\epsilon=1/9$ & 3 & -1.0777$\times10^{-2}$ &  0.56259  & -0.12183 & 0.10776 \\
%\cline{2-6}  & 4 & 1.3256$\times10^{-2}$  &  0.30836 & -5.5157$\times10^{-2}$ & -0.1311 \\
\hline
\end{tabular}
\caption{Estimation of the last 
cross-correlation sequence for
$M$-band Shannon and Meyer wavelets.\label{tab:numM2}} }
\end{center}
\end{table*}

\begin{table*}[h]
%\begin{center}
\hspace*{-0.8cm}
{\footnotesize
\begin{tabular}{| c| c| c| c| c| c| c| c| c| }
\hline
\multicolumn{1}{| r}{Wavelets}  &\multicolumn{1}{r |}{$\setminus$ \;$\ell$} & -3 & -2 & -1 & 0 & 1  & 2 & 3   \\
\hline
\hline
%Mey 2
%Test majorant OK
%ans =
Meyer 2-band & $\gamma_{\psi_0,\psi_1^\HH}(\ell)$ &   9.1502$\times10^{-2}$ &  -0.10848 &   0.11800 &  -0.11800  &  0.10848 &  -9.1502$\times10^{-2}$  &  7.0491$\times10^{-2}$\\
\cline{2-9}
$\epsilon = 1/3$ & $\gamma_{\psi_1,\psi_0^\HH}(\ell)$ &  -8.1258$\times10^{-2}$  &  0.10073 &  -0.11434  &  0.11924 &  -0.11434 &   0.10073 &  -8.1258$\times10^{-2}$\\
\hline
\hline
%Splines cub
%Test majorant OK
%ans =
Splines  & $\gamma_{\psi_0,\psi_1^\HH}(\ell)$ &  -8.2660$\times10^{-2}$  &  0.13666 &  -0.18237  &  0.18237 &  -0.13666  &  8.2660$\times10^{-2}$  & -4.5433$\times10^{-2}$\\
\cline{2-9}
order $3$  & $\gamma_{\psi_1,\psi_0^\HH}(\ell)$ &  6.1604$\times10^{-2}$ &  -0.10838  &  0.16319 &  -0.18941 &   0.16319 &  -0.10838  &  6.1604$\times10^{-2}$\\
\hline
\hline
%Splines Frank
%Test majorant OK
%ans =
%Splines $N=1$ &  -3.3440$\times10^{-2}$ &   0.13632 &  -0.30451 &   0.30451 &  -0.13632 &  3.3440$\times10^{-2}$ &  -9.7915$\times10^{-3}$\\
%   & 1.7422$\times10^{-2}$ &  -6.9831$\times10^{-2}$  &  0.22495 &  -0.33682 &   0.22495 &  -6.9831$\times10^{-2}$  &  1.7422$\times10^{-2}$\\
%\hline
%Haar
%Test majorant OK
%ans =
Haar    & $\gamma_{\psi_0,\psi_1^\HH}(\ell)$ & -9.2323$\times10^{-3}$ &  -2.2034$\times10^{-2}$  & -0.16656  &  0.44127 &  -0.16656 &  -2.2034$\times10^{-2}$ &  -9.2323$\times10^{-3}$\\
\cline{2-9}
  & $\gamma_{\psi_1,\psi_0^\HH}(\ell)$ & -3.1567$\times10^{-3}$  & -1.9621$\times10^{-2}$  &  0.35401 &  -0.35401  &  1.9621$\times10^{-2}$ &  3.1567$\times10^{-3}$  &  1.0758$\times10^{-3}$\\
\hline
\hline
%Mey 3
%1
%2 
Meyer  & $\gamma_{\psi_0,\psi_1^\HH}(\ell)$ & -8.4807$\times10^{-2}$ & 8.8904$\times10^{-2}$ & -8.8904$\times10^{-2}$ & 8.4807$\times10^{-2}$ & -7.7120$\times10^{-2}$ & 6.6763$\times10^{-2}$ & -5.4904$\times10^{-2}$\\
\cline{2-9}	
$3$-band & $\gamma_{\psi_1,\psi_0^\HH}(\ell)$ & 6.0944$\times10^{-2}$ & -7.2206$\times10^{-2}$ & 8.1363$\times10^{-2}$ & -8.7347$\times10^{-2}$ & 8.9428$\times10^{-2}$ & -8.7347$\times10^{-2}$ & 8.1363$\times10^{-2}$\\
\cline{2-9}
%2
%3
$\epsilon=1/4$ &$\gamma_{\psi_1,\psi_2^\HH}(\ell)$ & -6.3891$\times10^{-2}$& -7.4738$\times10^{-2}$ & -8.3192$\times10^{-2}$ & -8.8252$\times10^{-2}$ & -8.9297$\times10^{-2}$ & -8.6196$\times10^{-2}$ & -7.9333$\times10^{-2}$ \\
\hline
\hline
%Mey 4
%1
%2
%ans =
Meyer & $\gamma_{\psi_0,\psi_1^\HH}(\ell)$ & 6.5090$\times10^{-2}$ & -6.9156$\times10^{-2}$ & 7.1274$\times10^{-2}$ & -7.1274$\times10^{-2}$ & 6.9156$\times10^{-2}$	& -6.5090$\times10^{-2}$ &5.9394$\times10^{-2}$\\
\cline{2-9}	
$4$-band & $\gamma_{\psi_1,\psi_0^\HH}(\ell)$ & -6.2421$\times10^{-2}$ & 6.7350$\times10^{-2}$ & -7.0473$\times10^{-2}$ & 7.1543$\times10^{-2}$ & -7.0473$\times10^{-2}$ & 6.7350$\times10^{-2}$ & -6.2421$\times10^{-2}$	\\
\cline{2-9}
%2
%3
%ans =
$\epsilon=1/5$ & $\gamma_{\psi_1,\psi_2^\HH}(\ell)$ & 6.0949$\times10^{-2}$&6.6274$\times10^{-2}$&	6.9878$\times10^{-2}$&	7.1475$\times10^{-2}$&	7.0939$\times10^{-2}$&	6.8312$\times10^{-2}$&	6.3804$\times10^{-2}$\\
\cline{2-9}
%3
%4
%ans =
& $\gamma_{\psi_2,\psi_3^\HH}(\ell)$ & -6.5090$\times10^{-2}$&6.9156$\times10^{-2}$&	-7.1274$\times10^{-2}$&	7.1274$\times10^{-2}$&	-6.9156$\times10^{-2}$&	6.5090$\times10^{-2}$&	-5.9394$\times10^{-2}$\\	
\hline
%1 2
%Hadamard & -9.2323$\times10^{-3}$ & -2.2034$\times10^{-2}$ & -0.16656 & 0.44127 & -0.16656 & -2.2034$\times10^{-2}$ & -9.2323$\times10^{-3}$\\	
%$M=4$ & -3.1567$\times10^{-3}$ & -1.9621$\times10^{-2}$ & 0.35401 & -0.35401 & 1.9621$\times10^{-2}$ & 3.1567$\times10^{-3}$ & 1.0758$\times10^{-3}$\\
%1 3
%\cline{2-8}
%& -6.6339$\times10^{-3}$ & -1.6017$\times10^{-2}$ & -0.15546 & 0.38850 & -0.15546 & -1.6017$\times10^{-2}$ & -6.6339$\times10^{-3}$\\
%& -4.1816$\times10^{-3}$ & -2.6186$\times10^{-2}$ & 0.31505 & -0.31505 & 2.6186$\times10^{-2}$ & 4.1816$\times10^{-3}$ & 1.4226$\times10^{-3}$\\
%1 4
%\cline{2-8}
%& 1.1607$\times10^{-3}$ & 4.2425$\times10^{-3}$ & 0.11124 & 1.1589$\times10^{-18}$ & -0.11124 & -4.2425$\times10^{-3}$ & -1.1607$\times10^{-3}$\\
%& -1.4907$\times10^{-2}$ & -5.3095$\times10^{-2}$ & 9.6254$\times10^{-2}$ & 9.6254$\times10^{-2}$ & -5.3095$\times10^{-2}$ & -1.4907$\times10^{-2}$ & -7.1950$\times10^{-3}$\\
%2 3
%\cline{2-8}
%& -1.1607$\times10^{-3}$ & -4.2425$\times10^{-3}$ & -0.11124 & 2.9669$\times10^{-18}$ & 0.11124 & 4.2425$\times10^{-3}$ & 1.1607$\times10^{-3}$\\
%2 4
%\cline{2-8}
%& 4.6490$\times10^{-4}$ & 1.6069$\times10^{-3}$ & 9.3135$\times10^{-2}$ & 0.85959 & 9.3135$\times10^{-2}$ & 1.6069$\times10^{-3}$ & 4.6490$\times10^{-4}$\\
%3 4
%\cline{2-8}
%& 3.3095$\times10^{-4}$	 & 1.2354$\times10^{-3}$ & 0.10659& -0.21631 & 0.10659 & 1.2354$\times10^{-3}$ & 3.3095$\times10^{-4}$\\
%\hline
\end{tabular}
\caption{Inter-band cross-correlation values for some wavelet families. We recall that Property \eqref{eq:symmmplmc} holds and
that, for $M$-band Meyer wavelets $\gamma_{\psi_m,\psi_{m'}^\HH}$ is zero when $|m-m'|> 1$.
%For dyadic and Meyer $M$-band wavelets ($M=\{3,4\}$), first line $\gamma_{\psi_0,\psi_1^\HH}$  and second line $\gamma_{\psi_1,\psi_0^\HH}$. In addition, for $M$-band case, $\gamma_{\psi_m,\psi_{m'}^\HH}$ with $m'\neq0$, $m'>m$ and $m'-m=1$, $m \in \{1,...,M-2\}$. 
%For Hadamard $4$-bands wavelets, the three firt double lines, we display $\gamma_{\psi_0,\psi_{m'}^\HH}$ with $m'\neq0 \,,\,m'\in\{1,M-1\}$  and $\gamma_{\psi_m,\psi_0^\HH}$ with $m\neq0 \,,\,m\in\{1,M-1\}$ respectively. We then show,  $\gamma_{\psi_m,\psi_{m'}^\HH}$ with $m'\neq0$, $m'>m$, $m \in \{1,M-2\}$. 
\label{tab:simulsinterb}}}
%\end{center}
\end{table*}

\end{document}